\documentclass[notitlepage,reqno,11pt]{amsart}
\usepackage{latexsym,amssymb, epsfig, amsmath,amsfonts, subfigure,amsthm}

\usepackage{rotating}
\usepackage[toc,page]{appendix}
\usepackage{color}
\usepackage{multirow}
\usepackage{relsize}
\usepackage{microtype}
\usepackage{dsfont}
\usepackage{cases}

\usepackage[colorlinks, allcolors=blue]{hyperref}

\usepackage[margin=1in]{geometry}

\numberwithin{equation}{section}
\newtheorem{assumption}{Assumption}[section]
\newtheorem{lemma}[assumption]{Lemma}
\newtheorem{theorem}[assumption]{Theorem}
\newtheorem{definition}[assumption]{Definition}
\newtheorem{coro}[assumption]{Corollary}
\newtheorem{prop}[assumption]{Proposition}
\newtheorem{remark}[assumption]{Remark}

\newlength{\defbaselineskip}
\newcommand{\setlinespacing}[1]%
{\setlength{\baselineskip}{#1 \defbaselineskip}}

\newcommand{\RR}{{\mathbb R}}

\def\E{\mathbb{E}}
\def\P{\mathbb{P}}

\newcommand{\beql}[1]{\begin{equation}\label{#1}}
\newcommand{\eeq}{\end{equation}}

\newcommand{\beqal}[1]{\begin{eqnarray}\label{#1}}
\newcommand{\eeqa}{\end{eqnarray}}
\newcommand{\beq}{\begin{displaymath}}
\newcommand{\eeqno}{\end{displaymath}}
\newcommand{\bali}[1]{\begin{align*}\label{#1}}
\newcommand{\eali}{\begin{align*}}
\newcommand{\balino}{\begin{align*}}
\newcommand{\ealino}{\begin{align*}}

\newcommand{\ep}{\epsilon}

\newcommand{\Cov}{\text{\rm Cov}}

\newcommand{\sign}{\text{\rm sign}}

\newcommand{\R}  {\mathbb{R}}
\newcommand{\N}  {\mathbb{N}}

\newcommand{\bD}{{\mathbf D}}

\newcommand{\bC}{{\mathbf C}}

\newcommand{\bone}{{\mathbf 1}}

\newcommand{\non}{\nonumber}

\newcommand{\baa}{\begin{eqnarray*}}
\newcommand{\eaa}{\end{eqnarray*}}

\newcommand{\ttl}{\Large Functional central limit theorems for epidemic models
\\[5pt]
with varying infectivity and waning immunity}

\newcommand{\ttls}{\large FCLT for epidemic models with varying infectivity and waning immunity }

\begin{document}

\title[\ttls]{\ttl}

%
%
%

\author[Arsene--Brice Zotsa--Ngoufack]{Arsene--Brice Zotsa--Ngoufack}
\address{Aix Marseille Univ, CNRS, I2M, Marseille, France}
\email{arsene-brice.zotsa-ngoufack@univ-amu.fr}

\date{\today}

\begin{abstract} 
We study an individual-based stochastic epidemic model in which infected individuals become susceptible again following each infection (generalized SIS model). Specifically, after each infection, the infectivity is a random function of the time elapsed since the infection, and  each recovered individual loses immunity gradually  (equivalently, becomes gradually susceptible) after some time according to a random susceptibility function. 
The epidemic dynamics is described by the average infectivity and susceptibility processes in the population together with the numbers of infected and susceptible/uninfected individuals. 
In \cite{forien-Zotsa2022stochastic}, a functional law of large numbers (FLLN) is proved as the population size goes to infinity, and asymptotic endemic behaviors are also studied. 
In this paper, we prove a functional central limit theorem (FCLT) for the stochastic fluctuations of the epidemic dynamics around the FLLN limit.  The FCLT limit for the aggregate infectivity and susceptibility processes  is given by a system of stochastic non-linear integral equation driven by a two-dimensional Gaussian process.  



\end{abstract}


\maketitle
\allowdisplaybreaks

\section{Introduction}

Many infectious diseases become endemic over a long time horizon, for which waning of immunity plays a critical role in addition to the infection process. 
The classical compartment model, SIRS (susceptible-infectious-recovered-susceptible), assumes that immunity at the individual level is binary, that is, each individual is either fully immune or fully susceptible.  However, that is largely unrealistic since it does not allow for partial immunity or gradual waning of immunity. 
Various models have been developed to study the effects of the waning of immunity and the associated vaccination policies \cite{hethcote1976,hethcote1999,thieme2002endemic,barbarossa2015mathematical,ehrhardt2019,carlsson2020,strube2021,childs2022,safan2022,khalifi2022extending,foutel2023optimal,elgart2023}. 
All the models except \cite{carlsson2020,foutel2023optimal} start from an ODE model with the additional waning immunity characteristic. In particular, El Khalifi and Britton \cite{khalifi2022extending} recently studied an extension of the ODE for the classical SIRS model with a linear or exponential waning function. They started with an approximations using a fixed number of immunity levels and then discussed the corresponding ODE-PDE limiting model (similar to \cite{thieme2002endemic}) associated with the age of immunity as the number of immunity level goes to infinity.  See also \cite{elgart2023} for a perturbation analysis of a model with an arbitrarily large number of discrete compartments with varying levels of disease immunity. Carlsson et al. \cite{carlsson2020} study an age-structured PDE model that takes into account waning immunity. Despite the interesting findings, there has been lack of individual-based stochastic epidemic models that take into account waning immunity. 

In \cite{forien-Zotsa2022stochastic} the authors Forien, Pang, Pardoux and Zotsa first introduced an individual-based stochastic epidemic model that captures waning immunity as well as varying infectivity \cite{FPP2020b}. More precisely, they proposed a general stochastic epidemic model which takes into account  a random infectivity and a random and gradual loss of immunity (also referred to as waning immunity or varying susceptibility).  See Figure \ref{fig1} for a realization of the infectivity and susceptibility of an individual after an infection. 
Individuals experience the susceptible-infected-immune-susceptible cycle. When an individual becomes infected, the infected period may include a latent exposed period and then an infectious period. Once an individual recovers from the infection, after some potential immune period (whose duration can be zero), the immunity is gradually lost, and the individual progressively becomes susceptible again.
Then the individual may be infected again, and repeat the process at each new infection with a different realization of the random infectivity and susceptibility functions. This model can be regarded as a generalized SIS model. 
When ``I" is interpreted as ``infected" including exposed and infectious periods, and ``S" is interpreted as including immune and susceptible periods.
It can of course also be regarded as a generalized SEIRS model. We also mention the recent work  \cite{foutel2023optimal}, where a similar stochastic model of varying infectivity and waning immunity with vaccination is studied, where the focus is on the effect of vaccination policies to prevent endemicity. 
We notice the difference from our modeling approach besides the vaccination aspect: the random susceptibility function and varying infectivity function are taken independently  in each infection. However, we do not impose the independence between the random infectivity and susceptibility functions in each infection. 

\begin{figure}[tbp] \label{fig1}
\includegraphics[width=0.8\textwidth]{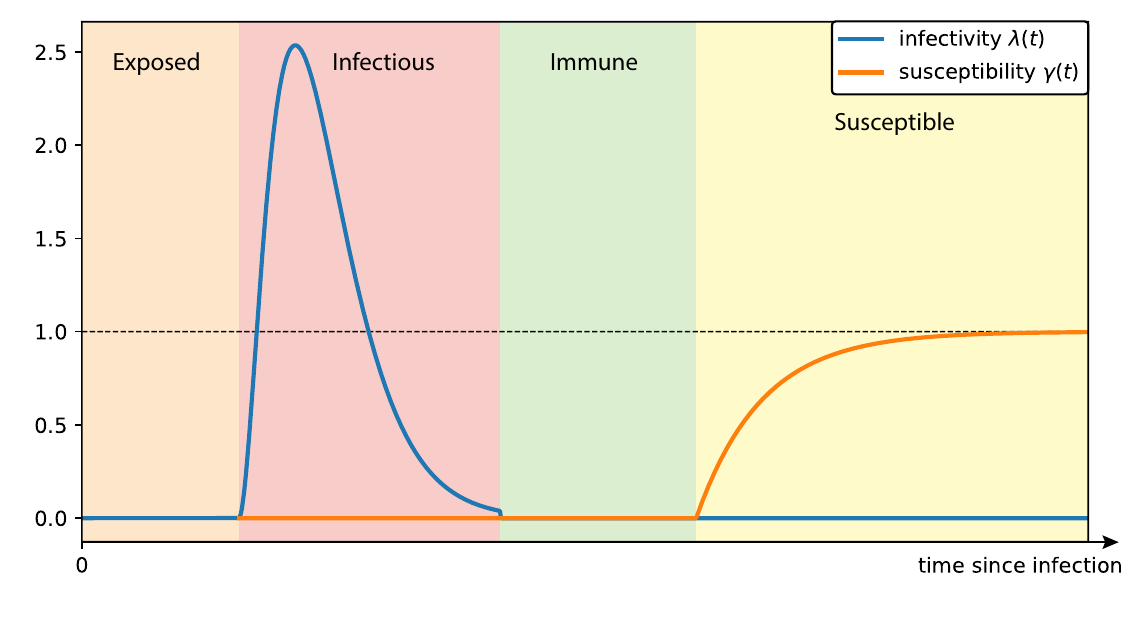}
\caption{Illustration of a typical realization of the random infectivity and susceptibility functions of an individual from the time of infection to the time of recovery, and then to the time of losing immunity and becoming fully susceptible (or in general, partially susceptible).}
\end{figure}

In \cite{forien-Zotsa2022stochastic}, the authors have proved a functional law of large numbers (FLLN)  in which both the average susceptibility and the force of infection converge, when the size of the population goes to infinity, to a deterministic limiting model given by a system of integral equations depending on the law of the susceptibility and on the mean of the infectivity function  (see Theorem~\ref{thm-FLLN} below). 
Under a particular set of random infectivity and susceptibility functions and initial conditions, they also show that a PDE model with infection-age can be derived from the limiting model, which reduces to the model introduced by Kermack and McKendrick in \cite{kermack_contributions_1932,kermack_contributions_1933} (see the reformulation in \cite{inaba2001kermack}). 
They also characterize the threshold of endemicity which depends on the law of susceptibility and not only on the mean, and prove the global asymptotical stability of the disease-free steady state when the basic reproduction number is lower than the above-mentioned threshold. When the basic reproduction number is larger than this threshold, they authors prove existence and uniqueness of the endemic equilibrium and under additional assumptions, they prove that the disease-free-steady state is unstable. 



The goal of this work is to  study the stochastic fluctuations of  the dynamics around the deterministic limits for the stochastic epidemic models with random varying infectivity and a random and gradual loss of immunity, see the result in Theorem~\ref{TCL}. More precisely we study jointly the fluctuation of the average total force of infection and average of susceptibility and then deduce the fluctuations of the proportions of the compartment counting processes. 
The fluctuation limit of the average total force of infection and average susceptibility is given by a system of stochastic non-linear integral equation driven by a two-dimensional Gaussian process. Given these, the limits of the compartment counting processes are expressed in terms of the solutions of the above non-linear stochastic integral equation driven by another two-dimensional Gaussian process.    
This result extend the functional central limit theorem (FCLT) results of Pang-Pardoux in \cite{pang2022-CLT-functional,PandPardoux-2020}  for the non-Markovian models without gradual loss of immunity and in \cite{britton2018stochastic,kurtz_limit_1971} for the Markovian case. 

To prove the FCLT (Theorem~\ref{TCL}), we first obtain a decomposition for the scaled infectivity and susceptibility processes, each of which has two component processes.
 We employ the central limit theorems for $\bD$-valued random variables \cite{hahn1978central}  to prove the convergence of one component since it can be regarded as a sum of i.i.d. $\bD$-valued random variables. The convergence of the other component is much more challenging, and we must develop novel methods to prove tightness and convergence.
We need more assumptions on the pair of random function $(\lambda,\gamma)$ than the ones used to establish the FLLN, and 
these assumptions are crucial to establish tightness. 
For that purpose, we need to establish moment estimates and maximal inequalities for the increments of the processes. 
This is extremely difficult because of the complicated interactions among the individuals, as well as the randomness in the infectivity and susceptibility. 
Some of the expressions involve stochastic integrals with respect to Poisson random measures, with integrands which are not predictable but depend on the future. 
The classical result for moment calculations of stochastic integrals cannot be used in our setting, for example,  \cite[Theorem~6.2]{ccinlar2011probability}. 
Thus, we establish a new theorem to calculate the moments for such stochastic integrals (see Theorem~\ref{TCL-th-1}). 

In addition, 
we develop an approximation technique by introducing a quarantine model, in which one infected individual is quarantined so that the number of infected descendants of that individual can be bounded conveniently (this scheme can be extended to more than one quarantined individual). Using this approximation, we compare the processes counting the number of infections of each individual for the original process to the number of infections of each individual for the quarantine model, and as a consequence, we obtain the moment estimates and maximal inequalities to prove tightness (see Lemmas~\ref{TCL-A-9-lem_inq} and \ref{TCL-A-9-lem-inq2}). 


Finally it is worth noting that our individual-based stochastic model resembles the recent studies of models with interactions, for instance, interacting age-dependent Hawkes process in \cite{chevallier2017mean,chevallier_fluctuations_2017}, age-structured population model in \cite{tran2006modeles} and stochastic excitable membrane models in \cite{riedler2012limit}. 
In the proof of the FLLN in \cite{forien-Zotsa2022stochastic}, the authors adapted the tools to the theory of  propagation of chaos (see Sznitman \cite{sznitman1991topics}) by constructing a family of i.i.d. processes with a well-chosen coupling. A similar approach was taken in  \cite{chevallier2017mean,tran2006modeles}. However, for the FCLT, we derive from the approach of studying fluctuations from the mean limit that were taken in \cite{chevallier2017mean,tran2006modeles}, since it is more challenging for our non-Markovian model. In that approach one has to work with processes taking values in a Hilbert space (dual of some Sobolev space of test functions) and the limit is characterized by an SDE in infinite dimension driven by a Gaussian noise. On the contrast, we work directly with the real-valued processes and prove their convergence with the conventional tightness criteria, which leads to a finite-dimensional stochastic integral equation driven by Gaussian processes. 


\subsection*{Organization of the paper} The rest of the paper is organized as follows. In Section \ref{TCL-sec-1}, we describe the model and recall the FLLN results from \cite{forien-Zotsa2022stochastic}. Next, we state the assumptions and the FCLT result.
The proof for the FCLT is presented in Section \ref{TCL-sec-2}. 
In Section \ref{TCL-US-R},  we present some preliminary results that will be used in the proofs. 
 In Section~\ref{TCL-sec-G10} we characterize the limit of the convergent subsequences. 
 In Section~\ref{TCL-sec-L1} we approximate the limit,  and finally we prove tightness in Section~\ref{TCL-sec-Th}.

\subsection*{Notation}
Throughout the paper, all the random variables and processes are defined on a common complete probability space $(\Omega, \mathcal{F},\P)$.  
We use $\xrightarrow[N\to+\infty]{\mathbb{P}}$ to denote convergence in probability as the parameter $N\to \infty$.
Let $\N$ denote the set of natural numbers and $\R^k (\R^k_+)$ the space of $k$-dimensional vectors with real (nonnegative) coordinates, with $\R(\R_+)$ for $k=1$.  We use $\mathds{1}{\{\cdot\}}$ for the indicator function. Let $\bD=\bD(\R_+; \R)$ be the space of $\R$-valued c{\`a}dl{\`a}g functions defined on $\R_+$, with convergence in $\bD$ meaning convergence in the Skorohod $J_1$ topology (see, e.g., \cite[Chapter 3]{billingsley1999convergence}).  Also, we use $\bD^k$ to denote the $k$-fold product with the product $J_1$ topology. Let $\bC$ be the subset of $\bD$ consisting of continuous functions and $\bD_+$ the subset of $\bD$ of c{\`a}dl{\`a}g functions with values in $\R_+$. We use $\Rightarrow$ to denote the weak convergence in $\bD$.

\bigskip

\section{Model and Results}\label{TCL-sec-1}
\subsection{Model description}\label{TCL-sec1}
We start with a population with a fixed finite size $N$, and enumerate the individuals of the population with the parameter $k,\,1\leq k\leq N$.   

Let $(\lambda_{k,i},\gamma_{k,i})_{k\geq1,i\geq1}$ be a collection of i.i.d. random functions and also, $(\lambda_{k,0},\gamma_{k,0})_{k\geq1}$ be a collection of i.i.d. random functions taking values in the same space, independent from the previous one.
Let $(Q_k)_{k\geq1}$ be a family of independent standard Poisson random measures on $\RR_+^2$, independent from the two previously defined families. $\lambda_{k,i}$ represents the infectivity of the $k$-th individual after its $i$-th infection and $\gamma_{k,i}$ represents the susceptibility of the $k$-th individual after its $i$-th infection. Similarly, $\lambda_{k,0}\;\; (\text{resp. }\gamma_{k,0})$ represents the infectivity (resp. susceptibility) of the $k$-th individual in the beginning of the epidemic. 


We assume that each infected individual has infectious contacts at a rate equal to its current infectivity. At each infectious contact, an individual is chosen uniformly in the population and this individual becomes infected with probability given by its susceptibility. Thus if we let $A^N_k(t)$ be the number of times that the $k$-th individual has been infected between time $0$ and $t$, then the infectivity of the $k$-th individual at time $t$ is given by $\lambda_{k,A^N_k(t)}(\varsigma^{N}_k(t))$ and its susceptibility is $\gamma_{k,A^N_k(t)}(\varsigma^{N}_k(t))$ where 
\begin{equation}\label{TCL-t-1}
\varsigma^N_k(t) := t - \left(\sup\lbrace s \in [0,t] : A^N_k(s) = A^N_k(s^-) + 1 \rbrace \vee 0\right)
\end{equation}
is the time elapsed since the last time when it was infected or since the start of the epidemic if it has not been infected yet (we use the convention
$\sup\emptyset=-\infty$).

Hence, let $\{A^N_k(t),\,t\ge0,\,1\leq k\leq N\}$ be the solution of
\begin{equation*}
A^N_k(t)=\int_{[0,t]\times\R_+}\mathds{1}_{u\leq\Upsilon^N_k(r^-)}Q_k(dr,du)\\
\end{equation*} 
where \[\Upsilon^N_k(t)=\gamma_{k,A_k^N(t)}(\varsigma^N_k(t))\overline{\mathfrak{F}}^N(t)\] is the instantaneous infectivity rate function at time $t$ with
\begin{equation}
\overline{\mathfrak{F}}^N(t)=\frac{1}{N}\sum_{k=1}^{N}\lambda_{k,A^N_k(t)}(\varsigma^{N}_k(t))\,.
\label{def-I}
\end{equation}
The total force of infection $\mathfrak{F}^N(t)$ at time $t$ is the sum of the infectivities of all the infected individuals at time $t$.

We also define the average susceptibility of the population by 
\begin{equation}\overline{\mathfrak{S}}^N(t)=\frac{1}{N}\sum_{k=1}^N\gamma_{k,A_k^N(t)}(\varsigma^N_k(t)).\label{def-G}\end{equation}
Let  \[\Upsilon^N(t)=\sum_{k=1}^N\Upsilon^N_k(t)=N\overline{\mathfrak{S}}^N(t)\overline{\mathfrak{F}}^N(t)\]
be the total instantaneous infection rate function  in the population at time $t$. 

 Define 
\[\eta_{k,i}=\sup\{t>0,\,\lambda_{k,i}(t)>0\}\]
for each $i\in\mathbb{N}_0$  and $k=1,\dots,N$, representing 
the duration of the $i$-th infection of the $k$-th individual.  By the i.i.d. assumption on $(\lambda_{k,i})_{k \ge 1, i\ge 1}$, the variables 
$(\eta_{k,i})_{k \ge 1, i \ge 1}$ are i.i.d., similarly for $(\eta_{k,0})_{k \ge 1}$. Also, the two families of random variables are independent. 
We denote their cumulative distribution functions by
\begin{equation*}F_0(t)=\mathbb{P}\left(\eta_{1,0} \le t\right),\quad\quad F(t)=\mathbb{P}\left(\eta_{1,1}\le t\right), \quad t \ge 0. \end{equation*}
Let $F_0^c=1-F_0(t)$ and $F^c(t) = 1-F(t)$ for $t\ge 0$.

We define the number of infectious individuals at time $t$ by 
\begin{equation}\label{n_inf}
I^N(t)=\sum_{k=1}^{N}\mathds{1}_{\varsigma_k^N(t)<\eta_{k,A^N_k(t)}} \,\,,
\end{equation}
and the number of uninfected individuals at time $t$ by
\begin{equation}
U^N(t)=\sum_{k=1}^{N}\mathds{1}_{\varsigma_k^N(t)\geq\eta_{k,A^N_k(t)}}=N-I^N(t)\,. \label{n_sus}
\end{equation}

	\subsection{Already known results}
	From \cite[Lemma~$6.1$]{forien-Zotsa2022stochastic}, there exists a unique $\overline{\mathfrak F}\in \bD(\R_+)$ such that 
	\begin{equation*}
	\overline{\mathfrak F}(t)=\E\left[\lambda_{1,A_1(t)}(\varsigma_{1}(t))\right]\text{ and }	
	\overline{\mathfrak S}(t)=\E\left[\gamma_{1,A_1(t)}(\varsigma_{1}(t))\right],
	\end{equation*} 
	where for $k\geq1$ the process $A_k$ is defined as:
	\[A_k(t)=\int_{[0,t]\times\R_+}\mathds{1}_{u\leq\Upsilon_k(r^-)}Q_k(dr,du),\] with
	\[\Upsilon_k(t)=\gamma_{k,A_k(t)}(\varsigma_{k}(t))\overline{\mathfrak{F}}(t),\] 
	and
	$\varsigma_k$ is defined in the same manner as $\varsigma^N_1$ with $ A_k$ instead of $ A^N_1$, see \eqref{TCL-t-1}.
	 In this definition  we use the same $(\lambda_{k,i},\gamma_{k,i},Q_k)$ as in the definition of the model in subsection~\ref{TCL-sec1}. 
Moreover, note that, as the $\left((\lambda_{k,i})_i,(\gamma_{k,i})_i,Q_k\right)_{k\geq1}$ are i.i.d, the $(A_k)_k$ are also i.i.d.

We make the following assumption. 
\begin{assumption}\label{TCL-AS-lambda-0}
	There exists a deterministic constant $ \lambda_* < \infty $ such that $ 0\leq\lambda_{k,i}(t) \leq \lambda_* $ almost surely, and $ 0 \leq \gamma_{k,i}(t) \leq 1 $ for all $ t \geq 0 $, for all $ i \geq 0 $ and $ 1 \leq k \leq N $, almost surely. 
	Moreover,
	\begin{equation}\label{eqq10}
	\sup\{t\geq0,\,\lambda_{k,i}(t)>0\}\leq\inf\{t\geq0,\,\gamma_{k,i}(t)>0\},
	\end{equation}
	almost surely for all $1 \leq k \leq N$ and $i \geq 0$. 
	
	Let us define
	\[	\overline{\lambda}_0(t)=\E\big[\lambda_{1,0}(t)\big|\eta_{1,0}>0\big],\,\quad \text{and} \quad \overline{\lambda}(t)=\E\left[\lambda_{1,1}(t)\right]\] and let $\mu$ be the law of $\gamma_{1,1}$, which is in $\mathcal{P}(\bD)$.
\end{assumption}
For simplicity, we write $\gamma$ and $\gamma_0$ as random functions with the same law as $\gamma_{1,1}$ and $\gamma_{1,0}$ respectively.
 
Then we recall the following FLLN result from \cite{forien-Zotsa2022stochastic}.  Let $\big(\overline{U}^N,\overline{I}^N\big) =N^{-1}  (U^N, I^N)$. 
\begin{theorem}\label{thm-FLLN}
	Under Assumption~\ref{TCL-AS-lambda-0},
	\begin{equation} \label{eqn-mfk-SF-conv}
	\big(\overline{\mathfrak{S}}^N,\overline{\mathfrak{F}}^N\big)\xrightarrow[N\to+\infty]{\mathbb{P}}(\overline{\mathfrak{S}},\overline{\mathfrak{F}})\quad\text{ in }\quad \bD^2
	\end{equation}
	where  $(\overline{\mathfrak{S}},\overline{\mathfrak{F}})$ satisfies the following system of equations, 
	\begin{numcases}{}
	\overline{\mathfrak{S}}(t)=\E \left[\gamma_{0}(t)\exp \left( - \int_{0}^{t} \gamma_0(r) \overline{\mathfrak F}(r) dr \right)\right] \nonumber \\
	\qquad \qquad + \int_{0}^{t} \E \left[\gamma(t-s) \exp \left( - \int_{s}^{t} \gamma(r-s) \overline{\mathfrak F}(r) dr \right) \right] \overline{\mathfrak S}(s) \overline{\mathfrak F}(s) ds, \label{eqG2s-G}\\
	\overline{\mathfrak{F}}(t)=\overline{I}(0)\overline{\lambda}_0(t)+\int_{0}^{t}\overline{\lambda}(t-s)\overline{\mathfrak S}(s)\overline{\mathfrak F}(s)ds\,.\label{eqG2s-F}
	\end{numcases}
	
	Given the solution   $(\overline{\mathfrak{S}},\overline{\mathfrak{F}})$, 
	\[(\overline{U}^N,\overline{I}^N)\xrightarrow[N\to+\infty]{\mathbb{P}} (\overline{U},\overline{I})\quad\text{ in }\quad \bD^2\]
	where $(\overline{U},\overline{I})$ is given by 
	\begin{align}
	\overline{U}(t)&=\E\left[ \mathds{1}_{t\geq\eta_0} \exp\left(-\int_0^t \gamma_0(r) \overline{\mathfrak{F}}(s) dr \right)\right]  \non
	\\ &\qquad+ \int_0^t \E\left[ \mathds{1}_{t-s\geq\eta} \exp\left(- \int_s^t \gamma(r-s) \overline{\mathfrak{F}}(r)dr \right)\right]\overline{\mathfrak S}(s) \overline{\mathfrak F}(s)ds\,, \label{eqn-barS}\\
	\overline{I}(t)&=\overline{I}(0)F^c_0(t)+\int_{0}^{t}F^c(t-s)\overline{\mathfrak{S}}(s)\overline{\mathfrak{F}}(s)ds\,. \label{eqn-barI}
	\end{align} 
\end{theorem}
\begin{remark}
Note that for each $t\geq0,\,\overline{\mathfrak S}(t)\leq1,\,\overline{\mathfrak F}(t)\leq\lambda_*$ and  $\overline{U}(t)+\overline{I}(t)=1$.
\end{remark}
\subsection{Main Results}
The purpose of this section is to establish an FCLT for the fluctuations of the stochastic sequence around its deterministic limit. More precisely, we define the following fluctuation process: for all $t\geq0,$ 
\begin{equation}\label{TCL-fluc-1}
\hat{\mathfrak F}^N(t):=\sqrt{N}\big(\overline{\mathfrak{F}}^N(t)-\overline{\mathfrak{F}}(t)\big),\text{ and }\hat{\mathfrak S}^N(t):=\sqrt{N}\big(\overline{\mathfrak{S}}^N(t)-\overline{\mathfrak{S}}(t)\big),
\end{equation}
and we want to find the limiting law of the pair $(\hat{\mathfrak S}^N,\hat{\mathfrak F}^N)$. 

\subsubsection{Assumptions}We introduce the following Assumptions.

\begin{assumption} \label{TCL-AS-lambda-1}
	The random functions $(\lambda,\gamma)$, of which $(\lambda_{k,i},\gamma_{k,i})_{k\ge 1, i \ge 1}$ are i.i.d. copies, satisfy the following properties:
		There exist a number $\ell\in\mathbb{N}^*$, a two random sequences 
		$0=\xi^0<\xi^1<\cdots<\xi^\ell=+\infty$ and $0=\zeta^0<\zeta^1<\cdots<\zeta^\ell=+\infty$
		and random functions $\lambda^j\in \bC,\,\gamma^j\in\bC$, $1\le j\le \ell$ such that 
		\begin{equation} \label{eqn-lambda}
		\lambda(t)=\sum_{j=1}^\ell\lambda^j(t){\bf1}_{[\xi^{j-1},\xi^j)}(t) \quad \text{ and } \quad \gamma(t)=\sum_{j=1}^\ell\gamma^j(t){\bf1}_{[\zeta^{j-1},\zeta^j)}(t).
		\end{equation}
		In addition, for any $T>0$, there exists deterministic nondecreasing function $\varphi_T \in \bC$ with $\varphi_T(0)=0$ such that $|\lambda^j(t)-\lambda^j(s)|\le \varphi_T(t-s)$ and $|\gamma^j(t)-\gamma^j(s)|\le \varphi_T(t-s)$ almost surely, for all $0 \le t, s \le T$, $1 \le j \le k$. 
\end{assumption}
\begin{assumption} \label{TCL-AS-lambda-2}
	There exists $\alpha>1/2$  such that for all $0\leq t\leq T,$ the function $\varphi_T$ from Assumption \ref{TCL-AS-lambda-1} satisfy 
	\begin{equation}\label{eqn-lambda-inc}
	\varphi_T(t)\le C t^\alpha, 
	\end{equation} 
	for some constant $C>0$. 
	Also, if $F_j$ denotes the c.d.f. of the r.v. $\xi^j$, and $G_j$ denotes the c.d.f. of the r.v. $\zeta^j,$ there exist $C'>0,$ and  
	$\rho>1/2$ such that, for any $1\le j\le \ell-1$, $0\le s<t\leq T$,
	\begin{equation}\label{hypF}
	F_j(t)-F_j(s)\le C'(t-s)^\rho \quad \text{ and } \quad G_j(t)-G_j(s)\leq C'(t-s)^\rho.
	\end{equation}
\end{assumption}


\begin{assumption} \label{TCL-AS-lambda}
		
		 There exist non-decreasing continuous functions $\phi_1,\,\phi_2,\,\psi_1,\,\psi_2$ and constants $\alpha_1>\frac{1}{2},\,\alpha_2>\frac{1}{2},\,\beta_1>1,\,\beta_2>1$ such that for all $0\leq s<u<t$, 
		 		\begin{itemize}
			\item[i)]$\E\left[\left(\lambda(t)-\lambda(s)\right)^2\right]\leq\left(\phi_1(t)-\phi_1(s)\right)^{\alpha_1};$
			\item[ii)]$\E\left[\left(\gamma(t)-\gamma(s)\right)^2\right]\leq\left(\phi_2(t)-\phi_2(s)\right)^{\alpha_2};$
			\item[iii)]$\E\left[\left(\lambda(t)-\lambda(u)\right)^2\left(\lambda(u)-\lambda(s)\right)^2\right]\leq\left(\psi_1(t)-\psi_1(s)\right)^{\beta_1};$
			\item[iv)]$\E\left[\left(\gamma(t)-\gamma(u)\right)^2\left(\gamma(u)-\gamma(s)\right)^2\right]\leq\left(\psi_2(t)-\psi_2(s)\right)^{\beta_2}.$
		\end{itemize} 
	
\end{assumption}
We note that, Assumptions~\ref{TCL-AS-lambda-1}, \ref{TCL-AS-lambda-2} and~\ref{TCL-AS-lambda} are not required to establish the FLLN in  \cite{forien-Zotsa2022stochastic}. These additional Assumptions are used to establish the tightness, of $\hat{\mathfrak F}^N$ and $\hat{\mathfrak S}^N$, see the proof of Lemma~\ref{TCL-tight-F-2}. 
There are many examples of the pair $(\lambda,\gamma)$ that satisfy them. 
A typical example of pair $(\lambda,\gamma)$ can be given by:
\begin{align} \label{TCL-def_lambda_gamma_SIS-1}
\lambda(t) = \lambda \mathds{1}_{0 \leq t < \eta}, \text{ and } \gamma(t) = \mathds{1}_{t \geq \eta}\text{ or }\gamma(t)=\left(1-e^{-(t-\eta)}\right)\mathds{1}_{t \geq \eta}.
\end{align}
For more examples and discussions on $\lambda(\cdot)$ we refer to Section 2.3 in \cite{pang2022-CLT-functional}, and  on $\gamma(\cdot)$ in  \cite{khalifi2022extending}.


\subsubsection{Statement of the main theorem}
\begin{definition}\label{TCL-def-1}
	Let $(\hat{\mathfrak J},\hat{\mathfrak M})$ be a two-dimensional centered continuous Gaussian process, with covariance functions: for $t,t'\geq0,$
	\begin{gather*}
	\Cov\big(\hat{\mathfrak J}(t),\hat{\mathfrak J}(t')\big)=\Cov\left(\gamma_{1,A_1(t)}(\varsigma_1(t)),\gamma_{1,A_1(t')}(\varsigma_1(t'))\right),\\
	\Cov\big(\hat{\mathfrak M}(t),\hat{\mathfrak M}(t')\big)=\Cov\left(\lambda_{1,A_1(t)}(\varsigma_1(t)),\lambda_{1,A_1(t')}(\varsigma_1(t'))\right),\\
	\Cov\big(\hat{\mathfrak M}(t),\hat{\mathfrak J}(t')\big)=\Cov\left(\lambda_{1,A_1(t)}(\varsigma_1(t)),\gamma_{1,A_1(t')}(\varsigma_1(t'))\right).
	\end{gather*} 
\end{definition}
Note that, thanks to Assumption~\ref{TCL-AS-lambda}, the process $(\hat{\mathfrak J},\hat{\mathfrak M})$ is continuous by applying Kolmogorov's continuity theorem for Gaussian processes.
\begin{remark}
In subsection~\ref{TCL-sec-NLG} Lemma~\ref{TCL-cv-F2-G2'} we give another expression for $\big(\hat{\mathfrak J},\hat{\mathfrak M}\big)$.
\end{remark}

We consider the following system of stochastic integral equations for which we have $(x,y)\in \bC^2$: 
\begin{numcases}{}
	x(t)=-\int_{0}^{t}\E\left[\gamma_{0}(t)\gamma_{0}(s)\exp\left(-\int_{0}^{t}\gamma_{0}(r)\overline{\mathfrak F}(r)dr\right)\right]y(s)ds\nonumber\\
	\hspace{1.2cm}-\int_{0}^{t}\int_{s}^{t}\E\left[\gamma(t-s)\gamma(r-s)\exp\left(-\int_{s}^{t}\gamma(u-s)\overline{\mathfrak F}(u)du\right)\right]y(r)\overline{\mathfrak{F}}(s)\overline{\mathfrak{S}}(s)drds\nonumber\\
	\hspace{1.2cm} +\int_{0}^{t}\E\left[\gamma(t-s)\exp\left(-\int_{s}^{t}\gamma(r-s)\overline{\mathfrak F}(r)dr\right)\right]\left(x(s)\overline{\mathfrak F}(s)-\hat{\mathfrak J}(s)\overline{\mathfrak F}(s)+\overline{\mathfrak S}(s)y(s)\right)ds\nonumber\\
	\hspace*{1.2cm}+\hat{\mathfrak J}(t),\label{TCL-in-uni-1-1}\\\nonumber\\
	y(t)=\int_{0}^{t}\overline{\lambda}(t-s)\left(x(s)\overline{\mathfrak F}(s)-\hat{\mathfrak J}(s)\overline{\mathfrak F}(s)+\overline{\mathfrak S}(s)y(s)\right)ds+\hat{\mathfrak M}(t).\label{TCL-in-uni-2-1}
\end{numcases}

\begin{lemma}\label{TCL-exist}
 The set of equations	\eqref{TCL-in-uni-1-1}-\eqref{TCL-in-uni-2-1} has a unique solution $(x,y)\in \bC^2.$
 
 We denote its solution by $(\hat{\mathfrak S},\hat{\mathfrak F})\in \bC^2$.
\end{lemma}
\begin{proof}
	If we denote by $(x^1,y^1)$ and $(x^2,y^2)$ two solutions of  \eqref{TCL-in-uni-1-1}-\eqref{TCL-in-uni-2-1},
	 as $\overline{\mathfrak S}\leq1$ and $\overline{\mathfrak F}\leq\lambda_*$,
	 we obtain by easy computations that there exists $C_T$ such that for all $t\in[0,T]$, 
	\begin{equation}\label{TCL-ex}
	|x^1(t)-x^2(t)|+|y^1(t)-y^2(t)|\leq C_T\int_0^t |x^1(s)-x^2(s)|+|y^1(s)-y^2(s)|ds.
	\end{equation} 
	Uniqueness then follows from Gronwall's Lemma. 
	
	Now local existence follows by an approximation procedure, which exploits the estimate \eqref{TCL-ex}. Global existence then follows 
	from the estimates \eqref{TCL-ex}, which forbid explosion. Lemma \ref{TCL-exist} is established.
\end{proof}

The following is our main result.
\begin{theorem}\label{TCL}
	Under Assumptions~\ref{TCL-AS-lambda-0}--\ref{TCL-AS-lambda},  
	\begin{equation}
		\big(\hat{\mathfrak S}^N,\hat{\mathfrak F}^N\big)\Rightarrow\big(\hat{\mathfrak S},\hat{\mathfrak F}\big)\quad\text{ in }  \quad \bD^2,
	\end{equation}
	where $\big(\hat{\mathfrak S},\hat{\mathfrak F}\big)$ is the unique continuous solution of the system of stochastic integral equations \eqref{TCL-in-uni-1-1}-\eqref{TCL-in-uni-2-1}, that is, 
	\begin{numcases}{}
\hat{\mathfrak S}(t)=-\int_{0}^{t}\E\left[\gamma_{0}(t)\gamma_{0}(s)\exp\left(-\int_{0}^{t}\gamma_{0}(r)\overline{\mathfrak F}(r)dr\right)\right]\hat{\mathfrak F}(s)ds\nonumber\\
\hspace{1cm}-\int_{0}^{t}\int_{s}^{t}\E\left[\gamma(t-s)\gamma(r-s)\exp\left(-\int_{s}^{t}\gamma(u-s)\overline{\mathfrak F}(u)du\right)\right]\hat{\mathfrak F}(r)\overline{\mathfrak{F}}(s)\overline{\mathfrak{S}}(s)drds\nonumber\\
\quad\quad+\int_{0}^{t}\E\left[\gamma(t-s)\exp\left(-\int_{s}^{t}\gamma(r-s)\overline{\mathfrak F}(r)dr\right)\right]\left(\hat{\mathfrak S}(s)\overline{\mathfrak F}(s)-\hat{\mathfrak J}(s)\overline{\mathfrak F}(s)+\overline{\mathfrak S}(s)\hat{\mathfrak F}(s)\right)ds\nonumber\\
\hspace*{1cm}+\hat{\mathfrak J}(t),\label{TCL-in-uni-1}\\\nonumber\\
\hat{\mathfrak F}(t)=\int_{0}^{t}\overline{\lambda}(t-s)\left(\hat{\mathfrak S}(s)\overline{\mathfrak F}(s)-\hat{\mathfrak J}(s)\overline{\mathfrak F}(s)+\overline{\mathfrak S}(s)\hat{\mathfrak F}(s)\right)ds+\hat{\mathfrak M}(t),\label{TCL-in-uni-2}
\end{numcases}
where we recall that $\left(\hat{\mathfrak J},\hat{\mathfrak M}\right)$ is specified by Definition~\ref{TCL-def-1}.
	
\end{theorem}
%


We define
\[\hat{ I}^N(t)=\sqrt{N}\big(\overline{I}^N(t)-\overline{I}(t)\big) \quad \text{ and } \quad \hat{U}^N(t)=\sqrt{N}\big(\overline{U}^N(t)-\overline{U}(t)\big), \quad t \ge 0. \]
Replacing $\lambda_{k,i}(t)$ by $ \mathds{1}_{t<\eta_{k,i}}$ and using the fact that $\overline{I}^N(t)+\overline{U}^N(t)=1$ and $\overline{I}(t)+\overline{U}(t)=1,$
We obtain the following Corollary from Theorem~\ref{TCL} and the system of equations \eqref{TCL-in-uni-1}-\eqref{TCL-in-uni-2}. 
\begin{coro} \label{coro-UI}
Under Assumptions~\ref{TCL-AS-lambda-0}--\ref{TCL-AS-lambda},  
\begin{equation}
\big(\hat{U}^N,\hat{I}^N\big)\Rightarrow\big(\hat{U},\hat{I}\big)\quad\text{ in } \quad \bD^2,
\end{equation}
where $\big(\hat{U},\hat{I}\big)$ is the unique continuous solution of the system of equations:
\begin{numcases}{}
\hat{U}(t)=-\int_{0}^{t}\E\left[\mathds{1}_{t\geq\eta_0}\gamma_{0}(s)\exp\left(-\int_{0}^{t}\gamma_{0}(r)\overline{\mathfrak F}(r)dr\right)\right]\hat{\mathfrak F}(s)ds\nonumber\\
\hspace{1cm}-\int_{0}^{t}\int_{s}^{t}\E\left[\mathds{1}_{t-s\geq\eta}\gamma(r-s)\exp\left(-\int_{s}^{t}\gamma(u-s)\overline{\mathfrak F}(u)du\right)\right]\hat{\mathfrak F}(r)\overline{\mathfrak{F}}(s)\overline{\mathfrak{S}}(s)drds\nonumber\\
\quad\quad+\int_{0}^{t}\E\left[\mathds{1}_{t-s\geq\eta}\exp\left(-\int_{s}^{t}\gamma(r-s)\overline{\mathfrak F}(r)dr\right)\right]\left(\hat{\mathfrak S}(s)\overline{\mathfrak F}(s)-\hat{\mathfrak J}_1(s)\overline{\mathfrak F}(s)+\overline{\mathfrak S}(s)\hat{\mathfrak F}(s)\right)ds\nonumber\\
\hspace*{1cm}+\hat{\mathfrak J}_1(t),\label{TCL-in-uni-1-U}\\\nonumber\\
\hat{I}(t)=\int_{0}^{t}F^c(t-s)\left(\hat{\mathfrak S}(s)\overline{\mathfrak F}(s)-\hat{\mathfrak J}(s)\overline{\mathfrak F}(s)+\overline{\mathfrak S}(s)\hat{\mathfrak F}(s)\right)ds+\hat{\mathfrak M}_1(t),\label{TCL-in-uni-2-I}
\end{numcases}
where $\big(\hat{\mathfrak J}_1,\hat{\mathfrak M}_1\big)$ is a centered continuous Gaussian process as given in Definition~\ref{TCL-def-1} where we replace $\lambda_{1,A_1(t)}$ and $\gamma_{1,A_1(t)}$ by $\mathds{1}_{t<\eta_{1,A_1(t)}}$ and $\mathds{1}_{t\geq\eta_{1,A_1(t)}}$ in the expressions of $\hat{\mathfrak J}$ and $\hat{\mathfrak M}$, respectively.
\end{coro}

\subsection{Relaxing Assumption~\ref{TCL-AS-lambda-0}}
From \cite{forien-Zotsa2022stochastic} without condition $\eqref{eqq10}$ of Assumption~\ref{TCL-AS-lambda-0}, this means that an infected individual can be reinfected, the limit obtained in the FLLN satisfies a different set of equations.
More precisely, equation \eqref{eqG2s-F} is replaced by \eqref{eqG2s-G-w} and \eqref{eqn-barI} by \eqref{eqG2s-F-w}, where \eqref{eqG2s-G-w} and \eqref{eqG2s-F-w} are given below:

\begin{align}
&\overline{\mathfrak{F}}(t)=\E \left[\lambda_{0}(t)\exp \left( - \int_{0}^{t} \gamma_0(r) \overline{\mathfrak F}(r) dr \right)\right] \nonumber\\
&\hspace{2cm} + \int_{0}^{t} \E \left[\lambda(t-s) \exp \left( - \int_{s}^{t} \gamma(r-s) \overline{\mathfrak F}(r) dr \right) \right] \overline{\mathfrak S}(s) \overline{\mathfrak F}(s) ds, \label{eqG2s-G-w}\\
&\overline{I}(t)=\E\left[ \mathds{1}_{\eta_0>t} \exp\left(-\int_0^t \gamma_0(r) \overline{\mathfrak{F}}(s) dr \right)\right] \nonumber\\
&\hspace{2cm}+ \int_0^t \E\left[ \mathds{1}_{\eta>t-s} \exp\left(- \int_s^t \gamma(r-s) \overline{\mathfrak{F}}(r)dr \right)\right]\overline{\mathfrak S}(s) \overline{\mathfrak F}(s)ds\,.\label{eqG2s-F-w}
\end{align}

In the same way without condition $\eqref{eqq10}$ of Assumption~\ref{TCL-AS-lambda-0}, the limit obtained in the FCLT (Theorem \ref{TCL} and Corollary \ref{coro-UI}) satisfies a different set of equations.
More precisely, equation \eqref{TCL-in-uni-2} is replaced by \eqref{TCL-eqG2s-G-w} and \eqref{TCL-in-uni-2-I} by \eqref{TCL-eqG2s-F-w}, where \eqref{TCL-eqG2s-G-w} and \eqref{TCL-eqG2s-F-w} are given below. In that case, the convergence of $\hat{\mathfrak F}^N(t)$ follows an analogous argument as that used for $\hat{\mathfrak S}^N(t)$ in  Theorem \ref{TCL}. 
%
More precisely, without condition $\eqref{eqq10}$ of Assumption~\ref{TCL-AS-lambda-0}, equation \eqref{TCL-n-lbda} is replaced by a similar equation in  \eqref{TCL-gamma-decom}. In fact, for each fixed $k$, we replace $\gamma_{k,A_k^N(t)}$ and $\gamma_{k,i}$ by $\lambda_{k,A_k^N(t)}$ and $\lambda_{k,i}$ in the expressions in \eqref{TCL-gamma-decom}. Consequently instead of the expression in \eqref{TCL-eqc10}, one gets a different expression, which resembles the expression in \eqref{TCL-eq-eqc10-1}, so that the proof follows from a similar argument.
\begin{align}
&\hat{\mathfrak F}(t)=-\int_{0}^{t}\E\left[\lambda_{0}(t)\gamma_{0}(s)\exp\left(-\int_{0}^{t}\gamma_{0}(r)\overline{\mathfrak F}(r)dr\right)\right]\hat{\mathfrak F}(s)ds\nonumber\\
&\hspace{1cm}-\int_{0}^{t}\int_{s}^{t}\E\left[\lambda(t-s)\gamma(r-s)\exp\left(-\int_{s}^{t}\gamma(u-s)\overline{\mathfrak F}(u)du\right)\right]\hat{\mathfrak F}(r)\overline{\mathfrak{F}}(s)\overline{\mathfrak{S}}(s)drds\nonumber\\
&\quad\quad+\int_{0}^{t}\E\left[\lambda(t-s)\exp\left(-\int_{s}^{t}\gamma(r-s)\overline{\mathfrak F}(r)dr\right)\right]\left(\hat{\mathfrak S}(s)\overline{\mathfrak F}(s)-\hat{\mathfrak J}(s)\overline{\mathfrak F}(s)+\overline{\mathfrak S}(s)\hat{\mathfrak F}(s)\right)ds\nonumber\\
&\hspace*{1cm}+\hat{\mathfrak M}(t), \label{TCL-eqG2s-G-w}\\
&\hat{ I}(t)=-\int_{0}^{t}\E\left[\mathds{1}_{\eta_0>t}\gamma_{0}(s)\exp\left(-\int_{0}^{t}\gamma_{0}(r)\overline{\mathfrak F}(r)dr\right)\right]\hat{\mathfrak F}(s)ds\nonumber\\
&\hspace{1cm}-\int_{0}^{t}\int_{s}^{t}\E\left[\mathds{1}_{\eta>t-s}\gamma(r-s)\exp\left(-\int_{s}^{t}\gamma(u-s)\overline{\mathfrak F}(u)du\right)\right]\hat{\mathfrak F}(r)\overline{\mathfrak{F}}(s)\overline{\mathfrak{S}}(s)drds\nonumber\\
&\quad\quad+\int_{0}^{t}\E\left[\mathds{1}_{\eta>t-s}\exp\left(-\int_{s}^{t}\gamma(r-s)\overline{\mathfrak F}(r)dr\right)\right]\left(\hat{\mathfrak S}(s)\overline{\mathfrak F}(s)-\hat{\mathfrak J}(s)\overline{\mathfrak F}(s)+\overline{\mathfrak S}(s)\hat{\mathfrak F}(s)\right)ds\nonumber\\
&\hspace*{1cm}+\hat{\mathfrak M}_1(t) \label{TCL-eqG2s-F-w}.
\end{align}


\bigskip

\section{Proof of Theorem~\ref{TCL}}\label{TCL-sec-2}
We recall the definitions of $\hat{\mathfrak F}^N$ and $\hat{\mathfrak S}^N$, from \eqref{TCL-fluc-1} below: 
\begin{equation*}
\hat{\mathfrak F}^N(t):=\sqrt{N}\left(\overline{\mathfrak{F}}^N(t)-\overline{\mathfrak{F}}(t)\right),\text{ and }
\hat{\mathfrak S}^N(t):=\sqrt{N}\left(\overline{\mathfrak{S}}^N(t)-\overline{\mathfrak{S}}(t)\right).
\end{equation*}
Let us write, 
\begin{align}\label{TCL-eq-eqc10-01}
\hat{\mathfrak F}^N(t)&=\sqrt{N}\left(\frac{1}{N}\sum_{k=1}^{N}\left(\lambda_{k,A^N_k(t)}(\varsigma^{N}_k(t))-\lambda_{k,A_k(t)}(\varsigma_k(t))\right)\right) \nonumber\\
& \qquad +\sqrt{N}\left(\frac{1}{N}\sum_{k=1}^{N}\lambda_{k,A_k(t)}(\varsigma_k(t))-\E\left[\lambda_{1,A_1(t)}(\varsigma_1(t))\right]\right)\nonumber\\
&=:\hat{\mathfrak F}^N_1(t)+\hat{\mathfrak F}^N_2(t).
\end{align}
and 
\begin{align}\label{TCL-eq-eqc10-0}
\hat{\mathfrak S}^N(t)&=\sqrt{N}\left(\frac{1}{N}\sum_{k=1}^{N}\left(\gamma_{k,A^N_k(t)}(\varsigma^{N}_k(t))-\gamma_{k,A_k(t)}(\varsigma_k(t))\right)\right) \nonumber\\
& \qquad +\sqrt{N}\left(\frac{1}{N}\sum_{k=1}^{N}\gamma_{k,A_k(t)}(\varsigma_k(t))-\E\left[\gamma_{1,A_1(t)}(\varsigma_1(t))\right]\right)\nonumber\\
&=:\hat{\mathfrak S}^N_1(t)+\hat{\mathfrak S}^N_2(t).
\end{align}

	Since $\left(\gamma_{k,A_k(\cdot)}(\varsigma_k(\cdot)),\lambda_{k,A_k(\cdot)}(\varsigma_k(\cdot))\right)_k$ are i.i.d $\bD^2$-valued random variables with each component satisfying Assumption~\ref{TCL-AS-lambda}, by applying the central limit theorem in $\bD$, each component (see Theorem~2 in \cite{hahn1978central}) it follows that $\hat{\mathfrak S}^N_2\Rightarrow \hat{\mathfrak J}$ and $\hat{\mathfrak F}^N_2\Rightarrow\hat{\mathfrak M}$ in $\bD$ as $N\to\infty$ respectively. Then, as $\bD$ is separable from \cite[Lemma~$5.2$]{pang2007martingale} the pair $\big(\hat{\mathfrak S}^N_2,\hat{\mathfrak F}^N_2\big)$ is $\bC-$tight in $\bD^2$ and using the uniqueness of the limit of $\hat{\mathfrak S}^N_2$ and $\hat{\mathfrak F}^N_2$, the convergence in $\bD^2$ of the pair $\big(\hat{\mathfrak S}^N_2,\hat{\mathfrak F}^N_2\big)$ follows. Moreover, given the convergence of $\big(\hat{\mathfrak S}^N_2,\hat{\mathfrak F}^N_2\big)$, by the continuous mapping theorem $\hat{\mathfrak S}^N_2\hat{\mathfrak F}^N_2$ converges in $\bD$. It follows that for each $t,t'\geq0$, the covariance of $\hat{\mathfrak S}^N_2(t')$ and $\hat{\mathfrak F}^N_2(t),$ which is given by $\Cov\left(\gamma_{1,A_1(t')}(\varsigma_1(t'),\lambda_{1,A_1(t)}(\varsigma_1(t)))\right),$ converges to the covariance of the limit process of $\hat{\mathfrak S}^N_2(t')$ and $\hat{\mathfrak F}^N_2(t)$. Hence we have the following Lemma: 
\begin{lemma}\label{TCL-cv-F2-G2}
	Under Assumption~\ref{TCL-AS-lambda}, as $N\to+\infty$,
	\begin{equation*}
	\big(\hat{\mathfrak S}^N_2,\hat{\mathfrak F}^N_2\big)\Rightarrow\big(\hat{\mathfrak J},\hat{\mathfrak M}\big) \quad \text{ in } \quad \bD^2.
	\end{equation*}
	where $\big(\hat{\mathfrak J},\hat{\mathfrak M}\big)$ is a centered continuous 2-dimensional Gaussian process 
 given in Definition \ref{TCL-def-1}.
\end{lemma}
proving the convergence of the pair $\big(\hat{\mathfrak S}^N_1,\hat{\mathfrak F}^N_1\big)$ is highly nontrivial.
We start by the following decomposition of the pair in the next subsection.

\subsection{Decomposition of the fluctuations  $\big(\hat{\mathfrak S}^N_1,\hat{\mathfrak F}^N_1\big)$}

In \cite{forien-Zotsa2022stochastic} we had made a coupling between $A^N_k$ and $A_k$ and if instead we give ourselves a coupling on $\R_+\times \bD^2\times\R_+$, we can define a new coupling. However if we look at the law of the first time for which $A^N_k\neq A_k$, then the law of this time is the same for both couplings.
 
We then introduce a Poisson random measure $Q_k$ on $\R_+\times \bD^2\times\R_+$, so that the mean measure of the PRM is 
\[ds \times\mathbb{P}(d\lambda,d\gamma)\times du.\]
We denote by $\overline{Q}_k$ its compensated measure.

 For $0 \le s <t$, and $\gamma\in\bD,\,\phi\in\bD$ we set 
\begin{equation*}
P_{k}(s,t,\gamma,\phi)=\int_{[s,t]\times\bD^2\times\R_+}\mathds{1}_{u <\gamma(r-s)\phi(r)}Q_k(dr,d\lambda',d\gamma',du).
\end{equation*}
So we define $A^N_k$, as follows
\begin{equation*}
A^N_k(t)=\int_{[0,t]\times\bD^2\times\R_+}\mathds{1}_{u\leq\Upsilon^N_k(r^-)}Q_k(dr,d\lambda,d\gamma,du)\\
\end{equation*} 
where
\begin{align*}
\Upsilon^N_k(t)=\left(\gamma_{k,0}(t)\mathds{1}_{P_{k}(0,t,\gamma_{k,0},\overline{\mathfrak F}^N)=0}+\int_{0}^{t}\int_{\bD^2}\int_{0}^{\Upsilon_k^N(s-)}\gamma(t-s)\mathds{1}_{P_{k}(s,t,\gamma,\overline{\mathfrak F}^N)=0}Q_k(ds,d\lambda,d\gamma,du)\right)\overline{\mathfrak F}^N(t)
\end{align*}
and 
\begin{equation}\label{TCL-eq-force-2}
	\overline{\mathfrak F}^N(t)=\frac{1}{N}\sum_{k=1}^{N}\left\{\lambda_{k,0}(t)\mathds{1}_{P_{k}(0,t,\gamma_{k,0},\overline{\mathfrak F}^N)=0}+\int_{0}^{t}\int_{\bD^2}\int_{0}^{\Upsilon_k^N(s-)}\lambda(t-s)\mathds{1}_{P_{k}(s,t,\gamma,\overline{\mathfrak F}^N)=0}Q_k(ds,d\lambda,d\gamma,du)\right\}.
\end{equation}
Under the conditions on $\lambda_{k,i}$ in Assumption~\ref{TCL-AS-lambda-0},  \eqref{TCL-eq-force-2} becomes
\begin{equation}\label{TCL-n-lbda}
	\overline{\mathfrak F}^N(t)
	=\frac{1}{N}\sum_{k=1}^{N}\left\{\lambda_{k,0}(t)+\int_{0}^{t}\int_{\bD^2}\int_{0}^{\Upsilon^N_k(s^-)}\lambda(t-s)Q_k(ds,d\lambda,d\gamma,du)\right\}.
\end{equation}
On the other hand, we have
\begin{equation}\label{TCL-gamma-decom}
\overline{\mathfrak S}^N(t)=\frac{1}{N}\sum_{k=1}^{N}\left\{\gamma_{k,0}(t)\mathds{1}_{P_{k}(0,t,\gamma_{k,0},\overline{\mathfrak F}^N)=0}+\int_{0}^{t}\int_{\bD^2}\int_{0}^{\Upsilon_k^N(s-)}\gamma(t-s)\mathds{1}_{P_{k}(s,t,\gamma,\overline{\mathfrak F}^N)=0}Q_k(ds,d\lambda,d\gamma,du)\right\}.
\end{equation}

Similarly, 
\begin{equation*}
A_k(t)=\int_{[0,t]\times\bD^2\times\R_+}\mathds{1}_{u\leq\Upsilon_k(r^-)}Q_k(dr,d\lambda,d\gamma,du)\\
\end{equation*} 
with
\begin{align*}
\Upsilon_k(t)=\left(\gamma_{k,0}(t)\mathds{1}_{P_{k}(0,t,\gamma_{k,0},\overline{\mathfrak F})=0}+\int_{0}^{t}\int_{\bD^2}\int_{0}^{\Upsilon_k(s-)}\gamma(t-s)\mathds{1}_{P_{k}(s,t,\gamma,\overline{\mathfrak F})=0}Q_k(ds,d\lambda,d\gamma,du)\right)\overline{\mathfrak F}(t)
\end{align*}
where $\overline{\mathfrak{F}}$ is given by \eqref{eqG2s-F}.

Consequently we have
\begin{equation}\label{TCL-lbda-decomp}
\sum_{k=1}^{N}\lambda_{k,A_k(t)}(\varsigma_k(t)):=\sum_{k=1}^N\left\{\lambda_{k,0}(t)+\int_{0}^{t}\int_{\bD^2}\int_{0}^{\Upsilon_k(s^-)}\lambda(t-s)Q_k(ds,d\lambda,d\gamma,du)\right\},
\end{equation}
and
\begin{align}
\sum_{k=1}^{N}\gamma_{k,A_k(t)}(\varsigma_k(t))
:=\sum_{k=1}^{N}\left\{\gamma_{k,0}(t)\mathds{1}_{P_{k}(0,t,\gamma_{k,0},\overline{\mathfrak F})=0}+\int_{0}^{t}\int_{\bD^2}\int_{0}^{\Upsilon_k(s-)}\gamma(t-s)\mathds{1}_{P_{k}(s,t,\gamma,\overline{\mathfrak F})=0}Q_k(ds,d\lambda,d\gamma,du)\right\}\label{TCL-gamma-decom-2}\,. 
\end{align}

Using the fact that 
\begin{equation*}
\Upsilon_k^N(s)-\Upsilon_k(s)=\gamma_{k,A^N_k(s)}(\varsigma_{k}^N(s))\overline{\mathfrak F}^N(s)-\gamma_{k,A_k(s)}(\varsigma_{k}(s))\overline{\mathfrak F}(s),
\end{equation*}
from expression \eqref{TCL-eq-eqc10-0}, it follows that,
\begin{equation}
\frac{1}{\sqrt{N}}\sum_{k=1}^{N}\left(\Upsilon_k^N(s)-\Upsilon_k(s)\right)=\hat{\mathfrak S}^N(s)\overline{\mathfrak F}^N(s)-\hat{\mathfrak S}^N_2(s)\overline{\mathfrak F}^N(s)+\tilde{\mathfrak S}^N(s)\hat{\mathfrak F}^N(s),\label{eqc1-2}
\end{equation}
where 
\[\tilde{\mathfrak S}^N(t)=\frac{1}{N}\sum_{k=1}^{N}\gamma_{k,A_k(t)}(\varsigma_k(t)),\]
and from expressions \eqref{TCL-eq-eqc10-01} and \eqref{TCL-eq-eqc10-0}, we obtain the following Proposition.

\begin{prop}
	Under the condition \eqref{eqq10} in Assumption~\ref{TCL-AS-lambda-0},  for every $t\geq0,$
	\begin{align}
	\hat{\mathfrak F}^N_1(t)
	&=\frac{1}{\sqrt{N}}\sum_{k=1}^{N}\int_{0}^{t}\int_{\bD^2}\int_{\Upsilon_k(s-)\wedge\Upsilon^N_k(s-)}^{\Upsilon_k(s-)\vee\Upsilon^N_k(s-)}\lambda(t-s) \sign(\Upsilon^N_k(s^-)-\Upsilon_k(s^-))\overline{Q}_k(ds,d\lambda,d\gamma,du)\nonumber\\
	&\quad+\int_{0}^{t}\overline{\lambda}(t-s)\hat{\mathfrak S}^N(s)\overline{\mathfrak F}^N(s)ds-\int_{0}^{t}\overline{\lambda}(t-s)\hat{\mathfrak S}^N_2(s)\overline{\mathfrak F}^N(s)ds+\int_{0}^{t}\overline{\lambda}(t-s)\tilde{\mathfrak S}^N(s)\hat{\mathfrak F}^N(s)ds,\label{TCL-eqc10}
	\end{align}
	and
	\begin{align}\label{TCL-eq-eqc10-1}
	\hat{\mathfrak S}^N_1(t)&=\frac{1}{\sqrt{N}}\sum_{k=1}^{N}\gamma_{k,0}(t)\left(\mathds{1}_{P_{k}(0,t,\gamma_{k,0},\overline{\mathfrak F}^N)=0}-\mathds{1}_{P_{k}(0,t,\gamma_{k,0},\overline{\mathfrak F})=0}\right)\nonumber\\
	&+\frac{1}{\sqrt{N}}\sum_{k=1}^{N}\int_{0}^{t}\int_{\bD^2}\int_{0}^{\Upsilon_k^N(s)}\gamma(t-s)\left(\mathds{1}_{P_k\left(s,t,\gamma,\overline{\mathfrak{F}}^N\right)=0}-\mathds{1}_{P_k\left(s,t,\gamma,\overline{\mathfrak{F}}\right)=0}\right)Q_k(ds,d\lambda,d\gamma,du)\nonumber
	\\&+\frac{1}{\sqrt{N}}\sum_{k=1}^{N}\int_{0}^{t}\int_{\bD^2}\int_{\Upsilon_k(s)\wedge\Upsilon_k^N(s)}^{\Upsilon_k(s)\vee\Upsilon_k^N(s)}\gamma(t-s)\mathds{1}_{P_k\left(s,t,\gamma,\overline{\mathfrak{F}}\right)=0}\sign(\Upsilon_k^N(s^-)-\Upsilon_k(s^-))Q_k(ds,d\lambda,d\gamma,du)\nonumber\\
	&=:\hat{\mathfrak{S}}^N_{1,0}(t)+\hat{\mathfrak{S}}^N_{1,1}(t)+\hat{\mathfrak{S}}^N_{1,2}(t).
	\end{align}
\end{prop}

\begin{remark}
	Note that, in the rest $Q_k(dr,du)$ can be seen as the projection of $Q_k(dr,d\lambda,d\gamma,du)$ on $\R_+\times\R_+$.
\end{remark}
 
\subsection{Two continuous integral mappings}
Let $\phi_1:\R_+^3\to \R_+$ and $\phi_2:\R_+^2\to\R,$ be bounded functions and Borel measurable and let $\varPsi_1:\bD_+^7\to \bD$, be given by 
\begin{multline}\label{TCL-uni-c1}
\varPsi_1(f)(t)=f_2(t)-f_3(t)-f_4(t)+\int_{0}^{t}\int_{s}^{t}\phi_1(t,s,r)f_1(r)f_5(s)f_6(s)drds\\
-\int_{0}^{t}\phi_2(t,s)f_5(s)f_2(s)ds+\int_{0}^{t}\phi_2(t,s)f_5(s)f_4(s)ds-\int_{0}^{t}\phi_2(t,s)f_1(s)f_7(s)ds,
\end{multline}
 for all $t$,  where $f=(f_1,f_2,f_3,f_4,f_5,f_6,f_7)\in\bD_+^7 $.
\begin{lemma}\label{TCL-cont-map-1-c}
	If $f^n\to f$ in $\bD^7_+,$ as $n\to\infty$ and $f$ is continuous on $\R_+$, then $\varPsi_1(f^n)\to \varPsi_1(f)$ in $\bD$ as $n\to\infty$.
\end{lemma}
\begin{proof}
Since $f^n\to f$ in $\bD^7_+,$ as $n\to\infty$ and $f$ is continuous, $\|f^n-f\|_T\to0$ as $n\to\infty$. Consequently $(\|f^n\|_T)_n$ is bounded and it follows easily that there exists $C_T>0$,
	\begin{equation*}
	\|\varPsi_1(f^n)-\varPsi_1(f)\|_T\leq C_T\|f^n-f\|_T.
	\end{equation*}
	Where for $f,g\in\bD^7_+,$ we define,
	\begin{equation*}
	\|f-g\|_T=\sup_{0\le t\le T}|f(t)-g(t)|,
	\end{equation*}
	with $|\cdot|$ the Euclidian norm with respect to the dimension of the space. 
\end{proof}
On the other hand, let $\varPsi_2:\bD_+^6\to \bD$, be given for all $t$ by
\begin{multline}
\varPsi_2(f)(t)
=f_1(t)-f_2(t)-\int_{0}^{t}\overline{\lambda}(t-s)f_3(s)f_4(s)ds-\int_{0}^{t}\overline{\lambda}(t-s)f_5(s)f_4(s)ds+\int_{0}^{t}\overline{\lambda}(t-s)f_6(s)f_1(s)ds,
\end{multline}
for any $f=(f_1,f_2,f_3,f_4,f_5,f_6)\in\bD_+^6$.
\begin{lemma}
	If $f^n\to f$ in $\bD^6_+,$ as $n\to\infty$ and $f$ is continuous on $\R_+$, then $\varPsi_2(f^n)\to \varPsi_2(f)$ in $\bD$ as $n\to\infty$.
\end{lemma}
The proof is similar to Lemma~\ref{TCL-cont-map-1-c}.

\begin{remark}
If $\varPsi_1\big(\hat{\mathfrak F},\hat{\mathfrak S},\hat{\mathfrak S}_{1,0},\hat{\mathfrak J},\overline{\mathfrak{F}},\overline{\mathfrak{S}},\overline{\mathfrak{S}}\big)=0$ and
	$\varPsi_2\big(\hat{\mathfrak F},\hat{\mathfrak S},\hat{\mathfrak M},\hat{\mathfrak J},\overline{\mathfrak{F}},\overline{\mathfrak{S}},\overline{\mathfrak{S}}\big)=0,$ then we have a solution of \eqref{TCL-in-uni-1}-\eqref{TCL-in-uni-2}.
\end{remark}
\subsection{Proof of Theorem~\ref{TCL}}
We refer to Section~\ref{TCL-sec-tight-r} for the proof of the following tightness result.
\begin{lemma}\label{TCL-tight-F-2}
	The sequence $\big(\hat{\mathfrak S}^N,\hat{\mathfrak F}^N,\hat{\mathfrak S}^N_2,\hat{\mathfrak F}^N_2,\hat{\mathfrak S}^N_{1,0},\hat{\mathfrak S}^N_{1,1}\big)_{N\ge1}$ is tight in $\bD^6$. 
\end{lemma}
We have the following characterisation for the limit of any converging subsequence, of $\hat{\mathfrak S}_{1,0}^N$ in $\bD$ and we refer to Section~\ref{TCL-sec-G10} for the proof.

\begin{lemma}\label{TCL-lem-L2}
	Let $\big(\hat{\mathfrak F},\hat{\mathfrak S}_{1,0}\big)$ be a limit of a converging subsequence of $\big(\hat{\mathfrak F}^N,\hat{\mathfrak S}_{1,0}^N\big)$. Then, almost surely, 
	\begin{equation*}
	\hat{\mathfrak S}_{1,0}(t)=-\int_{0}^{t}\E\left[\gamma_{0}(t)\gamma_{0}(s)\exp\left(-\int_{0}^{t}\gamma_{0}(r)\overline{\mathfrak F}(r)dr\right)\right]\hat{\mathfrak F}(s)ds.
	\end{equation*}
\end{lemma}
Taking
\begin{equation*}
	\phi_1(t,s,r)=\E\left[\gamma(t-s)\gamma(r-s)\exp\left(-\int_{s}^{t}\gamma(u-s)\overline{\mathfrak F}(u)du\right)\right],
\end{equation*}
and 
\begin{equation*}
	\phi_2(t,s)=\E\left[\gamma(t-s)\exp\left(-\int_{s}^{t}\gamma(u-s)\overline{\mathfrak F}(u)du\right)\right] 
\end{equation*}
in \eqref{TCL-uni-c1} we can establish the following Lemma:
\begin{lemma}\label{TCL-lem-L1} Under Assumptions~\ref{TCL-AS-lambda-0}--\ref{TCL-AS-lambda}, as $N\to\infty,$
	\begin{equation*}
	\varPsi_1\big(\hat{\mathfrak F}^N,\hat{\mathfrak S}^N,\hat{\mathfrak S}^N_{1,0},\hat{\mathfrak S}^N_2,\overline{\mathfrak{F}}^N,\overline{\mathfrak{S}}^N,\tilde{\mathfrak{S}}^N\big)\Rightarrow0 \quad \text{ in }\quad \bD,
	\end{equation*}
	and 
	\begin{equation*}
		\varPsi_2\big(\hat{\mathfrak F}^N,\hat{\mathfrak S}^N,\hat{\mathfrak F}^N_2,\hat{\mathfrak S}^N_2,\overline{\mathfrak{F}}^N,\overline{\mathfrak{S}}^N,\tilde{\mathfrak{S}}^N\big)\Rightarrow0 \quad \text{ in }\quad \bD.
	\end{equation*}
\end{lemma}
	We refer to Section~\ref{TCL-sec-L1} for the proof.

Hence the following characterisation follows:
\begin{lemma}\label{TCL-Lem-L3}
	Let $\big(\hat{\mathfrak S},\hat{\mathfrak F},\hat{\mathfrak S}_{1,0}\big)$ be a limit of a converging subsequence, of  $\big(\hat{\mathfrak S}^N,\hat{\mathfrak F}^N,\hat{\mathfrak S}_{1,0}^N\big)$ in $\bD^3$. Then almost surely,
	\begin{equation}\label{TCL-lem-eq-L1}
	\begin{cases}{}
	\varPsi_1\big(\hat{\mathfrak F},\hat{\mathfrak S},\hat{\mathfrak S}_{1,0},\hat{\mathfrak J},\overline{\mathfrak{F}},\overline{\mathfrak{S}},\overline{\mathfrak{S}}\big)=0, \\
	\varPsi_2\big(\hat{\mathfrak F},\hat{\mathfrak S},\hat{\mathfrak M},\hat{\mathfrak J},\overline{\mathfrak{F}},\overline{\mathfrak{S}},\overline{\mathfrak{S}}\big)=0,
	\end{cases}
	\end{equation}
	where we recall that the pair $(\hat{\mathfrak J},\hat{\mathfrak M})$ is given by Definition~\ref{TCL-def-1}.
\end{lemma}
\begin{proof}

	As the space $\bD$ is separable, and $\overline{\mathfrak{F}}^N,\,\overline{\mathfrak{S}}^N$, and $\tilde{\mathfrak{S}}^N$ are tight in $\bD$, from Lemma~\ref{TCL-tight-F-2}. 
	 $$\big(\hat{\mathfrak F}^N,\hat{\mathfrak S}^N,\hat{\mathfrak S}^N_{1,0},\hat{\mathfrak S}^N_2,\overline{\mathfrak{F}}^N,\overline{\mathfrak{S}}^N,\tilde{\mathfrak{S}}^N\big)$$ is tight in $\bD^7$. We can extract a subsequence denoted again $\big(\hat{\mathfrak F}^N,\hat{\mathfrak S}^N,\hat{\mathfrak S}^N_{1,0},\hat{\mathfrak S}^N_2,\overline{\mathfrak{F}}^N,\overline{\mathfrak{S}}^N,\tilde{\mathfrak{S}}^N\big)$ that converges to 
	 $\big(\hat{\mathfrak F},\hat{\mathfrak S},\hat{\mathfrak S}_{1,0},\hat{\mathfrak J},\overline{\mathfrak{F}},\overline{\mathfrak{S}},\overline{\mathfrak{S}} \big)$ in law in $\bD^7$. 
	By Lemma~\ref{TCL-lem-L1}, and the continuous mapping theorem, \eqref{TCL-lem-eq-L1} follows.
\end{proof}
Note that
 \begin{multline}
 \varPsi_1(\hat{\mathfrak F},\hat{\mathfrak S},\hat{\mathfrak S}_{1,0},\hat{\mathfrak J},\overline{\mathfrak{F}},\overline{\mathfrak{S}},\tilde{\mathfrak{S}})(t)=\hat{\mathfrak S}(t)-\hat{\mathfrak S}_{1,0}(t)-\hat{\mathfrak J}(t)\\
 \begin{aligned}
 &+\int_{0}^{t}\int_{s}^{t}\E\left[\gamma(t-s)\gamma(r-s)\exp\left(-\int_{s}^{t}\gamma(u-s)\overline{\mathfrak F}(u)du\right)\right]\hat{\mathfrak{F}}(r)\overline{\mathfrak{F}}(s)\overline{\mathfrak{S}}(s)drds\\
 &+\int_{0}^{t}\E\left[\gamma(t-s)\exp\left(-\int_{s}^{t}\gamma(r-s)\overline{\mathfrak F}(r)dr\right)\right]\left(\overline{\mathfrak F}(s)\hat{\mathfrak{J}}(s)-\overline{\mathfrak F}(s)\hat{\mathfrak{S}}(s)-\hat{\mathfrak F}(s)\tilde{\mathfrak{S}}(s)\right)ds,
 \end{aligned}
 \end{multline}
 and
 \begin{multline}
 \varPsi_2(\hat{\mathfrak F},\hat{\mathfrak S},\hat{\mathfrak M},\hat{\mathfrak J},\overline{\mathfrak{F}},\overline{\mathfrak{S}},\tilde{\mathfrak{S}})(t)\\
 \begin{aligned}
 &=\hat{\mathfrak F}(t)-\hat{\mathfrak{M}}(t)-\int_{0}^{t}\overline{\lambda}(t-s)\hat{\mathfrak S}(s)\overline{\mathfrak F}(s)ds-\int_{0}^{t}\overline{\lambda}(t-s)\hat{\mathfrak J}(s)\overline{\mathfrak F}(s)ds+\int_{0}^{t}\overline{\lambda}(t-s)\tilde{\mathfrak S}(s)\hat{\mathfrak F}(s)ds.
 \end{aligned}
 \end{multline}
 Combining Lemma~\ref{TCL-Lem-L3} with Lemma~\ref{TCL-lem-L2}, it follows that the pair $(\hat{\mathfrak S},\hat{\mathfrak F})$ satisfies the set of equations	\eqref{TCL-in-uni-1-1}-\eqref{TCL-in-uni-2-1} and from Lemma~\ref{TCL-exist} we conclude the proof of Theorem~\ref{TCL}.

 
\subsection{Alternative expressions of $(\hat{\mathfrak S}^N_2,\hat{\mathfrak F}^N_2)$ and their limits}\label{TCL-sec-NLG}

Note that the following expression follows from \eqref{TCL-lbda-decomp} and \eqref{TCL-gamma-decom} 

\begin{align*}
	\hat{\mathfrak S}^N_2(t) &=\hat{\mathfrak S}^N_{2,0}(t)+\hat{\mathfrak S}^N_{2,1}(t) +\hat{\mathfrak S}^N_{2,2}(t) \\
	& \quad + \int_{0}^{t}\E\left[\gamma_{1,1}(t-s)\exp\left(-\int_{s}^{t}\gamma_{1,1}(r-s)\overline{\mathfrak F}(r)dr\right)\right]\overline{\mathfrak F}(s)\hat{\mathfrak{S}}^N_2(s)ds,
\end{align*}
where 
\begin{align*}
\hat{\mathfrak{S}}^N_{2,0}(t)&:=\frac{1}{\sqrt N}\sum_{k=1}^{N}\left(\gamma_{k,0}(t)\mathds{1}_{P_{k}(0,t,\gamma_{k,0},\overline{\mathfrak F})=0}-\E\left[\gamma_{1,0}(t)\exp\left(-\int_{0}^{t}\gamma_{1,0}(s)\overline{\mathfrak F}(s)ds\right)\right]\right), \\
\hat{\mathfrak{S}}^N_{2,1}(t)&:=\frac{1}{\sqrt N}\sum_{k=1}^{N}\int_{0}^{t}\int_{\bD^2}\int_{0}^{\Upsilon_k(s-)}\left[\gamma(t-s)\mathds{1}_{P_{k}(s,t,\gamma,\overline{\mathfrak F})=0}\right.\\ &\hspace{4cm}\left.-\int_{\bD}\tilde\gamma(t-s)\exp\left(-\int_{s}^{t}\tilde\gamma(r-s)\overline{\mathfrak F}(r)dr\right)\mu(d\tilde\gamma)\right]Q_k(ds,d\lambda,d\gamma,du)\,,\\
 \hat{\mathfrak{S}}^N_{2,2}(t) &:=\frac{1}{\sqrt N}\sum_{k=1}^{N}\int_{0}^{t}\int_{\bD^2}\int_{0}^{\Upsilon_k(s-)}\int_{\bD}\tilde\gamma(t-s)\exp\left(-\int_{s}^{t}\tilde\gamma(r-s)\overline{\mathfrak F}(r)dr\right)\mu(d\tilde\gamma)\overline Q_k(ds,d\lambda,d\gamma,du)\,. 
\end{align*}
Similarly, we have  
\begin{align*}
	\hat{\mathfrak F}^N_2(t)&= 	\hat{\mathfrak F}^N_{2,0}(t) + \hat{\mathfrak F}^N_{2,1}(t) +\int_{0}^{t}\overline{\lambda}(t-s)\overline{\mathfrak F}(s)\hat{\mathfrak{S}}^N_2(s)ds\,,
\end{align*}
where
\begin{align*}
\hat{\mathfrak F}^N_{2,0}(t) &:=\frac{1}{\sqrt N}\sum_{k=1}^{N}\left(\lambda_{k,0}(t)-\E\left[\lambda_{1,0}(t)\right]\right)\,, \\
\hat{\mathfrak F}^N_{2,1}(t) &:= \frac{1}{\sqrt N}\sum_{k=1}^{N}\int_{0}^{t}\int_{\bD^2}\int_{0}^{\Upsilon_k(s-)}\lambda(t-s)\overline Q_k(ds,d\lambda,d\gamma,du)\,.
\end{align*}

Let 
\begin{align*}
\hat{\mathfrak{S}}_{2,1}(t)&:=\int_{0}^{t}\int_{\bD^2}\int_{0}^{\Upsilon_1(s-)}\left[\gamma(t-s)\mathds{1}_{P_{1}(s,t,\gamma,\overline{\mathfrak F})=0}\right.\\ &\hspace{4cm}\left.-\int_{\bD}\tilde\gamma(t-s)\exp\left(-\int_{s}^{t}\tilde\gamma(r-s)\overline{\mathfrak F}(r)dr\right)\mu(d\tilde\gamma)\right]Q_1(ds,d\lambda,d\gamma,du)\,,\\
\hat{\mathfrak{S}}_{2,2}(t) &:=\int_{0}^{t}\int_{\bD^2}\int_{0}^{\Upsilon_1(s-)}\int_{\bD}\tilde\gamma(t-s)\exp\left(-\int_{s}^{t}\tilde\gamma(r-s)\overline{\mathfrak F}(r)dr\right)\mu(d\tilde\gamma)\overline Q_1(ds,d\lambda,d\gamma,du)\,,\\
\hat{\mathfrak{F}}_{2,1}(t)&:=\int_{0}^{t}\int_{D}\int_{0}^{\Upsilon_1(s-)}\lambda(t-s)\overline Q_1(ds,d\lambda,du).
\end{align*}

Recall that $\left(\gamma_{k,A_k(\cdot)}(\varsigma_k(\cdot)),\lambda_{k,A_k(\cdot)}(\varsigma_k(\cdot))\right)_k$ are i.i.d. $\bD^2$-valued random variables and that each component satisfies Assumption~\ref{TCL-AS-lambda} and $(\overline{Q}_k)_k$ are also i.i.d. (Note that the processes $P_k$ depends only on $Q_k$ and $\Upsilon_k$ depends only on $\gamma_{k,A_k}$.)  
As a result, the processes $\hat{\mathfrak{S}}^N_{2,0}(t)$, $\hat{\mathfrak{S}}^N_{2,1}(t)$, $\hat{\mathfrak{S}}^N_{2,2}(t)$, $\hat{\mathfrak F}^N_{2,0}(t) $ and $\hat{\mathfrak F}^N_{2,1}(t) $ can all be regarded as sums of i.i.d. random processes in $\bD$. Thus, applying the central limit theorem in $\bD$ to each term (see Theorem~2 in \cite{hahn1978central}), we obtain the convergence of these processes: 
$\big(\hat{\mathfrak S}^N_{2,0}, \hat{\mathfrak S}^N_{2,1}, \hat{\mathfrak S}^N_{2,2}, \hat{\mathfrak F}^N_{2,0}, \hat{\mathfrak F}^N_{2,1}\big) \Rightarrow 
\big(\hat{\mathfrak{J}}_{0,1}, W^\gamma_{1}, W^\gamma_{2}, \hat{\mathfrak{M}}_{0,1}, W^\lambda\big) 
$ in $\bD^5$, where the limits are Gaussian processes as defined below: $\hat{\mathfrak{J}}_{0,1}$ has covariance function, for $t, t'\ge0$, 
\[\Cov(\hat{\mathfrak{J}}_{0,1}(t),\hat{\mathfrak{J}}_{0,1}(t'))=\Cov\left(\gamma_{1,0}(t)\mathds{1}_{P_{1}(0,t,\gamma_{1,0},\overline{\mathfrak F})=0},\gamma_{1,0}(t')\mathds{1}_{P_{1}(0,t',\gamma_{1,0},\overline{\mathfrak F})=0}\right),
\]
and $\hat{\mathfrak{M}}_{0,1} $ has covariance function, for $t, t'\ge0$, 
\[
		\Cov(\hat{\mathfrak{M}}_{0,1}(t),\hat{\mathfrak{M}}_{0,1}(t'))=\Cov\left(\lambda_{1,0}(t),\lambda_{1,0}(t')\right)\,,
\]
and the covariances between any two processes can be obtained, 
\[
		\Cov(\hat{\mathfrak{J}}_{0,1}(t),W^\gamma_{1}(t'))=\Cov\left(\gamma_{1,0}(t)\mathds{1}_{P_{1}(0,t,\gamma_{1,0},\overline{\mathfrak F})=0},\hat{\mathfrak{S}}_{2,1}(t')\right)
		\]
and so on. 
 Then by the continuous mapping theorem, we obtain the convergence of $\hat{\mathfrak S}^N_{2}$, and then given its convergence and the convergence of 
$\big( \hat{\mathfrak F}^N_{2,0}, \hat{\mathfrak F}^N_{2,1}\big)$, we obtain the convergence of $\hat{\mathfrak F}^N_{2}$. 
This leads to the following lemma. 

%
\begin{lemma}\label{TCL-cv-F2-G2'}
	Under Assumption~\ref{TCL-AS-lambda}, as $N\to+\infty$,
	\begin{equation*}
	\big(\hat{\mathfrak S}^N_2,\hat{\mathfrak F}^N_2\big)\Rightarrow\big(\hat{\mathfrak J},\hat{\mathfrak M}\big) \quad \text{ in } \quad \bD^2.
	\end{equation*}
	where $\big(\hat{\mathfrak J},\hat{\mathfrak M}\big)$ is a centered continuous 2-dimensional Gaussian process 
	given in Definition \ref{TCL-def-1}. More precisely, $\big(\hat{\mathfrak J},\hat{\mathfrak M}\big)$ satisfies the following system of equations: 
	\begin{numcases}{}
	\hat{\mathfrak{J}}(t)=
	\hat{\mathfrak{J}}_{0,1}(t)+W^\gamma_{1}(t)+W^\gamma_{2}(t)
	+\int_{0}^{t}\E\left[\gamma_{1,1}(t-s)\exp\left(-\int_{s}^{t}\gamma_{1,1}(r-s)\overline{\mathfrak F}(r)dr\right)\right]\overline{\mathfrak F}(s)\hat{\mathfrak{J}}(s)ds\,,\nonumber\\\nonumber\\
	\hat{\mathfrak{M}}(t)=
	\hat{\mathfrak{M}}_{0,1}(t)+W^\lambda(t)
	+\int_{0}^{t}\overline\lambda(t-s)\overline{\mathfrak F}(s)\hat{\mathfrak{J}}(s)ds.\nonumber
	\end{numcases}
	where 
	$\big(\hat{\mathfrak J}_{0,1},W^\gamma_{1},W^\gamma_{2},\hat{\mathfrak M}_{0,1},W^\lambda\big)$ 
	is a centered continuous 5-dimensional Gaussian process whose covariance functions as described above.
\end{lemma}

\bigskip

\section{Some preliminary results}\label{TCL-US-R}
To establish the tightness of $\mathfrak{F}^N,\,\mathfrak{S}^N$ we need a $\bC$-tightness criterion, a new result on stochastic integrals with respect to Poisson random measures,  the moment estimates,  an approximation of the original model by a quarantine model, and and an estimate on the pair $(\lambda,\gamma)$, which are given in the next five subsections. 

\subsection{Tightness criterion}
We recall the following theorem used to prove of $\bC$-tightness in $\bD$ (see Theorem $3.21$ in \cite[page 350]{jacod2013limit}). 


\begin{theorem}\label{th-tight}
	Let $(X^N)_N$ be a sequence of r.v. taking values in $\bD$. It is $\bC$-tight in $\bD$, if 
	\begin{itemize}
		\item[(i)]for any $T>0,\,\epsilon>0$ there exist $N_0\in\mathbb{N}$ and $C>0$ 
		such that
		\begin{equation}
		N\geq N_0\implies\mathbb{P}\left(\sup_{0\leq s\leq T}|X^N_s|>C\right)\leq\epsilon.
		\end{equation}
		\item[(ii)]for any $T>0,\,\epsilon>0,\,\theta>0$ there exist $N_0\in\mathbb{N}$ and $\delta>0$ such that
		\begin{equation}
		N\geq N_0\implies \mathbb{P}\left(w_T(X^N,\delta)\geq\theta\right)\leq\epsilon.
		\end{equation}
	\end{itemize}
	where 
	\[w_T(\alpha,\delta)=\sup_{0\leq s<t< T,|t-s|\leq\delta}|\alpha(t)-\alpha(s)|.\]
\end{theorem}
Since we work with processes in $\bD$, we will simply write $\bC$-tightness below for brevity. In fact, $\bC$-tightness means that the limit of subsequences are continuous.
We also recall the following result from \cite[Lemma~$3.1$]{pang2022-CLT-functional}.
\begin{lemma}\label{TCL-Lem-20}
	Let $\{X^N\}_{N\ge1}$ be a sequence of random elements in $\bD$ such that $X^N(0)=0$.
	If for all $T>0$, $\ep>0$, as $\delta\to0$,
	\begin{align*}
	\limsup_{N\to\infty}\sup_{0\le t\le T}\frac{1}{\delta}\P\bigg(\sup_{0\le u\le \delta}|X^N(t+u)-X^N(t)|>\ep\bigg)\to0,
	\end{align*}
	then the sequence $X^N$ is $\bC$-tight.
\end{lemma} 
\subsection{A property of stochastic integrals with respect to Poisson random measures}
\begin{theorem}\label{TCL-th-1}
	Let $(E,\mathcal{B}(E),\nu)$ be a $\sigma-$finite measured space and let $Q$ be a Poisson random measure on $\R_+\times E$ of intensity $ds\nu(du)$ and \,$(\mathcal{F}_t)_t$ a filtration which such that for all $t\geq0,\,Q|_{[0,t]\times E}$ is $\mathcal{F}_t-$measurable, and for $0\leq s<t,\,\mathcal{F}_s$ and $Q|_{]s,t]\times E}$ are independent. Let $h\,:\,\R_+\times E\to\R,$ be a predictable process such that for all $t\in\R_+$,
	\begin{equation*}
	\E\left[\int_{0}^{t}\int_E|h(s,u)|\nu(du)ds\right]<\infty.
	\end{equation*}
	Let $f\,:\,\R_+\times E\times\mathcal{M}_F(\R_+\times E)\to\R,$ be  a bounded and measurable deterministic function. Then
	\begin{equation*}
	\E\left[\int_{[0,t]\times E}h(s,u)f(s,u,Q|_{]s,t]\times E})Q(ds,du)\right]=\E\left[\int_{0}^{t}\int_E h(s,u)\overline{f}(s,u,t)\nu(du)ds\right],
	\end{equation*}
	where 
	\begin{equation*}
	\overline{f}(s,u,t)=\E\left[f(s,u,Q|_{]s,t]\times E})\right].
	\end{equation*}
\end{theorem}
\begin{proof}
	Let $(s_i,u_i)_i$ be an arbitrary ordering be the atoms of the measure $Q$.
	We note that 
	\begin{equation*}
	\int_{[0,t]\times E}h(s,u)f(s,u,Q|_{]s,t]\times E})Q(ds,du)=\sum_{i\geq1}h(s_i,u_i)f(s_i,u_i,Q|_{]s_i,t]\times E})\mathds{1}_{s_i\leq t}\,.
	\end{equation*}   
	As 
	\begin{equation*}
	\E\left[\int_{0}^{t}\int_E|h(s,u)|\nu(du)ds\right]<\infty,
	\end{equation*}
	by Fubini's theorem 
			\begin{equation*}
			\E\left[\int_{[0,t]\times E}h(s,u)f(s,u,Q|_{]s,t]\times E})Q(ds,du)\right]=\sum_{i\geq1}\E\left[h(s_i,u_i)\E\left[f(s_i,u_i,Q|_{]s_i,t]\times E})\big|\mathcal{F}_{s_i}\right]\mathds{1}_{s_i\leq t}\right],
		\end{equation*}
	and as $Q|_{]s_i,t]\times E}$ and $\mathcal{F}_{s_i}$ are independent, 
	\[\E\left[f(s_i,u_i,Q|_{]s_i,t]\times E})\big|\mathcal{F}_{s_i}\right]=\overline{f}(s_i,u_i,t),\]
	it follows that,
	\begin{align*}
	\E\left[\int_{[0,t]\times E}h(s,u)f(s,u,Q|_{]s,t]\times E})Q(ds,du)\right]&=\sum_{i\geq1}\E\left[h(s_i,u_i)\overline{f}(s_i,u_i,t)\mathds{1}_{s_i\leq t}\right]\\
	&=\E\left[\int_{[0,t]\times E}h(s,u)\overline{f}(s,u,t)Q(ds,du)\right].
	\end{align*}
	 Consequently, as $h$ is a progressive process, by a classical result on Poisson random measures \cite[Theorem~6.2]{ccinlar2011probability}
	\begin{equation*}
	\E\left[\int_{[0,t]\times E}h(s,u)f(s,u,Q|_{]s,t]\times E})Q(ds,du)\right]=\E\left[\int_0^t\int_E h(s,u)\overline{f}(s,u,t)\nu(du)ds\right].
	\end{equation*} 
\end{proof}

\subsection{Moment Inequalities}

We recall the following Lemma from \cite{forien-Zotsa2022stochastic}.

\begin{lemma}\label{lem_inq}For $k\in\mathbb{N}$ and $T\geq0$, 
	\begin{equation}
	\mathbb{E}\left[\sup_{t\in[0,T]}\left|A^N_k(t)-A_k(t)\right|\right]\leq\int_{0}^{T}\mathbb{E}\Big[\left|\Upsilon^N_k(t)-\Upsilon_k(t)\right|\Big]dt=:\delta^N(T)\label{eqA}
	\end{equation}
	and 
	\begin{equation*}
	\mathbb{P}\Big((\varsigma_{k}^N(s))_{t\in[0,T]}\neq(\varsigma_{k}(t))_{t\in[0,T]}\Big)\leq T\delta^N(T).
	\end{equation*}
	Moreover, 
	\begin{equation}\delta^N(T)\leq\frac{\lambda_*}{\sqrt{N}}T\exp(2\lambda_*T).\label{eqdelta}\end{equation}
\end{lemma}
From Lemma~\ref{lem_inq} or \cite[Lemma~$6.3$]{forien-Zotsa2022stochastic}, we deduce the following Corollary. 
\begin{coro}\label{lem-c2}
	For $k\in\mathbb{N}$ and $T\geq0$, 
		\begin{equation}
	\E\left[\sup_{t\in[0,T]}\left|\gamma_{k,A_k^N(t)}(\varsigma^N_k(t))-\gamma_{k,A_k(t)}(\varsigma_k(t))\right|\right]\leq\frac{\lambda_*}{\sqrt{N}}T\exp(2\lambda_*T),\label{eqgam}
	\end{equation}
	\begin{equation}
	\E\left[\sup_{t\in[0,T]}\left|\lambda_{k,A_k^N(t)}(\varsigma^N_k(t))-\lambda_{k,A_k(t)}(\varsigma_k(t))\right|\right]\leq\frac{\lambda_*^{2}}{\sqrt{N}}T\exp(2\lambda_*T),\label{eqlam}
	\end{equation}
	and
	\begin{equation}
	\E\left[\sup_{t\in[0,T]}\left|\mathds{1}_{\varsigma^N_k(t)<\eta_{k,A_k^N(t)}}-\mathds{1}_{\varsigma_k(t)<\eta_{k,A_k(t)}}\right|\right]\leq\frac{\lambda_*}{\sqrt{N}}T\exp(2\lambda^*T).\label{eqind}
	\end{equation}
\end{coro}
From \cite[Remark~$6.2$]{forien-Zotsa2022stochastic} we deduce the following Corollary. 
\begin{coro}\label{TCL-cor-1}
	For $k\in\mathbb{N}$ and $t\geq0$,
	\begin{align*}
	&\E\left[\left|\overline{\mathfrak F}^N(t)-\overline{\mathfrak{F}}(t)\right|\right]\leq \frac{\lambda_*}{\sqrt{N}}\left(1+\lambda_*t\exp(2\lambda_*t)\right),\qquad \E\left[\left|\overline{\mathfrak S}^N(t)-\overline{\mathfrak{S}}(t)\right|\right]\leq \frac{1}{\sqrt{N}}\left(1+\lambda_*t\exp(2\lambda_*t)\right)\\&\text{ and  } \quad 
	\E\left[\left|\Upsilon^N_k(t)-\Upsilon_k(t)\right|\right]\leq \frac{\lambda_*}{\sqrt{N}}\left(1+2\lambda_*t\exp(2\lambda_*t)\right).
	\end{align*}
\end{coro}
By exchangeability we have the following Corollary. 
\begin{coro}
For  $t\ge 0$, 
	\begin{align}\label{fluc-1}
	\E\left[\left|\overline\Upsilon^N(t)-\tilde\Upsilon^N(t)\right|\right]\leq\frac{\lambda_*}{\sqrt{N}}\left(1+2\lambda_*t\exp(2\lambda_*t)\right), 
	\end{align}	
	where 
	\[\tilde\Upsilon^N(t)=\frac{1}{N}\sum_{k=1}^{N}\Upsilon_k(t).\]
\end{coro}
Now we establish similar inequalities as in Corollary~\ref{TCL-cor-1} for higher moments. 
Let
\begin{equation*}
\chi_N^{(k)}(t):=\mathbb{P}\Big((\varsigma_{k'}^N(s))_{s\in[0,t]}\neq(\varsigma_{k'}(s))_{s\in[0,t]},\,\forall k'=1,\cdots, k\Big)\,.
\end{equation*}
We establish the following Proposition.
\begin{prop}\label{prop-c1}
	For all $N\geq k$, and $t\in[0,T],$ there are positives constants $C_{k,T}$ and $C_{k,T}'$ depending on $k$ and $T$, such that 
	\begin{equation*}\chi_N^{(k)}(t)\leq C_{k,T}N^{-k/2},\quad \text{ and } \quad \xi_N^{(k)}(t):=\E\left[\left|\overline{\mathfrak F}^N(t)-\overline{\mathfrak F}(t)\right|^k\right]\leq C_{k,T}'(t)N^{-k/2}.\end{equation*}
\end{prop}
\begin{proof}
	We adapt the proof of \cite[Proposition~$3.1$]{chevallier_fluctuations_2017}. Let 
	\[\Delta^N_k(t):=\int_{0}^{t}\int_{\Upsilon_k^N(s^-)\wedge\Upsilon_k(s^-)}^{\Upsilon_k^N(s^-)\vee\Upsilon_k(s^-)}Q_k(ds,du).\]
	Observe that for all $k',\,(\varsigma_{k'}^N(s))_{s\in[0,t]}=(\varsigma_{k'}(s))_{s\in[0,t]}$ if and only if $\Delta^N_{k'}(t)=0$.
	As $\Delta^N_{k'}$ takes  integer values, 
	\begin{equation*}
	\chi_N^{(k)}(t)\leq\E\left[\prod_{k'=1}^{k}\Delta^N_{k'}(t)\right].
	\end{equation*}
	Let us set, for all $k,p\in\mathbb{N}$ such that $N\geq k$,
	\begin{equation*}
	\varepsilon^{(k,p)}_N(t)=\E\left[\prod_{k'=1}^{k}\left(\Delta^N_{k'}(t)\right)^p\right].
	\end{equation*}
	We next show by induction on $k$ that 
	\begin{equation}
	\varepsilon^{(k,p)}_N(t)\leq C_{T,k,p} N^{-k/2}.
	\end{equation}
	From Lemma~\ref{lem_inq}, 
	\begin{equation}
	\varepsilon_N^{(1,1)}(t)\leq C_TN^{-1/2}\label{eqc2}.
	\end{equation}
	Note that for all $p\leq q$,
	\[\varepsilon^{(k,p)}_N(t)\leq\varepsilon^{(k,q)}_N(t),\]
	because the counting process $\Delta^N_k$ takes values in $\mathbb{N}$.
	
	As in \cite[Proposition~$3.1$]{chevallier_fluctuations_2017} we have for all $p\in\mathbb{N}$,
	\begin{equation}\label{CLT-1}
	\varepsilon^{(1,p)}_N(t)=\E\left[\left(\Delta^N_{1}(t)\right)^p\right] \leq C_{T,p} N^{-1/2}.
	\end{equation}
	Indeed, noting that, as the process $(\Delta_1^N(t))_t$ jumps at each time from $\Delta_1^N(t^-)$ to $\Delta_1^N(t^-)+1,$ the process $((\Delta_1^N(t))^p)_t$ jumps from $(\Delta_1^N(t^-))^p$ to $(\Delta_1^N(t^-)+1)^p$. Consequently, from the fact that 
	\[(\Delta_1^N(t^-)+1)^p-(\Delta_1^N(t^-))^p=\sum_{p'=0}^{p-1}\dbinom{p}{p'}\left(\Delta_1^N(t^-)\right)^{p'},\]
	it follows that 
	\[\left(\Delta_1^N(t)\right)^p=\sum_{p'=0}^{p-1}\dbinom{p}{p'}\int_0^t \left(\Delta_1^N(s^-)\right)^{p'}\Delta_1^N(ds).\]
	Moreover, as $\left(\Delta_1^N(s^-)\right)^{p'}\leq\left(\Delta_1^N(s^-)\right)^{p }$ for $p'\leq p$, we deduce that
	\begin{align*}
	\varepsilon^{(1,p)}_N(t)=\E\left[\left(\Delta_1^N(t)\right)^p\right]&\leq\E\left[\int_0^t \Delta_1^N(ds)\right]+2^p\E\left[\int_0^t \left(\Delta_1^N(s^-)\right)^{p}\Delta_1^N(ds)\right]\\
	&=\varepsilon_N^{(1,1)}(t)+2^p\int_{0}^{t}\E\left[\left(\Delta_1^N(s)\right)^{p}|\Upsilon_1(s)-\Upsilon_1^N(s)|\right]ds\\
	&\leq C_{T,p}N^{-1/2}+2^{p}\lambda_*\int_{0}^{t}\varepsilon^{(1,p)}_N(s)ds,
	\end{align*}
	where the last inequality comes from \eqref{eqc2} and the fact that $|\Upsilon_1(s)-\Upsilon_1^N(s)|\leq\lambda^*.$ The conclusion~\eqref{CLT-1} follows by Gronwall's Lemma.
	
	When $k\geq2$ and $p\geq1$, as above,  noting that
	\begin{equation*}
	\prod_{i=1}^{k}\left(\Delta_i^N(t^-)\right)^p=\sum_{j=1}^{k}\sum_{p'=0}^{p-1}\dbinom{p}{p'}\int_0^t \left( \prod_{i\neq j,i=1}^k\left(\Delta_i^N(s^-)\right)^{p}\right)\left(\Delta_j^N(s^-)\right)^{p'}\Delta_j^N(ds)\,,
	\end{equation*} 
	almost surely, and using exchangeability of the processes $(\Delta_i^N)_i$ and the fact that the integrand is predictable, it follows that
	\begin{align*}
	\varepsilon^{(k,p)}_N(t)&=\sum_{j=1}^{k}\sum_{p'=0}^{p-1}\dbinom{p}{p'}\E\left[\int_0^t \prod_{i\neq j,i=1}^k\left(\Delta_i^N(s^-)\right)^{p}\left(\Delta_j^N(s^-)\right)^{p'}\Delta_j^N(ds)\right]\\
	&=k\sum_{p'=0}^{p-1}\dbinom{p}{p'}\int_0^t\E\left[\left(\Delta_1^N(s)\right)^{p'} \prod_{i=2}^k\left(\Delta_i^N(s)\right)^{p}|\Upsilon_1(s)-\Upsilon_1^N(s)|\right]ds\\
	&\leq k\int_{0}^{t} \Bigg(\E\left[\prod_{i=2}^k\left(\Delta_i^N(s)\right)^{p}|\Upsilon_1(s)-\Upsilon_1^N(s)|\right]+2^p\lambda_*\varepsilon^{(k,p)}_N(s) \Bigg) \, ds.
	\end{align*}
	However, as
	\begin{align*}
	|\Upsilon_1(s)-\Upsilon_1^N(s)|&\leq|\overline{\mathfrak F}^N(s)-\overline{\mathfrak F}(s)|+\lambda_\ast|\gamma_{1,A^N_1(s)}(\varsigma_1^N(s))-\gamma_{1,A_1(s)}(\varsigma_1(s))|\\
	&\leq|\overline{\mathfrak F}^N(s)-\overline{\mathfrak F}(s)|+\lambda_\ast\mathds{1}_{\varsigma_{1}^N(s)\neq\varsigma_{1}(s)\text{ or }A^N_1(s)\neq A_1(s)}\\
	&\leq|\overline{\mathfrak F}^N(s)-\overline{\mathfrak F}(s)|+\lambda_\ast\Delta_1^N(s),
	\end{align*}
	it follows that,
	\begin{equation*}
	\varepsilon^{(k,p)}_N(t)\leq 
	k\int_{0}^{t}\left(M^N(s)+2^{p+1}\lambda_*\varepsilon^{(k,p)}_N(s)\right)ds,
	\end{equation*}
	with $M^N(s)=\E\left[\prod_{i=2}^k\left(\Delta_i^N(s)\right)^{p}|\overline{\mathfrak F}^N(s)-\overline{\mathfrak F}(s)|\right]$.
	Using exchangeability we can replace each term $\left(\Delta_i^N(s)\right)^{p}$ in the expression of $M^N$ by the following sum
	\[\frac{1}{\lfloor\frac{N}{k}\rfloor}\sum_{j=(i-1)\lfloor\frac{N}{k}\rfloor+1}^{i\lfloor\frac{N}{k}\rfloor}(\Delta^N_j(s))^p\]
	without changing the value of $M^N,$ since the sums are taken on disjoint indices. Then,
	\begin{align*}
	M^N(s)&=\E\left[\prod_{i=2}^k\left(\frac{1}{\lfloor\frac{N}{k}\rfloor}\sum_{j=(i-1)\lfloor\frac{N}{k}\rfloor+1}^{i\lfloor\frac{N}{k}\rfloor}(\Delta^N_j(s))^p\right)|\overline{\mathfrak F}^N(s)-\overline{\mathfrak F}(s)|\right]\nonumber\\
	&\leq\left(\prod_{i=2}^k\E\left[\left(\frac{1}{\lfloor\frac{N}{k}\rfloor}\sum_{j=(i-1)\lfloor\frac{N}{k}\rfloor+1}^{i\lfloor\frac{N}{k}\rfloor}(\Delta^N_j(s))^p\right)^k\right]\right)^{1/k}\left(\E\left[|\overline{\mathfrak F}^N(s)-\overline{\mathfrak F}(s)|^k\right]\right)^{1/k}.
	\end{align*}  
	
	Consequently, using Young's inequality, it follows that
	\begin{equation}
	\varepsilon^{(k,p)}_N(t)\lesssim \int_{0}^{t} \bigg( \frac{k-1}{k}E_{N,k,p}(s)+\frac{1}{k}\xi_N^{(k)}(s)+\varepsilon_N^{(k,p)}(s) \bigg) ds\,,
	\end{equation}
	where
	\begin{align*}
	E_{N,k,p}(s)&:=\E\left[\left(\frac{1}{\lfloor\frac{N}{k}\rfloor}\sum_{j=1}^{\lfloor\frac{N}{k}\rfloor}(\Delta^N_j(s))^p\right)^k\right]\\
	&\leq C_k\left(\sum_{k'=1}^{k-1}N^{k'-k}\varepsilon_N^{(k',kp)}(s)+\varepsilon_N^{(k,p)}(s)\right),
	\end{align*}
	and 
	\begin{align*}
	\xi^{(k)}_N(t)&:=\E\left[\left|\overline{\mathfrak F}^N(t)-\overline{\mathfrak F}(t)\right|^k\right]\\
	&\leq2^{k-1}\bigg(\E\bigg[\bigg(\frac{1}{N}\sum_{i=1}^{N}\lambda_{i,A^N_i(t)}(\varsigma_i^N(t))-\lambda_{i,A_i(t)}(\varsigma_i(t))\bigg)^2\bigg]+\E\bigg[\bigg(\frac{1}{N}\sum_{i=1}^{N}\lambda_{i,A_i(t)}(\varsigma_i(t))-\overline{\mathfrak{F}}(t)\bigg)^2\bigg]\bigg)\\
	&\leq2^{k-1}\bigg(\E\bigg[\bigg(\frac{1}{N}\sum_{i=1}^{N}\Delta_{i}^N(t)\bigg)^2\bigg]+\E\bigg[\bigg(\frac{1}{N}\sum_{i=1}^{N}\lambda_{i,A_i(t)}(\varsigma_i(t))-\overline{\mathfrak{F}}(t)\bigg)^2\bigg]\bigg)\\
	&\leq\sum_{k'=1}^{k-1}N^{k'-k}\varepsilon_N^{(k',k)}(t)+\varepsilon_N^{(k,p)}(t)+C_kN^{-k/2},
	\end{align*}
	We refer to \cite[Proposition~$3.1$]{chevallier_fluctuations_2017} for more details.
	The conclusion of the proof of this Proposition follows by induction and Gronwall's Lemma.
\end{proof}
The proof of the following Corollary is similar to the proof of Proposition~\ref{prop-c1} replacing $\lambda_{i,j}$ by $\gamma_{i,j}$. 
\begin{coro}\label{inr}For all $\ell,\,k,\,T>0$, There exists $C_{T,k}>0$, such that for all $t\in[0,T],\,\forall N\geq k,$
	\begin{equation}\label{fluc-2}
	\E\left[\left|\overline{\mathfrak S}^N(t)-\overline{\mathfrak S}(t)\right|^k\right]\leq C_{T,k}N^{-k/2},\text{ and }\E\left[\left|\Upsilon^N_\ell(t)-\Upsilon_\ell(t)\right|^k\right]\leq C_{T,k}N^{-k/2}.
	\end{equation}
%
\end{coro}
Let us recall the following result, which is a case of the Borel-Cantelli Lemma.
\begin{prop}\label{ins}
	Let $(X^N)_N$ be a family of reals random variables and $X$ be a random variable. If for each $\epsilon>0$,
	\begin{equation*}
	\sum_{N\geq1}\mathbb{P}\left(|X^N-X|\geq\epsilon\right)<\infty,
	\end{equation*} 
	then $(X^N)_N$ converges almost surely to $X$.  
\end{prop}
As a consequence of Proposition~\ref{ins} and the Corollary~\ref{inr}, the following result follows.
\begin{coro}
	For each $t\geq0,\,\overline{\mathfrak F}^N(t)$ and $\overline{\mathfrak S}^N(t)$ converge almost surely respectively to $\overline{\mathfrak F}(t)$ and $\overline{\mathfrak S}(t)$. 
\end{coro}

\subsection{Quarantine model }\label{TCL-sub-qua}
Fix $k\in \mathbb{N}$, let us consider the model above where the  $k$-th individual is quarantined. Then we define
\begin{equation*}
\overline{\mathfrak{F}}^N_{(k)}(t)=\frac{1}{N}\sum_{\ell=1,\ell\neq k}^{N}\lambda_{\ell,A^N_{\ell,(k)}(t)}(\varsigma^N_{\ell,(k)}(t)),\text{ and }\overline{\mathfrak{S}}^N_{(k)}(t)=\frac{1}{N}\sum_{\ell=1,\ell\neq k}^{N}\gamma_{\ell,(k)}(t)(\varsigma^N_{\ell,(k)}(t)),
\end{equation*} 
where 
\[A^N_{\ell,(k)}(t)=\int_{0}^{t}\int_{0}^{+\infty}\mathds{1}_{\gamma_{\ell,A^N_{\ell,(k)}(s^-)}(\varsigma^N_{\ell,(k)}(s^-))\overline{\mathfrak{F}}_{(k)}^N(s^-)>u}Q_\ell(ds,du),\,\ell\neq k,\]
and we recall that $Q_k$ is a standard Poisson random measure on $\R^2_+$.
In the original model, the number of individuals infected by the $k$-th individual  can be described by 
\[\sum_{\ell=1,\ell\neq k}^N\int_{0}^{t}\int_{0}^{+\infty}\mathds{1}_{\frac{1}{N}\lambda_{k,A^N_{k}(s^-)}(\varsigma^N_{k}(s^-))\gamma_{\ell,A^N_{\ell}(s^-)}(\varsigma^N_{\ell}(s^-))>u}Q_\ell(ds,du).\]

Since $\lambda_{k,A^N_{k}(s^-)}(\varsigma^N_{k}(s^-))\gamma_{\ell,A^N_{\ell}(s^-)}(\varsigma^N_{\ell}(s^-))\leq \lambda_\ast$, we have 
\[\sum_{\ell=1,\ell\neq k}^N\int_{0}^{t}\int_{0}^{+\infty}\mathds{1}_{\frac{1}{N}\lambda_{k,A^N_{k}(s^-)}(\varsigma^N_{k}(s^-))\gamma_{\ell,A^N_{\ell}(s^-)}(\varsigma^N_{\ell}(s^-))>u}Q_\ell(ds,du)\leq\int_{0}^{t}\int_{0}^{\lambda_\ast}Q'_k(ds,du),\]
where $Q'_k$ is a standard Poisson random measure on $\R^2_+$ independent of $Q_k$.

So we can bound the number of infected descendants of the individual $k$,  by a pure birth process with birth rate $\lambda_\ast$, denoted by $Y(t-\tau^N_{k,1})$ (where $\tau^N_{k,1}$ denotes the first time of infection of the individual $k$ in the full model). Note that $Y(t-\tau^N_{k,1}))\le Y(t)$, and that $Y(t)$ follows a geometric distribution with parameter $\exp(-\lambda_\ast t)$ and is independent of $Q_k$. As the random variable $Y$ bounds the number of individuals who do not have the same state between the two models, as a result, we have
\begin{equation*}
\left|\overline{\mathfrak F}^N(t)-\overline{\mathfrak{F}}^N_{(k)}(t)\right|\leq\frac{\lambda_\ast}{N} Y(t) \quad \text{ and } \quad \left|\overline{\mathfrak S}^N(t)-\overline{\mathfrak{S}}^N_{(k)}(t)\right|\leq \frac{Y(t)}{N}.
\end{equation*}
Similarly we can define a quarantine model $(\overline{\mathfrak{F}}^N_{(k,\ell)},\overline{\mathfrak{S}}^N_{(k,\ell)})$ where the $k$-th and $\ell$-th individuals are quarantined such that 
\begin{equation}\label{TCL-eq1}
\left|\overline{\mathfrak F}^N(t)-\overline{\mathfrak{F}}^N_{(k,\ell)}(t)\right|\leq\frac{\lambda_\ast}{N} \tilde Y(t)\quad \text{ and } \quad \left|\overline{\mathfrak S}^N(t)-\overline{\mathfrak{S}}^N_{(k,\ell)}(t)\right|\leq \frac{\tilde Y(t)}{N},
\end{equation}
where $\tilde Y(t)$ follows a geometric distribution with parameter $\exp(-2\lambda_\ast t)$ and is independent of $Q_k$ and $Q_\ell$.	

%
\subsection{Useful inequalities for $\lambda$ and $\gamma$}
We establish the following Lemma where $\lambda$ and $\gamma$ are given in Assumption \ref{TCL-AS-lambda-1}-\ref{TCL-AS-lambda-2}.
\begin{lemma} \label{TCL-lem-barlambda-inc-bound}
	For $t\ge s\ge 0$, with $\alpha>1/2$ as in Assumption \ref{TCL-AS-lambda-1}-\ref{TCL-AS-lambda-2}, 
	\begin{align*}
	\big|\lambda(t) - \lambda(s) \big| \le (t-s)^\alpha + \lambda_*  \sum_{j=1}^{\ell-1} \bone_{s < \xi^j \le t}\,,
	\end{align*}
	\begin{align*}
	\big|\gamma(t) - \gamma(s) \big| \le (t-s)^\alpha + \sum_{j=1}^{\ell-1} \bone_{s < \zeta^j \le t}\,,
	\end{align*}
	and
	\begin{align*}
	|\bar\lambda(t) -\bar \lambda(s)| \le (t-s)^\alpha + \lambda^* \sum_{j=1}^{\ell-1} (F_j(t) - F_j(s))\,. 
	\end{align*}
\end{lemma}
\begin{proof}
	We have 
	\begin{align*}
	\lambda(t) - \lambda(s) =\sum_{j=1}^\ell \big( \lambda^j(t) - \lambda^j(s) \big) \bone_{\xi^{j-1} \le s,\, t < \xi^j}
	+  \big(  \lambda(t) - \lambda(s)\big) \sum_{j=1}^{\ell-1} \bone_{s < \xi^j \le t}\,. 
	\end{align*}
	Thus the statement follows from Assumption \ref{TCL-AS-lambda-2}.
\end{proof}

\section{Characterization of the limit of converging subsequences of $(\hat{\mathfrak F}^N,\hat{\mathfrak S}_{1,0}^N)$ }\label{TCL-sec-G10}
The aim of this section is to prove Lemma~\ref{TCL-lem-L2}.
We recall that 
\begin{equation*}
\hat{\mathfrak S}_{1,0}^N(t)=\frac{1}{\sqrt{N}}\sum_{k=1}^{N}\gamma_{k,0}(t)\left(\mathds{1}_{P_{k}(0,t,\gamma_{k,0},\overline{\mathfrak F}^N)=0}-\mathds{1}_{P_{k}(0,t,\gamma_{k,0},\overline{\mathfrak F})=0}\right).
\end{equation*}
Let 
\begin{equation}\label{TCL-eqr-1}
\Xi^N_{1,0}(t)=\frac{1}{\sqrt{N}}\sum_{k=1}^{N}\chi^N_{k}(t),
\end{equation}
where 
\begin{align*}
\chi^N_k(t)&=\gamma_{k,0}(t)\bigg(\mathds{1}_{P_{k}(0,t,\gamma_{k,0},\overline{\mathfrak F}^N)=0}-\mathds{1}_{P_{k}(0,t,\gamma_{k,0},\overline{\mathfrak F})=0} \\&\hspace{3cm}-\exp\left(-\int_{0}^{t}\gamma_{k,0}(r)\overline{\mathfrak F}^N(r)dr\right)+\exp\left(-\int_{0}^{t}\gamma_{k,0}(r)\overline{\mathfrak F}(r)dr\right)\bigg).
\end{align*}
\begin{lemma}\label{TCL-Lem-eq-1}Under Assumption~\ref{TCL-AS-lambda},
	for all $t\ge 0$, 
	\begin{equation}\label{TCL-eq-1}
	\lim_{N\to\infty}\E\left[\left(\Xi^N_{1,0}(t)\right)^2\right]=0.
	\end{equation}
\end{lemma}
\begin{proof}
	By exchangeability, we have 
	\begin{equation}
	\E\left[\left(\Xi^N_{1,0}(t)\right)^2\right]=\E\left[\left(\chi^N_1(t)\right)^2\right]+(N-1)\E\left[\chi^N_1(t)\chi^N_2(t)\right].
	\end{equation} 
	From Theorem~\ref{thm-FLLN}, as $N\to\infty,\,\overline{\mathfrak F}^N\Rightarrow\overline{\mathfrak F},$ in $\bD$, consequently as $N\to\infty,\,\chi^N_1(t)\to0$ in probability and as $|\chi^N_{1}(t)|\leq2$, it follows that, as $N\to\infty,$
	\begin{equation}
	\E\left[\left(\chi^N_1(t)\right)^2\right]\to0.
	\end{equation}
	To obtain \eqref{TCL-eq-1} it remains to show that 
	\begin{equation}\label{TCL-A-n-0}
	\left|\E\left[\chi^N_1(t)\chi^N_2(t)\right]\right|=o\left(\frac{1}{N}\right).
	\end{equation} 
	We consider a quarantine model where the first and second individuals are in quarantine (see subsection~\ref{TCL-sub-qua}). We denote by $\overline{\mathfrak{F}}^N_{(1,2)}(t)$ the force of infection in the population at time $t$ in this model. As in \eqref{TCL-eq1}, for any $T\geq0$, there exists a geometric random variable $Y$ independent from $Q_1$ and $Q_2$ such that almost surely for all $t\in[0,T]$,
	\[\overline{\mathfrak F}^N_{(1,2)}(t)-\frac{Y}{N}\leq\overline{\mathfrak F}^N(t)\leq \overline{\mathfrak F}^N_{(1,2)}(t)+\frac{Y}{N}.\] 
	Therefore,
	\[P_{k}(0,t,\gamma_{k,0},\overline{\mathfrak F}_{(1,2)}^N-\frac{Y}{N})\leq P_{k}(0,t,\gamma_{k,0},\overline{\mathfrak F}^N)\leq P_{k}(0,t,\gamma_{k,0},\overline{\mathfrak F}^N_{(1,2)}+\frac{Y}{N}).\] 
	Let 
	\[E^N_k=\left\{P_{k}(0,t,\gamma_{k,0},\overline{\mathfrak F}_{(1,2)}^N-\frac{Y}{N})=P_{k}(0,t,\gamma_{k,0},\overline{\mathfrak F}_{(1,2)}^N+\frac{Y}{N})\right\}.\]
	Then, for $k\in\{1,2\}$ as $Q_k$ is independent of $(Q_i,\gamma_{k,0},Y,\overline{\mathfrak F}^N_{(1,2)})$ for $k\neq i$, it follows by Markov's inequality  that 
	\begin{equation}\label{TCL-A-n-1}
	\mathbb{P}\left(\left(E^N_k\right)^c\big|\gamma_{k,0},\overline{\mathfrak F}^N_{(1,2)},Y,Q_i\right)\leq\frac{2tY}{N}.
	\end{equation} 
	Let 
	\begin{align*}
	\tilde\chi^N_k(t)&=\gamma_{k,0}(t)\left(\mathds{1}_{P_{k}(0,t,\gamma_{k,0},\overline{\mathfrak F}^N_{(1,2)})=0}-\mathds{1}_{P_{k}(0,t,\gamma_{k,0},\overline{\mathfrak F})=0}\right.\\&\hspace{3cm}\left.-\exp\left(-\int_{0}^{t}\gamma_{k,0}(r)\overline{\mathfrak F}^N_{(1,2)}(r)dr\right)+\exp\left(-\int_{0}^{t}\gamma_{k,0}(r)\overline{\mathfrak F}(r)dr\right)\right).
	\end{align*}
	For $k\in\{1,2\}$ as $Q_k,\,\gamma_{k,0},$ and $\overline{\mathfrak F}^N_{(1,2)}$ are independent, it follows that
	\begin{equation*}
	\E\left[\tilde \chi^N_k(t)\big|\overline{\mathfrak F}^N_{(1,2)},\gamma_{k,0}\right]=0.
	\end{equation*} 
	Moreover, as $Q_1,\,\gamma_{1,0},\,Q_2,\,\gamma_{2,0}$ and $\overline{\mathfrak F}^N_{(1,2)}$ are independent, it follows that
	\begin{equation}\label{TCL-A-n-1'}
	\E\left[\tilde\chi^N_1(t)\tilde\chi^N_2(t)\right]=\E\left[\E\left[\tilde\chi^N_1(t)\big|\overline{\mathfrak F}_{(1,2)}^N\right]\E\left[\tilde\chi^N_2(t)\big|\overline{\mathfrak F}_{(1,2)}^N\right]\right]=0.
	\end{equation}
	We have 
	\begin{align}\label{TCL-A-n-3}
	\E\left[\chi^N_1(t)\chi^N_2(t)-\tilde\chi^N_1(t)\tilde\chi^N_2(t)\right]&=\E\left[\tilde\chi^N_1(t)\left(\chi^N_2(t)-\tilde\chi^N_2(t)\right)\right]+\E\left[\left(\chi^N_1(t)-\tilde\chi^N_1(t)\right)\tilde\chi^N_2(t)\right]\nonumber\\&\hspace*{3.5cm}+\E\left[\left(\chi^N_1(t)-\tilde\chi^N_1(t)\right)\left(\chi^N_2(t)-\tilde\chi^N_2(t)\right)\right]. 
	\end{align}
	However,
	\begin{equation}\label{TCL-A-n-2}
	\left|\chi^N_k(t)-\tilde\chi_k^N(t)\right|\leq\mathds{1}_{\left(E^N_k\right)^c}+\frac{Yt}{N}.
	\end{equation}
	Hence from \eqref{TCL-A-n-1} and \eqref{TCL-A-n-2}, 
	\begin{align*}
	\left|\E\left[\tilde\chi^N_1(t)\left(\chi^N_2(t)-\tilde\chi^N_2(t)\right)\right]\right|&\leq\E\left[\left|\tilde\chi^N_1(t)\right|\mathds{1}_{\left(E^N_2\right)^c}\right]+\frac{t}{N}\E\left[\left|\tilde\chi^N_1(t)\right|Y\right]\nonumber\\
	&=\E\left[\left|\tilde\chi^N_1(t)\right|\mathbb{P}\left(\left(E^N_2\right)^c\big|\gamma_{1,0},\overline{\mathfrak F}^N_{(1,2)},Y,Q_1\right)\right]+\frac{t}{N}\E\left[\left|\tilde\chi^N_1(t)\right|Y\right]\nonumber\\
	&\leq\frac{3t}{N}\E\left[\left|\tilde\chi^N_1(t)\right|Y\right]\nonumber\\
	&=\frac{3t}{N}\E\left[\left|\tilde\chi^N_1(t)\right|\right]\E\left[Y\right], 
	\end{align*}
	where we use the fact that $Y$ and $\tilde\chi^N_1$ are independent.
	
	From \eqref{TCL-cor-1} as $\overline{\mathfrak F}^N\to\overline{\mathfrak F}$, and $|\overline{\mathfrak F}^N-\overline{\mathfrak F}^N_{(1,2)}|\leq\frac{Y}{N}$, it follows that as $N\to\infty$,
	\begin{align*}
	\E\left[\left|\tilde\chi^N_1(t)\right|\right]&\leq\E\left[\int_0^t\int_{\gamma_{1,0}(r^-)\left(\overline{\mathfrak F}^N_{(1,2)}(r^-)\wedge\overline{\mathfrak{F}}(r^-)\right)}^{\gamma_{1,0}(r^-)\left(\overline{\mathfrak F}^N_{(1,2)}(r^-)\vee\overline{\mathfrak{F}}(r^-)\right)}Q_1(dr,du)\right]+\int_{0}^{t}\E\left[\gamma_{1,0}(r)\left|\overline{\mathfrak F}^N_{(1,2)}(r)-\overline{\mathfrak{F}}(r)\right|\right]dr\\
	&=2\int_{0}^{t}\E\left[\gamma_{1,0}(r)\left|\overline{\mathfrak F}^N_{(1,2)}(r)-\overline{\mathfrak{F}}(r)\right|\right]dr\to0.
	\end{align*}  
	Thus
	\begin{equation}\label{TCL-A-n-4}
	\left|\E\left[\tilde\chi^N_1(t)\left(\chi^N_2(t)-\tilde\chi^N_2(t)\right)\right]\right|=o\left(\frac{1}{N}\right).
	\end{equation}
	Similarly, we show that,
	\begin{equation}\label{TCL-A-n-5}
	\left|\E\left[\left(\chi^N_1(t)-\tilde\chi^N_1(t)\right)\tilde\chi^N_2(t)\right]\right|=o\left(\frac{1}{N}\right).
	\end{equation}
	Moreover, from \eqref{TCL-A-n-2}, by the fact that $Q_1$ and $Q_2$ are independent, and exchangeability,
	\begin{multline*}
	\left|\E\left[\left(\chi^N_1(t)-\tilde\chi^N_1(t)\right)\left(\chi^N_2(t)-\tilde\chi^N_2(t)\right)\right]\right|\\
	\begin{aligned}
	&\leq\E\left[\left(\mathds{1}_{\left(E^N_1\right)^c}+\frac{Yt}{N}\right)\left(\mathds{1}_{\left(E^N_2\right)^c}+\frac{Yt}{N}\right)\right]\\
	&=\E\left[\mathbb{P}\left(\left(E^N_1\right)^c\big|\overline{\mathfrak F}^N_{(1,2)},Y\right)\mathbb{P}\left(\left(E^N_2\right)^c\big|\overline{\mathfrak F}^N_{(1,2)},Y\right)\right]+\frac{2t}{N}\E\left[Y\mathbb{P}\left(\left(E^N_1\right)^c\big|\overline{\mathfrak F}^N_{(1,2)},Y\right)\right]+\frac{t^2}{N^2}\E\left[Y^2\right]\\
	&\leq\frac{10t^2}{N^2}\E\left[Y^2\right],
	\end{aligned}
	\end{multline*}
	where we use \eqref{TCL-A-n-1}.
	
	Consequently,
	\begin{equation}\label{TCL-A-n-6}
	\left|\E\left[\left(\chi^N_1(t)-\tilde\chi^N_1(t)\right)\left(\chi^N_2(t)-\tilde\chi^N_2(t)\right)\right]\right|=o\left(\frac{1}{N}\right).
	\end{equation}
	Thus from \eqref{TCL-A-n-6}, \eqref{TCL-A-n-5}, \eqref{TCL-A-n-4}, \eqref{TCL-A-n-3} and \eqref{TCL-A-n-1'}, \eqref{TCL-A-n-0} holds.
\end{proof}

\begin{coro}\label{TCL-eqr-2}
	Under Assumption~\ref{TCL-AS-lambda}, for any $t\geq0$, as $N\to+\infty$,
	\begin{equation*}
	\Xi^N_{1,0}(t)\rightarrow0 \quad \text{in probability.}
	\end{equation*}	
\end{coro}
We set,
	\begin{equation*}
	\Xi^N_{1,0,1}(t)=\frac{1}{\sqrt N}\sum_{k=1}^{N}\gamma_{k,0}(t)\left(\exp\left(-\int_{0}^{t}\gamma_{k,0}(r)\overline{\mathfrak F}^N(r)dr\right)-\exp\left(-\int_{0}^{t}\gamma_{k,0}(r)\overline{\mathfrak F}(r)dr\right)\right).
	\end{equation*}
	Consequently, from \eqref{TCL-eqr-1} for any $t\geq0,\,\hat{\mathfrak S}_{1,0}^{N}(t)=\Xi^N_{1,0}(t)+\Xi^N_{1,0,1}(t)$. Therefore from Corollary~\ref{TCL-eqr-2} as $\Xi^N_{1,0}(t)\rightarrow0$ in probability, to conclude the proof of Lemma~\ref{TCL-lem-L2}, it suffices to establish the following Lemma.
	\begin{lemma}
		Let $\hat{\mathfrak F}$ be a limit of a converging subsequence of $\hat{\mathfrak F}^N$. Then, for any $t\geq0,$ almost surely,  as $N\to\infty,$
		\[\Xi^N_{1,0,1}(t)\to Z(t),\] where
		\begin{equation*}
		Z(t)=-\int_{0}^{t}\E\left[\gamma_{0}(t)\gamma_{0}(s)\exp\left(-\int_{0}^{t}\gamma_{0}(r)\overline{\mathfrak F}(r)dr\right)\right]\hat{\mathfrak F}(s)ds.
		\end{equation*}
	\end{lemma}
\begin{proof}
	By the Taylor formula we have
	\begin{align}\label{TCL-eq-3}
	\Xi^N_{1,0,1}(t)
	&=-Z^N(t)+\frac{1}{N}\sum_{k=1}^{N}r^N_k(t),
	\end{align}
	where 
	\begin{equation*}
	|r^N_k(t)|\leq \frac{t}{2}\sqrt{N}\int_{0}^{t}\left(\overline{\mathfrak F}^N(s)-\overline{\mathfrak F}(s)\right)^2ds,
	\end{equation*}
	and 
	\begin{equation*}
	Z^N(t)=\int_{0}^{t}\left(\frac{1}{N}\sum_{k=1}^{N}\gamma_{k,0}(t)\gamma_{k,0}(s)\exp\left(-\int_{0}^{t}\gamma_{k,0}(r)\overline{\mathfrak F}(r)dr\right)\right)\hat{\mathfrak F}^N(s)ds.
	\end{equation*}
	From Proposition~\ref{prop-c1}, it follows that,
	\begin{align}\label{TCL-eq-4}
	\E\left[\sup_{0\le t\le T}\left|\frac{1}{N}\sum_{k=1}^{N}r^N_k(t)\right|\right]&\leq T\sqrt N\int_{0}^{T}\E\left[\left(\overline{\mathfrak F}^N(s)-\overline{\mathfrak F}(s)\right)^2\right]ds\leq\frac{C_T}{\sqrt N}.	
	\end{align}
	Hence $\frac{1}{N}\sum_{k=1}^{N}r^N_k(t)\to0$ in probability.
	
	On the other hand, for any fixed $t$, as $(\gamma_{k,0})_k$ are i.i.d by the classical law of large numbers it follows that, for all $s\in[0,t]$, as $N\to\infty$,
	\begin{equation*}
	\frac{1}{N}\sum_{k=1}^{N}\gamma_{k,0}(t)\gamma_{k,0}(s)\exp\left(-\int_{0}^{t}\gamma_{k,0}(r)\overline{\mathfrak F}(r)dr\right)\to\E\left[\gamma_{0}(t)\gamma_{0}(s)\exp\left(-\int_{0}^{t}\gamma_{0}(r)\overline{\mathfrak F}(r)dr\right)\right]\quad a.s.
	\end{equation*} 
	and as $N\to\infty,\hat{\mathfrak F}^N\Rightarrow\hat{\mathfrak F}$ in $\bD$, by Slutsky's theorem,
	\begin{equation*} 	
	\left(\frac{1}{N}\sum_{k=1}^{N}\gamma_{k,0}(t)\gamma_{k,0}(s)\exp\left(-\int_{0}^{t}\gamma_{k,0}(r)\overline{\mathfrak F}(r)dr\right)\right)\hat{\mathfrak F}^N(s)\to\E\left[\gamma_{0}(t)\gamma_{0}(s)\exp\left(-\int_{0}^{t}\gamma_{0}(r)\overline{\mathfrak F}(r)dr\right)\right]\hat{\mathfrak F}(s),
	\end{equation*} 
	in law.
	For any fixed $t$, since the mapping $f\to\int_0^tf(s)ds$ is continuous from $\bD$ into $\R$, as $N\to\infty,$
	\begin{equation*}
	Z^N(t)\to\int_0^t \E\left[\gamma_{0}(t)\gamma_{0}(s)\exp\left(-\int_{0}^{t}\gamma_{0}(r)\overline{\mathfrak F}(r)dr\right)\right]\hat{\mathfrak F}(s)ds.
	\end{equation*}
	This concludes the proof.
	
	\end{proof}

\section{Approximation of the limit}\label{TCL-sec-L1}
In this section we establish the proof of Lemma~\ref{TCL-lem-L1}. 
 We organise this section as follows: In subsection~\ref{TCL-sub-secc-1} we give the proof of Lemma~\ref{TCL-lem-L1} with the proofs of the supporting lemmas  in subsection~\ref{TCL-Appr-L}.
\subsection{Proof of Lemma~\ref{TCL-lem-L1}}\label{TCL-sub-secc-1}


We first make the following observation related to the process $\hat{\mathfrak{S}}^N_{0}$. 

\begin{lemma}\label{TCL-l4} For $t\geq0$,
	\begin{multline*}
	\sqrt{N}\int_{0}^{t}\int_{\bD}\gamma(t-s)\left(\exp\left(-\int_{s}^{t}\gamma(r-s)\overline{\mathfrak F}^N(r)dr\right)-\exp\left(-\int_{s}^{t}\gamma(r-s)\overline{\mathfrak F}(r)dr\right)\right)\overline\Upsilon^N(s)\mu(d\gamma)ds\\
	=-\int_{0}^{t}\int_{s}^{t}\int_{\bD}\gamma(t-s)\gamma(r-s)\hat{\mathfrak{F}}^N(r)\exp\left(-\int_{s}^{t}\gamma(u-s)\overline{\mathfrak F}(u)du\right)\overline\Upsilon^N(s)\mu(d\gamma)drds+r^N(t),
	\end{multline*}
	where as $N\to\infty,\,r^N\to0$ in probability in $\bD$.
\end{lemma}
\begin{proof}
	By Taylor's formula we have 
	\begin{multline*}
	\sqrt{N}\int_{0}^{t}\int_{\bD}\gamma(t-s)\left(\exp\left(-\int_{s}^{t}\gamma(r-s)\overline{\mathfrak F}^N(r)dr\right)-\exp\left(-\int_{s}^{t}\gamma(r-s)\overline{\mathfrak F}(r)dr\right)\right)\overline\Upsilon^N(s)\mu(d\gamma)ds\\
	=-\int_{0}^{t}\int_{s}^{t}\int_{\bD}\gamma(t-s)\gamma(r-s)\hat{\mathfrak{F}}^N(r)\exp\left(-\int_{s}^{t}\gamma(u-s)\overline{\mathfrak F}(u)du\right)\overline\Upsilon^N(s)\mu(d\gamma)drds+r^N(t), 
	\end{multline*}
	where 
	\begin{equation}
	|r^N(t)|\leq\frac{\lambda_*\sqrt{N}}{2}\int_{0}^{t}\int_{\bD}\left(\int_{s}^{t}\gamma(r-s)\left(\overline{\mathfrak F}^N(r)-\overline{\mathfrak F}(r)\right)dr\right)^2\mu(d\gamma)ds\,.
	\end{equation}
	By H\"older's inequality and Proposition~\ref{prop-c1} it follows that,
	\begin{align*}
	\E\left[\sup_{0\le t\le T}|r^N(t)|\right]&\leq\frac{T\lambda_*\sqrt{N}}{2}\int_{0}^{T}\int_{s}^{T}\E\left[\left(\overline{\mathfrak F}^N(r)-\overline{\mathfrak F}(r)\right)^2\right]drds\leq\frac{T^3\lambda_*}{2\sqrt{N}}\,.
	\end{align*}
\end{proof}

Hence, from \eqref{TCL-eq-eqc10-0}, \eqref{TCL-eq-eqc10-1}, \eqref{TCL-uni-c1} and Lemma~\ref{TCL-l4}, it follows that 
\begin{align*}
& \varPsi_1(\hat{\mathfrak F}^N,\hat{\mathfrak S}^N,\hat{\mathfrak S}^N_{1,0},\hat{\mathfrak S}^N_2,\overline{\mathfrak{F}}^N,\overline{\mathfrak{S}}^N,\tilde{\mathfrak{S}}^N)(t) \\
& =\hat{\mathfrak S}^N(t)-\hat{\mathfrak S}^N_{1,0}(t)-\hat{\mathfrak S}^N_2(t)\\
&+\int_{0}^{t}\int_{s}^{t}\int_{\bD}\gamma(t-s)\gamma(r-s)\exp\left(-\int_{s}^{t}\gamma(u-s)\overline{\mathfrak F}(u)du\right)\mu(d\gamma)\hat{\mathfrak{F}}^N(r)\overline{\mathfrak{F}}^N(s)\overline{\mathfrak{S}}^N(s)drds\\
&+\int_{0}^{t}\int_{\bD}\gamma(t-s)\exp\left(-\int_{s}^{t}\gamma(r-s)\overline{\mathfrak F}(r)dr\right)\mu(d\gamma)\left(\overline{\mathfrak F}^N(s)\hat{\mathfrak{S}}^N_2(s)-\overline{\mathfrak F}^N(s)\hat{\mathfrak{S}}^N(s)-\hat{\mathfrak F}^N(s)\tilde{\mathfrak{S}}^N(s)\right)ds\\
&:=\Xi^N_3(t)+r^N(t)+\Xi^N_4(t),
\end{align*}
where 
\begin{align}\label{TCL-eq-eqc8-1}
\Xi^N_3(t)&=\hat{\mathfrak{S}}^N_{1,1}(t)-\sqrt{N}\int_{0}^{t}\int_{\bD}\gamma(t-s)\left(\exp\left(-\int_{s}^{t}\gamma(r-s)\overline{\mathfrak F}^N(r)dr\right)\right.\nonumber\\&\hspace{6cm}\left.-\exp\left(-\int_{s}^{t}\gamma(r-s)\overline{\mathfrak F}(r)dr\right)\right)\overline\Upsilon^N(s)\mu(d\gamma)ds,
\end{align}
\begin{equation}\label{TCL-eq-eqc8-2}
\Xi^N_4(t)=\hat{\mathfrak S}^N_{1,2}(t)-\int_{0}^{t}\int_{\bD}\gamma(t-s)\exp\left(-\int_{s}^{t}\gamma(r-s)\overline{\mathfrak F}(r)dr\right)\mu(d\gamma)\sqrt{N}\left(\overline{\Upsilon}^N(s)-\tilde{\Upsilon}^N(s)\right)ds,
\end{equation}
and where we use the fact that,
\[\sqrt{N}(\overline\Upsilon^N(s)-\tilde\Upsilon^N(s))=\overline{\mathfrak F}^N(s)\hat{\mathfrak{S}}^N_2(s)-\overline{\mathfrak F}^N(s)\hat{\mathfrak{S}}^N(s)-\hat{\mathfrak F}^N(s)\tilde{\mathfrak{S}}^N(s).\]
On the other hand, from \eqref{TCL-eq-eqc10-01} and \eqref{TCL-eqc10},
\begin{align}
& \varPsi_2(\hat{\mathfrak F}^N,\hat{\mathfrak S}^N,\hat{\mathfrak F}^N_2,\hat{\mathfrak S}^N_2,\overline{\mathfrak{F}}^N,\overline{\mathfrak{S}}^N,\tilde{\mathfrak{S}}^N)(t) \nonumber \\
&=\hat{\mathfrak F}^N(t)-\hat{\mathfrak{F}}_2^N(t)-\int_{0}^{t}\overline{\lambda}(t-s)\hat{\mathfrak S}^N(s)\overline{\mathfrak F}^N(s)ds-\int_{0}^{t}\overline{\lambda}(t-s)\hat{\mathfrak S}^N_2(s)\overline{\mathfrak F}^N(s)ds \nonumber\\
& \qquad \qquad +\int_{0}^{t}\overline{\lambda}(t-s)\tilde{\mathfrak S}^N(s)\hat{\mathfrak F}^N(s)ds \nonumber \\
&=\frac{1}{\sqrt{N}}\sum_{k=1}^{N}\int_{0}^{t}\int_{\bD^2}\int_{\Upsilon_k^N(r^-)\wedge\Upsilon_k(r^-)}^{\Upsilon_k^N(r^-)\vee\Upsilon_k(r^-)}\lambda(t-s)\sign(\Upsilon_k^N(s^-)-\Upsilon_k(s^-))\overline{Q}_k(ds,d\lambda,d\gamma,du)\nonumber \\
&=:\hat{\mathfrak F}^N_{1,1}(t).
\end{align}
To conclude the proof of Lemma~\ref{TCL-lem-L1} it suffices to establish the following three lemmas, whose proofs are given in the next subsection. 
\begin{lemma}\label{TCL-cv-F1}
	Under Assumption~\ref{TCL-AS-lambda}, as $N\to+\infty$,
	\begin{equation*}
	\hat{\mathfrak F}^N_{1,1}\Rightarrow 0 \quad \text{ in} \quad \bD,
	\end{equation*}	
	in probability. 
\end{lemma}
\begin{lemma}\label{TCL-lem-10}
	As $N\to+\infty$,
	\[\Xi^N_3\to0 \quad \text{ in } \quad \bD,
	\]
	in probability. 
\end{lemma}
\begin{lemma}\label{TCL-lem-11}
	As $N\to\infty$, 
	\begin{equation*}
	\Xi^N_4\to0 \quad \text{ in } \quad \bD,
	\end{equation*}
	in probability. 
\end{lemma}

\subsection{Proofs of Lemmas \ref{TCL-cv-F1}-\ref{TCL-lem-11}}\label{TCL-Appr-L}
\subsubsection{Proof of Lemma~\ref{TCL-cv-F1}}\label{TCL-pr-F1}
We recall that 
\begin{equation}\label{eqc3}
\hat{\mathfrak F}^N_{1,1}(t)=\frac{1}{\sqrt{N}}\sum_{k=1}^{N}\int_{0}^{t}\int_{\bD^2}\int_{\Upsilon_k^N(r^-)\wedge\Upsilon_k(r^-)}^{\Upsilon_k^N(r^-)\vee\Upsilon_k(r^-)}\lambda(t-s)\sign(\Upsilon_k^N(s^-)-\Upsilon_k(s^-))\overline{Q}_k(ds,d\lambda,d\gamma,du).
\end{equation}
As $(\overline{Q}_k)_k$ are i.i.d, by exchangeability and from \eqref{eqgam},  it follows that 
\begin{align*}
\E\left[\left(\hat{\mathfrak F}^N_{1,1}(t)\right)^2\right]&=\E\left[\int_{0}^{t}\lambda(t-s)^2|\Upsilon^N_1(s)-\Upsilon_1(s)|ds\right]\\
&\leq(\lambda_\ast)^2\int_{0}^{t}\E\left[\left|\Upsilon^N_1(s)-\Upsilon_1(s)\right|\right]ds\to0\text{ as }N\to+\infty.
\end{align*}
To conclude,  it suffices to prove tightness of $(\hat{\mathfrak F}^N_{1,1})_N$. By the expression in \eqref{eqc3}, tightness of the sequence processes $(\hat{\mathfrak F}^N_{1,1})_N$ can be deduced from the tightness of the following sequence processes since $\sign(\Upsilon_k^N(s^-)-\Upsilon_k(s^-))$ does not modify the setting:
\begin{align*}
&\Xi^N_1(t)=\frac{1}{\sqrt{N}}\sum_{k=1}^{N}\int_{0}^{t}\int_{\bD^2}\int_{\Upsilon_k^N(r^-)\wedge\Upsilon_k(r^-)}^{\Upsilon_k^N(r^-)\vee\Upsilon_k(r^-)}\lambda(t-s)Q_k(ds,d\lambda,d\gamma,du),\\
&\Xi^N_2(t)=\sqrt N\int_{0}^{t}\overline{\lambda}(t-s)\left|\tilde{\Upsilon}^N(s)-\overline\Upsilon^N(s)\right|ds,
\end{align*}
where we recall that 
\begin{equation*}
\overline{\Upsilon}^N(t)=\frac{1}{N}\sum_{k=1}^{N}\Upsilon_k^N(t)\text{ and }\tilde{\Upsilon}^N(t)=\frac{1}{N}\sum_{k=1}^{N}\Upsilon_k(t).
\end{equation*}
The tightness of $(\Xi^N_1)_N$ will be established in Lemma~\ref{TCL-tight-F} below. Hence it remains to prove the tightness of $(\Xi^N_2)_N$.

From \eqref{fluc-1} we have 
\begin{equation}
\E\left[\sup_{0\leq t\leq T}\Xi^N_2(t)\right]\leq\lambda_\ast\int_{0}^{T}\E\left[\left|\sqrt{N}\left(\tilde{\Upsilon}^N(r)-\overline\Upsilon^N(r)\right)\right|\right]dr
\leq\lambda_* C_T.\label{eqc4-11}
\end{equation}
From Lemma~\ref{TCL-lem-barlambda-inc-bound}, for any $0\leq s\leq t\leq T$, 
\begin{align}\label{eqc6-1}
\left|\Xi^N_2(t)-\Xi^N_2(s)\right|&\leq\lambda_*\sqrt{N}\int_{s}^{t}\left|\tilde{\Upsilon}^N(r)-\overline\Upsilon^N(r)\right|dr+\sqrt N\int_{0}^{s}\left|\overline\lambda(t-r)-\overline\lambda(s-r)\right|\left|\tilde{\Upsilon}^N(r)-\overline\Upsilon^N(r)\right|dr\nonumber\\
&\leq\lambda_*\sqrt{N}\int_{s}^{t}\left|\tilde{\Upsilon}^N(r)-\overline\Upsilon^N(r)\right|dr+\sqrt{N}(t-s)^\alpha\int_{0}^{s}\left|\tilde{\Upsilon}^N(r)-\overline\Upsilon^N(r)\right|dr\nonumber\\
&\hspace*{3cm}+\lambda_*\sqrt{N}\sum_{j=1}^{\ell-1}\int_{0}^{s}\left(F_j(t-r)-F_j(s-r)\right)\left|\tilde{\Upsilon}^N(r)-\overline\Upsilon^N(r)\right|dr.
\end{align}
By Markov's inequality, 
\begin{align}
&\mathbb{P}\left(\sup_{0\leq s<t< T,|t-s|\leq\delta}\sqrt N\int_{s}^{t}\left|\tilde{\Upsilon}^N(r)-\overline\Upsilon^N(r)\right|dr\geq\theta\right)\nonumber\\
&\leq\frac{N}{\theta^2}\E\left[\sup_{0\leq s<t< T,|t-s|\leq\delta}(t-s)\int_{s}^{t}\left|\tilde{\Upsilon}^N(r)-\overline\Upsilon^N(r)\right|^2dr\right]\nonumber\\
&\leq\frac{\delta}{\theta^2}N\int_{0}^{T}\E\left[\left|\overline\Upsilon^N(r)-\tilde\Upsilon^N(r)\right|^2\right]dr \nonumber \\
&\leq \frac{\delta C_T}{\theta^2},\label{eqc7}
\end{align}
where the last line follows from applying \eqref{fluc-2} with $k=2$.

On the other hand, 
\begin{align}
\mathbb{P}\left(\sup_{0\leq s<t< T,|t-s|\leq\delta}(t-s)^\alpha\sqrt{N}\int_{0}^{s}\left|\tilde{\Upsilon}^N(r)-\overline\Upsilon^N(r)\right|dr\geq\theta\right)&\leq\frac{\delta^\alpha}{\theta}\sqrt N\int_{0}^{T}\E\left[\left|\tilde{\Upsilon}^N(r)-\overline\Upsilon^N(r)\right|\right]dr\nonumber\\
&\leq\frac{C_T}{\theta}\delta^\alpha,\label{eqc8}
\end{align}
where the last line follows from applying \eqref{fluc-2} with $k=1$. 

Applying again \eqref{fluc-2} with $k=1$, by Markov's inequality and Assumption~\ref{TCL-AS-lambda-2},
\begin{multline}\label{TCL-eqc8'}
\mathbb{P}\left(\sup_{0\leq s<t< T,|t-s|\leq\delta}\sqrt{N}\sum_{j=1}^{\ell-1}\int_{0}^{s}\left(F_j(t-r)-F_j(s-r)\right)\left|\tilde{\Upsilon}^N(r)-\overline\Upsilon^N(r)\right|dr\geq\theta\right)\\
\begin{aligned}
&\leq\frac{\ell\delta^\rho}{\theta}\sqrt N\int_{0}^{T}\E\left[\left|\tilde{\Upsilon}^N(r)-\overline\Upsilon^N(r)\right|\right]dr\\
&\leq\frac{\ell C_T}{\theta}\delta^\rho.
\end{aligned}
\end{multline}
Thus from \eqref{eqc6-1}, \eqref{eqc7}, \eqref{eqc8}and \eqref{TCL-eqc8'}, 
\begin{equation*}
\lim_{\delta\to0}\limsup_N\mathbb{P}\left(w_T(\Xi^N_2,\delta)\geq\theta\right)=0,
\end{equation*}
combined with \eqref{eqc4-11}, it follows from Theorem~\ref{th-tight} that $(\Xi^N_2)_N$ is $\bC$-tight.


\subsubsection{Proof of Lemma~\ref{TCL-lem-10}}\label{TCL-pr-F2}
We recall that 
\begin{align*}
\hat{\mathfrak S}^N_{1,1}(t)=\frac{1}{\sqrt{N}}\sum_{k=1}^{N}\int_{0}^{t}\int_{\bD^2}\int_{0}^{\Upsilon^N_k(s^-)}\gamma(t-s)\left(\mathds{1}_{P_{k}(s,t,\gamma,\overline{\mathfrak F}^N)=0} -\mathds{1}_{P_{k}(s,t,\gamma,\overline{\mathfrak F})=0}\right)Q_k(ds,d\lambda,d\gamma,du)\,. 
\end{align*}
Define 
\begin{equation*}
\tilde{\Xi}^N_{3,31}(t)=\frac{1}{\sqrt{N}}\sum_{k=1}^{N}\chi^N_k(t)\,. 
\end{equation*}
where 
\begin{align*}
\chi^N_k(t)&=\int_{0}^{t}\int_{\bD^2}\int_{0}^{\Upsilon^N_k(s^-)}\left\{\gamma(t-s)\left(\mathds{1}_{P_{k}(s,t,\gamma,\overline{\mathfrak F}^N)=0} -\mathds{1}_{P_{k}(s,t,\gamma,\overline{\mathfrak F})=0}\right)\right.\\
&\left.-\int_{\bD}\tilde\gamma(t-s)\left(\exp\left(-\int_{s}^{t}\tilde\gamma(r-s)\overline{\mathfrak F}^N(r)dr\right)-\exp\left(-\int_{s}^{t}\tilde\gamma(r-s)\overline{\mathfrak F}(r)dr\right)\right)\mu(d\tilde\gamma)\right\}Q_k(ds,d\lambda,d\gamma,du).
\end{align*}
\begin{lemma}\label{TCL-l2} For any $t\geq0$, as $N\to+\infty$,
	\[\E\left[\left(\tilde{\Xi}^N_{3,31}(t)\right)^2\right]\to0.\]
\end{lemma}
\begin{proof}
	
	By exchangeability we have
	\begin{equation}\label{TCL-m0}
	\E\left[\left(\tilde{\Xi}^N_{3,31}(t)\right)^2\right]=\E\left[\left(\chi^N_1(t)\right)^2\right]+(N-1)\E\left[\chi^N_1(t)\chi^N_2(t)\right].
	\end{equation}
	Set
	\begin{align*}
	\tilde\chi^N_k(t)&=\int_{0}^{t}\int_{\bD^2}\int_{0}^{\Upsilon^N_k(s^-)}\left\{\gamma(t-s)\left(\mathds{1}_{P_{k}(s,t,\gamma,\overline{\mathfrak F}^N_{(1,2)})=0} -\mathds{1}_{P_{k}(s,t,\gamma,\overline{\mathfrak F})=0}\right)\right.\\
	&\left.-\int_{\bD}\tilde\gamma(t-s)\left(\exp\left(-\int_{s}^{t}\tilde\gamma(r-s)\overline{\mathfrak F}^N_{(1,2)}(r)dr\right)-\exp\left(-\int_{s}^{t}\tilde\gamma(r-s)\overline{\mathfrak F}(r)dr\right)\right)\mu(d\tilde\gamma)\right\}Q_k(ds,d\lambda,d\gamma,du),
	\end{align*}
	and 
	\begin{align*}
	\hat\chi^N_{k}(t)&:=\int_{0}^{t}\int_{\bD^2}\int_{0}^{\lambda_*}\left\{\gamma(t-s)\left|\mathds{1}_{P_{k}(s,t,\gamma,\overline{\mathfrak F}^N_{(1,2)})=0} -\mathds{1}_{P_{k}(s,t,\gamma,\overline{\mathfrak F})=0}\right|\right.\\
	&\left.+\int_{\bD}\tilde\gamma(t-s)\left|\exp\left(-\int_{s}^{t}\tilde\gamma(r-s)\overline{\mathfrak F}^N_{(1,2)}(r)dr\right)-\exp\left(-\int_{s}^{t}\tilde\gamma(r-s)\overline{\mathfrak F}(r)dr\right)\right|\mu(d\tilde\gamma)\right\}Q_k(ds,d\lambda,d\gamma,du).
	\end{align*}
	For $0\leq s\leq t$, and $\gamma\in D$, let $E^N_k(s,\gamma)$ denote the event 
	\begin{equation*}
	E^N_k(s,\gamma)=\left\{P_{k}\left(s,t,\gamma,\overline{\mathfrak F}^N_{(1,2)}-\frac{Y}{N}\right)=P_{k}\left(s,t,\gamma,\overline{\mathfrak F}^N_{(1,2)}+\frac{Y}{N}\right)\right\},
	\end{equation*}
	where $Y$ is a geometric random variable independent from $Q_1$ and $Q_2$ chosen as in \eqref{TCL-eq1}, such that almost surely for all $t\in[0,T]$,
	\[\overline{\mathfrak F}^N_{(1,2)}(t)-\frac{Y}{N}\leq\overline{\mathfrak F}^N(t)\leq \overline{\mathfrak F}^N_{(1,2)}(t)+\frac{Y}{N}.\]
	We denote
	\begin{equation*}
	\mathcal{F}_t^N=\sigma\{(\lambda_{k,i}(\varsigma^N_k(\cdot)))_{1\leq k\leq N,i\leq A^N_k(t)},\,(\gamma_{k,i}(\varsigma^N_k(\cdot)))_{1\leq k\leq N,i\leq A^N_k(t)},\,(Q_k|_{[0,t]\times E})_{k\leq N}\}.
	\end{equation*} 
	As
	\begin{equation*}
	P_{k}(s,t,\gamma,\overline{\mathfrak F}^N)\mathds{1}_{E_k^N(s,\gamma)}=P_{k}\left(s,t,\gamma,\overline{\mathfrak F}^N_{(1,2)}\right)\mathds{1}_{E_k^N(s,\gamma)},
	\end{equation*}
	and $\Upsilon^N_k(s^-)\leq\lambda_*,$ it follows that 
	\begin{equation}
	\left|\chi^N_k(t)-\tilde\chi^N_k(t)\right|\leq\int_{0}^{t}\int_{\bD^2}\int_{0}^{\lambda_*}\mathds{1}_{(E^N_k(s,\gamma))^c}Q_k(ds,d\lambda,d\gamma,du)
	+\frac{Yt}{N}B_k(t) ,\label{TCL-m1}
	\end{equation}
	where
	\begin{equation*}
	B_k(t)=\int_{0}^{t}\int_{\bD^2}\int_{0}^{\lambda_*}Q_k(ds,d\lambda,d\gamma,du).
	\end{equation*}
	
	Moreover, as $Q_k|_{]s,t]}$ is independent of $(\mathcal{F}_s^N,\,\overline{\mathfrak F}^N_{(1,2)},\,Y,\,Q_\ell)$, for $\ell\neq k$ and $k\in\{1,2\}$,
	\begin{equation}
	\mathbb{P}\left((E^N_k(s,\gamma))^c\Big|\mathcal{F}_s^N,\,\overline{\mathfrak F}^N_{(1,2)},\,Y,Q_\ell\right)\leq\frac{2tY}{N}.\label{TCL-m2}
	\end{equation}
	We have 
	\begin{align}\label{TCL-m2'}
	\E\left[\chi^N_1(t)\chi^N_2(t)-\tilde\chi^N_1(t)\tilde\chi^N_2(t)\right]&=\E\left[\tilde\chi^N_1(t)\left(\chi^N_2(t)-\tilde\chi^N_2(t)\right)\right]+\E\left[\left(\chi^N_1(t)-\tilde\chi^N_1(t)\right)\tilde\chi^N_2(t)\right]\nonumber\\&\hspace*{1cm}+\E\left[\left(\chi^N_1(t)-\tilde\chi^N_1(t)\right)\left(\chi^N_2(t)-\tilde\chi^N_2(t)\right)\right]. 
	\end{align}
	Note that \[|\chi^N_k|\leq\hat\chi^N_{k} \text{ and }|\tilde\chi^N_k|\leq\hat\chi^N_{k}.\]

	As $Q_2$ is independent of $(\overline{\mathfrak F}^N_{(1,2)},\,Y,\,Q_1)$, from \eqref{TCL-m1}, \eqref{TCL-m2} and Theorem~\ref{TCL-th-1},
	\begin{align*}
	\left|\E\left[\tilde\chi^N_1(t)\left(\chi^N_2(t)-\tilde\chi^N_2(t)\right)\right]\right|
	&\leq\lambda_*\E\left[\hat\chi^N_{1}(t)\int_{0}^{t}\int_{\bD}\mathbb{P}\left((E^N_2(s,\gamma))^c\Big|\mathcal{F}_s^N,\,\overline{\mathfrak F}^N_{(1,2)},\,Y,Q_1\right)\mu(d\gamma)ds\right]\\&\hspace*{8cm}+\frac{t}{N}\E\left[\hat\chi^N_{1}(t)YB_2(t)\right]\\
	&\leq\frac{3t^2\lambda_*}{N}\E\left[Y\hat\chi^N_{1}(t)\right],
	\end{align*}
	where as $N\to\infty,\,\E\left[Y\hat\chi^N_{1}(t)\right]\to0,$ because $\hat\chi^N_{1}(t)\to0$ in probability as $N\to\infty$ and $\hat\chi^N_{1}$ is bounded.

	Thus it follows that
	\begin{equation*}
	\left|\E\left[\tilde\chi^N_1(t)\left(\chi^N_2(t)-\tilde\chi^N_2(t)\right)\right]\right|=o\left(\frac{1}{N}\right).
	\end{equation*}
	Similarly we show that
	\begin{equation*}
	\left|\E\left[\left(\chi^N_1(t)-\tilde\chi^N_1(t)\right)\tilde\chi^N_2(t)\right]\right|=o\left(\frac{1}{N}\right).
	\end{equation*}
	From \eqref{TCL-m1},
	\begin{multline*}
	\E\left[\left|\chi^N_1(t)-\tilde\chi^N_1(t)\right|\left|\chi^N_2(t)-\tilde\chi^N_2(t)\right|\right]\\
	\begin{aligned}
	&\leq\E\left[\int_{0}^{t}\int_{\bD^2}\int_{0}^{\lambda_*}\mathds{1}_{(E^N_1(s,\gamma))^c}Q_1(ds,d\lambda,d\gamma,du)\int_{0}^{t}\int_{\bD^2}\int_{0}^{\lambda_*}\mathds{1}_{(E^N_2(s,\gamma))^c}Q_2(ds,d\lambda,d\gamma,du)\right]\\
	&+\frac{t}{N}\E\left[YB_2(t)\int_{0}^{t}\int_{\bD^2}\int_{0}^{\lambda_*}\mathds{1}_{(E^N_1(s,\gamma))^c}Q_1(ds,d\lambda,d\gamma,du)\right]\\
	&+\frac{t}{N}\E\left[YB_1(t)\int_{0}^{t}\int_{\bD^2}\int_{0}^{\lambda_*}\mathds{1}_{(E^N_2(s,\gamma))^c}Q_2(ds,d\lambda,d\gamma,du)\right]+\frac{t^2}{N^2}\E\left[Y^2B_1(t)B_2(t)\right].
	\end{aligned}
	\end{multline*}
	Conditionally on $\overline{\mathfrak F}^N_{(1,2)}$ and $Y$, the terms inside the first expectation above are independent because $Q_1$ and $Q_2$ are independent, and from \eqref{TCL-m2} and Theorem~\ref{TCL-th-1}, it follows that
	\begin{multline*}
	\E\left[\int_{0}^{t}\int_{\bD^2}\int_{0}^{\lambda_*}\mathds{1}_{(E^N_1(s,\gamma))^c}Q_1(ds,d\lambda,d\gamma,du)\int_{0}^{t}\int_{\bD^2}\int_{0}^{\lambda_*}\mathds{1}_{(E^N_2(s,\gamma))^c}Q_2(ds,d\lambda,d\gamma,du)\right]\\
	\begin{aligned}
	&=\lambda_*^2\E\left[\int_{0}^{t}\int_{\bD}\mathbb{P}\left((E^N_1(s,\gamma))^c\Big|\overline{\mathfrak F}^N_{(1,2)},\,Y\right)\mu(d\gamma)ds\int_{0}^{t}\int_{\bD}\mathbb{P}\left((E^N_2(s,\gamma))^c\Big|\overline{\mathfrak F}^N_{(1,2)},\,Y\right)\mu(d\gamma)ds\right]\\
	&\leq\frac{4\lambda_*^2t^4}{N^2}\E\left[Y^2\right].
	\end{aligned}
	\end{multline*}
	Hence using again \eqref{TCL-m1} and the fact that $Q_k$ is independent of $(\overline{\mathfrak F}^N_{(1,2)},\,Y,\,Q_\ell),$ with $k\neq\ell$ and $k,\ell\in\{1,2\}$, we deduce that 
	\begin{multline*}
	\E\left[\left|\chi^N_1(t)-\tilde\chi^N_1(t)\right|\left|\chi^N_2(t)-\tilde\chi^N_2(t)\right|\right]\\
	\begin{aligned}
	&\leq\frac{4\lambda_*^2t^4}{N^2}\E\left[Y^2\right]
	+\frac{\lambda_*t}{N}\E\left[YB_2(t)\int_{0}^{t}\int_{\bD}\mathbb{P}\left((E^N_1(s,\gamma))^c\Big|\overline{\mathfrak F}^N_{(1,2)},\,Y,Q_2\right)\mu(d\gamma)ds\right]\\
	&\qquad+\frac{\lambda_*t}{N}\E\left[YB_1(t)\int_{0}^{t}\int_{\bD}\mathbb{P}\left((E^N_2(s,\gamma))^c\Big|\overline{\mathfrak F}^N_{(1,2)},\,Y,Q_1\right)\mu(d\gamma)ds\right]+\frac{t^2}{N^2}\E\left[Y^2B_1(t)B_2(t)\right]\\
	&\leq\frac{4\lambda_*^2t^4}{N^2}\E\left[Y^2\right]+\frac{2\lambda_*t^2}{N^2}\E\left[YB_2(t)\right]+\frac{2\lambda_*t^2}{N^2}\E\left[YB_1(t)\right]+\frac{t^2}{N^2}\E\left[Y^2B_1(t)B_2(t)\right].
	\end{aligned}
	\end{multline*}
	Hence,
	\begin{equation*}
	\left|\E\left[\left(\chi^N_1(t)-\tilde\chi^N_1(t)\right)\left(\chi^N_2(t)-\tilde\chi^N_2(t)\right)\right]\right|=o\left(\frac{1}{N}\right).
	\end{equation*}
	In conclusion, coming back to \eqref{TCL-m2'}
	\begin{equation}\label{TCL-m3}
	\left|\E\left[\chi^N_1(t)\chi^N_2(t)-\tilde\chi^N_1(t)\tilde\chi^N_2(t)\right]\right|=o\left(\frac{1}{N}\right).
	\end{equation}
	On the other hand, since $Q_1,\,Q_2$ and $\overline{\mathfrak F}^N_{(1,2)}$ are independent, 
	\begin{equation*}
	\E\left[\tilde\chi^N_1(t)\tilde\chi^N_2(t)\big|\overline{\mathfrak F}^N_{(1,2)}\right]=0.
	\end{equation*}
	Hence from \eqref{TCL-m0} and \eqref{TCL-m3}, it follows that 
	\begin{align*}
	\E\left[\left(\tilde{\Xi}^N_{3,31}(t)\right)^2\right]&\leq\E\left[\left(\hat\chi^N_1(t)\right)^2\right]+(N-1)\left|\E\left[\chi^N_1(t)\chi^N_2(t)-\tilde\chi^N_1(t)\tilde\chi^N_2(t)\right]\right|\\
	&=\E\left[\left(\hat\chi^N_1(t)\right)^2\right]+o\left(1\right).
	\end{align*} 
	However, as $Q_1$ and $\overline{\mathfrak F}^N_{(1,2)}$ are independent, from Theorem~\ref{TCL-th-1}
	\begin{multline*}
	\E\left[\hat\chi^N_{1}(t)\Big|\overline{\mathfrak F}^N_{(1,2)}\right]=\lambda_*\int_{0}^{t}\int_{\bD}\gamma(t-s)\E\left[\left|\mathds{1}_{P_{1}(s,t,\gamma,\overline{\mathfrak F}^N_{(1,2)})=0} -\mathds{1}_{P_{1}(s,t,\gamma,\overline{\mathfrak F})=0}\right|\Big|\overline{\mathfrak F}^N_{(1,2)}\right]\mu(d\gamma)ds\\
	+\lambda_*\int_{0}^{t}\int_{\bD}\tilde\gamma(t-s)\E\left[\left|\exp\left(-\int_{s}^{t}\tilde\gamma(r-s)\overline{\mathfrak F}^N_{(1,2)}(r)dr\right)-\exp\left(-\int_{s}^{t}\tilde\gamma(r-s)\overline{\mathfrak F}(r)dr\right)\right|\Big|\overline{\mathfrak F}^N_{(1,2)}\right]\mu(d\tilde\gamma)ds.
	\end{multline*}
	Hence since $\overline{\mathfrak F}^N_{(1,2)}\to\overline{\mathfrak F}$, as $N\to\infty,$ and $\hat\chi^N_{1}$ is bounded, it follows that as $N\to\infty,\,\E\left[\hat\chi^N_{1}(t)\right]\to0$ and $\E\left[(\hat\chi^N_{1}(t))^2\right]\to0$.
\end{proof}
Now, let 
\begin{equation*}
\Psi^N_{32}(t)=\frac{1}{\sqrt{N}}\sum_{k=1}^N\chi^N_k(t),	
\end{equation*}
where
\begin{align*}
\chi^N_k(t)=\int_{0}^{t}\int_{0}^{\Upsilon^N_k(s^-)}\int_{\bD}\tilde\gamma(t-s)\left(\exp\left(-\int_{s}^{t}\tilde\gamma(r-s)\overline{\mathfrak F}^N(r)dr\right)-\exp\left(-\int_{s}^{t}\tilde\gamma(r-s)\overline{\mathfrak F}(r)dr\right)\right)\mu(d\tilde\gamma)\\\hspace{14cm}\overline{Q}_k(ds,d\lambda,d\gamma,du).
\end{align*}
\begin{lemma}\label{TCL-l3}
For any $t\geq0$, as $N\to\infty$,
	\begin{equation*}
	\E\left[|\Psi^N_{32}(t)|\right]\to0.
	\end{equation*}
\end{lemma}
\begin{proof}
	Let 
	\begin{align*}
	\tilde\chi^N_k(t)=\int_{0}^{t}\int_{0}^{\Upsilon^N_k(s^-)}\int_{\bD}\tilde\gamma(t-s)\left(\exp\left(-\int_{s}^{t}\tilde\gamma(r-s)\overline{\mathfrak F}^N_{(1,2)}(r)dr\right)-\exp\left(-\int_{s}^{t}\tilde\gamma(r-s)\overline{\mathfrak F}(r)dr\right)\right)\mu(d\tilde\gamma)\\\hspace{14cm}\overline{Q}_k(ds,d\lambda,d\gamma,du).
	\end{align*}
	As in Lemma~\ref{TCL-l2} we have 
	\begin{align*}
	\left|\chi^N_k(t)-\tilde\chi^N_k(t)\right|\leq\frac{tY}{N}\left(B_k(t)+\lambda_*\right).
	\end{align*}
	On the other hand, since $Q_1,\,Q_2$ and $\overline{\mathfrak F}^N_{(1,2)}$ are independent, it follows that 
	\begin{equation*}
	\E\left[\tilde\chi^N_1(t)\tilde\chi^N_2(t)\big|\overline{\mathfrak F}^N_{(1,2)}\right]=0.
	\end{equation*}
	The Lemma follows by exchangeability and the fact that as $N\to\infty,\,\tilde\chi^N_k(t)\to0$, proceeding as in \eqref{TCL-m0} and \eqref{TCL-m2'}
\end{proof}
We recall that 
\begin{align*}
	\Xi^N_3(t)&=\hat{\mathfrak{S}}^N_{1,1}(t)-\sqrt{N}\int_{0}^{t}\int_{\bD}\gamma(t-s)\left(\exp\left(-\int_{s}^{t}\gamma(r-s)\overline{\mathfrak F}^N(r)dr\right)\right.\nonumber\\&\hspace{6cm}\left.-\exp\left(-\int_{s}^{t}\gamma(r-s)\overline{\mathfrak F}(r)dr\right)\right)\overline\Upsilon^N(s)\mu(d\gamma)ds.
\end{align*}

	From Lemmas~\ref{TCL-l2} and~\ref{TCL-l3}, for each $t\geq0,$ as $N\to\infty,\,\Xi^N_3(t)\to0$ in probability.
	As a consequence, for all $n\in\mathbb{N},\,t_0<t_1<\cdots<t_{n},$
	\[\left(\Xi^N_{3}(t_0),\Xi^N_{3}(t_1),\cdots,\Xi^N_{3}(t_n)\right)\to\left(0,0,\cdots,0\right).\]
	To conclude it remains to show the tightness of $(\Xi^N_{3})_N$, but from Lemma~\ref{TCL-tight-F-2} $(\hat{\mathfrak S}^N_0)_N$ is $\bC$-tight, hence using Lemma~\ref{TCL-l4}, it remains to show the tightness of the following sequence of processes: 
	\begin{equation}
		\Xi^N_{3,10}(t)=\int_{0}^{t}\int_{s}^{t}\int_{\bD}\gamma(t-s)\gamma(r-s)\hat{\mathfrak{F}}^N(r)\exp\left(-\int_{s}^{t}\gamma(u-s)\overline{\mathfrak F}(u)du\right)\overline\Upsilon^N(s)\mu(d\gamma)drds.
	\end{equation}  
	Since the pair $(\hat{\mathfrak{F}}^N,\overline\Upsilon^N)$ is tight in $\bD^2$, the result follows using the continuous mapping theorem. 

This concludes the proof of Lemma~\ref{TCL-lem-10}.

\subsubsection{Proof of Lemma~\ref{TCL-lem-11}}\label{TCL-pr-F3}
We recall that 
\begin{equation*}
\hat{\mathfrak S}^N_{1,2}(t):=\frac{1}{\sqrt{N}}\sum_{k=1}^{N}\int_{0}^{t}\int_{\bD^2}\int_{\Upsilon_k(s^-)\wedge\Upsilon_k^N(s^-)}^{\Upsilon_k(s^-)\vee\Upsilon_k^N(s^-)}\gamma(t-s)\mathds{1}_{P_k\left(s,t,\gamma,\overline{\mathfrak F}\right)=0}\sign(\Upsilon_k^N(s^-)-\Upsilon_k(s^-))Q_k(ds,d\lambda,d\gamma,du).
\end{equation*}
We define for any $k\in\mathbb{N}$, 
\begin{equation}\label{TCL-Delta}
\Delta^N_k(t)=\int_{0}^{t}\int_{\Upsilon_k(s^-)\wedge\Upsilon_k^N(s^-)}^{\Upsilon_k(s^-)\vee\Upsilon_k^N(s^-)}Q_k(ds,du)\,\,.
\end{equation}
and let $(\vartheta_{k,i}^N)_{k,i}$ be such that 
\begin{equation*}
	\Delta^N_k(t)=\sum_{i\geq1}\mathds{1}_{\vartheta_{k,i}^N\leq t}.
\end{equation*}

We note that, for some i.i.d $(\gamma_{k,i},\,i\geq1)$, 
\begin{equation}
\hat{\mathfrak S}^N_{1,2}(t)=\frac{1}{\sqrt{N}}\sum_{k=1}^{N}\sum_{i\geq1}\gamma_{k,i}(t-\vartheta_{k,i}^N)\mathds{1}_{P_{k}(\vartheta_{k,i}^N,t,\gamma_{k,i},\overline{\mathfrak F})=0}\sign(\Upsilon_k^N(\vartheta_{k,i}^N)-\Upsilon_k(\vartheta_{k,i}^N))\mathds{1}_{\vartheta_{k,i}^N\leq t }\,\,.
\end{equation}
We recall that 
\[\overline{\Upsilon}^N(t)=\overline{\mathfrak S}^N(t)\overline{\mathfrak F}^N(t)\text{ and }\tilde{\Upsilon}^N(t)=\frac{1}{N}\sum_{k=1}^{N}\gamma_{k,A_k(t)}(\varsigma_k(t))\overline{\mathfrak F}(t).\] 
We set
\begin{equation}
\mathcal{G}_t^N=\sigma\{(\lambda_{k,i})_{1\leq k\leq N,i< A^N_k(t)},\,(\gamma_{k,i})_{1\leq k\leq N,i< A^N_k(t)},\,(Q_k|_{[0,t]\times E})_{k\leq N}\},
\end{equation} 
and from \eqref{TCL-eq-eqc8-2} we recall
\begin{equation}\label{TCL-eq-eqc8-20}
\Xi^N_4(t)=\hat{\mathfrak S}^N_{1,2}(t)-\int_{0}^{t}\int_{\bD}\gamma(t-s)\exp\left(-\int_{s}^{t}\gamma(r-s)\overline{\mathfrak F}(r)dr\right)\mu(d\gamma)\sqrt{N}\left(\overline{\Upsilon}^N(s)-\tilde{\Upsilon}^N(s)\right)ds.
\end{equation}
Note that from Lermma~\ref{TCL-tight-F-2} $\Xi^N_4$ is tight. 

	We want to show that, as $N$ tends to $\infty,\,\Xi^N_4\to0$ in $\bD$ in probability. 
	To do this, 
	we define
	\begin{equation}
	\Xi^N_{31}(t)=\frac{1}{\sqrt{N}}\sum_{k=1}^{N}\sum_{i\geq1}\chi^N_{k,i}(t),
	\end{equation}
	with
	\begin{align}
	\chi^N_{k,i}(t)&=\left(\gamma_{k,i}(t-\vartheta_{k,i}^N)\mathds{1}_{P_{k}(\vartheta_{k,i}^N,t,\gamma_{k,i},\overline{\mathfrak F})=0}
	-\int_{\bD}\gamma(t-\vartheta_{k,i}^N)\exp\left(-\int_{\vartheta_{k,i}^N}^{t}\gamma(r-\vartheta_{k,i}^N)\overline{\mathfrak F}(r)dr\right)\mu(d\gamma)\right)\nonumber\\
	&\hspace*{6cm} \times \sign(\Upsilon_k^N(\vartheta_{k,i}^N)-\Upsilon_k(\vartheta_{k,i}^N))\mathds{1}_{\vartheta_{k,i}^N\leq t }\,\,.
	\end{align}
	We establish the following Lemma.
	\begin{lemma}\label{TCL-lem7}
		For any $t\geq0$, as $N\to\infty$,
		\begin{equation*}
		\E\left[\Xi^N_{31}(t)^2\right]\to0.
		\end{equation*}
	\end{lemma}
\begin{proof}
	For all $k,i\in\mathbb{N},\,\E\left[\chi^N_{k,i}(t)\big|\mathcal{G}^N_{\vartheta_{k,i}^N}\right]=0.$ 
	Moreover, for any $k\neq\ell$,
	\begin{align}
	\E\left[\chi^N_{k,i}(t)\chi^N_{\ell,j}(t)\right]&=\E\left[\mathds{1}_{\vartheta_{k,i}^N<\vartheta_{\ell,j}^N}\E\left[\chi^N_{k,i}(t)\chi^N_{\ell,j}(t)\big|\mathcal{G}^N_{\vartheta_{\ell,j}^N},Q_k\right]\right]+\E\left[\mathds{1}_{\vartheta_{\ell,j}^N<\vartheta_{k,i}^N}\E\left[\chi^N_{k,i}(t)\chi^N_{\ell,j}(t)\big|\mathcal{G}^N_{\vartheta_{k,i}^N},Q_\ell\right]\right]\nonumber\\
	&=\E\left[\chi^N_{k,i}(t)\mathds{1}_{\vartheta_{k,i}^N<\vartheta_{\ell,j}^N}\E\left[\chi^N_{\ell,j}(t)\big|\mathcal{G}^N_{\vartheta_{\ell,j}^N},Q_k\right]\right]+\E\left[\chi^N_{\ell,j}(t)\mathds{1}_{\vartheta_{\ell,j}^N<\vartheta_{k,i}^N}\E\left[\chi^N_{k,i}(t)\big|\mathcal{G}^N_{\vartheta_{k,i}^N},Q_\ell\right]\right]\nonumber
	\\
	&=0\label{ch-1}
	\end{align}
	because $\chi^N_{\ell,j}$ and $Q_k$ are independent and similarly for $\chi^N_{k,i}$ and $Q_\ell$. 
	Hence, as $\big|\chi^N_{k,i}(t)\big|\leq\mathds{1}_{\vartheta_{k,i}^N\leq t }$,  from \eqref{ch-1} and exchangeability,
	\begin{align*}
	\E\left[\Xi^N_{31}(t)^2\right]&=\frac{1}{N}\sum_{k=1}^N\E\left[\left(\sum_{i\geq1}\chi^N_{k,i}(t)\right)^2\right]
	+\frac{N-1}{2}\sum_{1\leq k<\ell\leq N}\sum_{i,j\geq1}\E\left[\chi^N_{k,i}(t)\chi^N_{\ell,j}(t)\right]\\
	&\leq\frac{1}{N}\sum_{k=1}^N\E\left[\left(\Delta^N_k(t)\right)^2\right]\\
	&=\E\left[\left(\Delta^N_1(t)\right)^2\right].
	\end{align*}
	where $\Delta^N_k$ is given by \eqref{TCL-Delta}.
	
	However from Corollary~\ref{TCL-cor-1},
	\begin{align*}
	\E\left[\left(\Delta^N_1(t)\right)^2\right]&\leq2\left(\int_{0}^{t}\E\left[\left|\Upsilon^N_1(s)-\Upsilon_1(s)\right|\right]ds+\E\left[\left(\int_{0}^{t}\left|\Upsilon^N_1(s)-\Upsilon_1(s)\right|ds\right)^2\right]\right)\\
	&\leq2\left(\int_{0}^{t}\E\left[\left|\Upsilon^N_1(s)-\Upsilon_1(s)\right|\right]ds+t\int_{0}^{t}\E\left[\left|\Upsilon^N_1(s)-\Upsilon_1(s)\right|^2\right]ds\right)\\
	&\leq \frac{C_t}{\sqrt{N}}.
	\end{align*}
	Consequently
	\begin{equation*}
		\E\left[\Xi^N_{31}(t)^2\right]\leq\frac{C_t}{\sqrt{N}}.
	\end{equation*}
	\end{proof}
	Note that from \eqref{TCL-eq-eqc8-20}, 
	\begin{multline*}
	\Xi^N_4(t)-\Xi^N_{31}(t)\\=\frac{1}{\sqrt{N}}\sum_{k=1}^{N}\int_{0}^{t}\int_{\Upsilon_k^N(s^-)\wedge\Upsilon_k(s^-)}^{\Upsilon_k^N(s^-)\vee\Upsilon_k(s^-)}\int_{\bD}\gamma(t-s)\exp\left(-\int_{s}^{t}\gamma(r-s)\overline{\mathfrak F}(r)dr\right)\mu(d\gamma)\sign(\Upsilon_k^N(s^-)-\Upsilon_k(s^-))\overline Q_k(ds,du)\,.
	\end{multline*}
	Hence, as $(Q_k)_k$ are i.i.d.,  
	\begin{multline*}
	\E\left[\Bigg(\frac{1}{\sqrt{N}}\sum_{k=1}^{N}\int_{0}^{t}\int_{\Upsilon_k^N(s^-)\wedge\Upsilon_k(s^-)}^{\Upsilon_k^N(s^-)\vee\Upsilon_k(s^-)}\int_{\bD}\gamma(t-s)\exp\left(-\int_{s}^{t}\gamma(r-s)\overline{\mathfrak F}(r)dr\right)\mu(d\gamma)\right.\\\hspace*{8cm}\left.\times sign(\Upsilon_k^N(s^-)-\Upsilon_k(s^-))\overline Q_k(ds,du)\Bigg)^2\right]\\
	=\frac{1}{N}\sum_{k=1}^{N}\E\left[\left(\int_{0}^{t}\int_{\Upsilon_k^N(s^-)\wedge\Upsilon_k(s^-)}^{\Upsilon_k^N(s^-)\vee\Upsilon_k(s^-)}\int_{\bD}\gamma(t-s)\exp\left(-\int_{s}^{t}\gamma(r-s)\overline{\mathfrak F}(r)dr\right)\mu(d\gamma)\right.\right.\\\hspace*{8cm}\left.\times sign(\Upsilon_k^N(s^-)-\Upsilon_k(s^-))\overline Q_k(ds,du)\Bigg)^2\right]\\
	=\int_{0}^{t}\left(\int_{\bD}\gamma(t-s)\exp\left(-\int_{s}^{t}\gamma(r-s)\overline{\mathfrak F}(r)dr\right)\mu(d\gamma)\right)^2\E\left[\left|\overline{\Upsilon}^N(s)-\tilde{\Upsilon}^N(s)\right|\right]ds\,. 
	\end{multline*}
	From Corollary~\ref{inr} and Lemma~\ref{TCL-lem7} it follows that, as $N$ tends to $\infty,\,\Xi^N_4(t)\to0$ in probability for any $t\geq0$. This implies the finite dimensional convergence and as $(\Xi^N_4)_N$ is tight in $\bD$, it follows that as $N\to\infty,\Xi^N_4\to0$ in law in $\bD$ and hence in probability. 
	

\section{Proof of Tightness}\label{TCL-sec-Th}
In this section we prove the tightness properties stated in Lemma~\ref{TCL-tight-F-2}. 
We summarize the proof strategy in subsection~\ref{TCL-subb-sec-1} and then prove the supporting lemmas in subsection~\ref{TCL-sec-tight-r}.

\subsection{Proof of Lemma~\ref{TCL-tight-F-2}}\label{TCL-subb-sec-1}

To prove Lemma~\ref{TCL-tight-F-2}, since $\bD$ is separable, it suffices to establish that each component of $\big(\hat{\mathfrak S}^N,\hat{\mathfrak F}^N,\hat{\mathfrak S}^N_2,\hat{\mathfrak F}^N_2,\hat{\mathfrak S}^N_{1,0}\big)$ is tight in $\bD$. Moreover, from Lemma~\ref{TCL-cv-F2-G2} the pair $(\hat{\mathfrak S}^N_2,\hat{\mathfrak F}^N_2)$ is C-tight in $\bD^2$, since it converges in $\bD^2$ to a continuous limit. The tightness of the rest follows from the following Lemmas, which will be proved in the next subsection. 
\begin{lemma}\label{TCL-tight-F}
	$\hat{\mathfrak F}^N$ is $\bC$-tight.
\end{lemma}
\begin{lemma}\label{TCL-tight-F2}
	$\hat{\mathfrak S}_{1,2}^N$ is $\bC$-tight.
\end{lemma}
\begin{lemma}\label{TCL-tight-F3}
	$\hat{\mathfrak S}_{1,1}^N$ is $\bC$-tight.
\end{lemma}
\begin{lemma}\label{TCL-tight-F4}
	$\hat{\mathfrak S}_{1,0}^N$ is $\bC$-tight.
\end{lemma}


\subsection{Proofs of  Lemmas \ref{TCL-tight-F}--\ref{TCL-tight-F4}} \label{TCL-sec-tight-r}	

\subsubsection{Proof of Lemma~\ref{TCL-tight-F}}
Since $(\hat{\mathfrak F}^N_2)_N$ is $\bC-$tight in $\bD$, it suffices to prove that $(\hat{\mathfrak F}^N_1)_N$ is $\bC-$tight in $\bD$. By the expression in \eqref{TCL-eqc10}, 
that claim can be deduced from the tightness of the following sequence of processes
\begin{align*}
&\Xi^N_1(t)=\frac{1}{\sqrt{N}}\sum_{k=1}^{N}\int_{0}^{t}\int_{\bD^2}\int_{\Upsilon_k(s-)\wedge\Upsilon^N_k(s-)}^{\Upsilon_k(s-)\vee\Upsilon^N_k(s-)}\lambda(t-s)Q_k(ds,d\lambda,d\gamma,du).
\end{align*}

From Assumption~\ref{TCL-AS-lambda-1} and Lemma~\ref{TCL-lem-barlambda-inc-bound}, for any $0\leq s\leq t\leq T$, 
\begin{align}\label{TCL-eqc4}
|\Xi^N_1(t)-\Xi^N_1(s)|&\leq\frac{\lambda_*}{\sqrt{N}}\sum_{k=1}^{N}\int_{s}^{t}\int_{\bD^2}\int_{\Upsilon_k(r^-)\wedge\Upsilon_k^N(r^-)}^{\Upsilon_k(r^-)\vee\Upsilon_k^N(r^-)}Q_k(dr,d\lambda,d\gamma,du)\nonumber\\
&\quad+\frac{\ell(t-s)^\alpha}{\sqrt{N}}\sum_{k=1}^{N}\int_{0}^{s}\int_{\bD^2}\int_{\Upsilon_k(r^-)\wedge\Upsilon_k^N(r^-)}^{\Upsilon_k(r^-)\vee\Upsilon_k^N(r^-)}Q_k(dr,d\lambda,d\gamma,du)\nonumber\\
&\quad+\frac{\lambda_*}{\sqrt{N}}\sum_{j=1}^{\ell-1}\sum_{k=1}^{N}\int_{0}^{s}\int_{\bD^2}\int_{\Upsilon_k(r^-)\wedge\Upsilon_k^N(r^-)}^{\Upsilon_k(r^-)\vee\Upsilon_k^N(r^-)}\mathds{1}_{s-r<\xi^j\leq t-r}Q_k(dr,d\lambda,d\gamma,du).
\end{align}
By Markov's inequality, 
\begin{multline*}
\mathbb{P}\left(\frac{\ell\delta^\alpha}{\sqrt{N}}\sum_{k=1}^{N}\int_{0}^{T}\int_{\bD^2}\int_{\Upsilon_k(r^-)\wedge\Upsilon_k^N(r^-)}^{\Upsilon_k(r^-)\vee\Upsilon_k^N(r^-)} Q_k(dr,d\lambda,d\gamma,du)\geq\theta\right)\\
\leq\frac{1}{\theta^2}\E\left[\left(\frac{\ell\delta^\alpha}{\sqrt{N}}\sum_{k=1}^{N}\int_{0}^{T}\int_{\bD^2}\int_{\Upsilon_k(r^-)\wedge\Upsilon_k^N(r^-)}^{\Upsilon_k(r^-)\vee\Upsilon_k^N(r^-)} Q_k(dr,d\lambda,d\gamma,du)\right)^2\right].
\end{multline*}
However, by exchangeability and H\"older's inequality, it follows that 
\begin{multline*}
\E\left[\left(\frac{\ell\delta^\alpha}{\sqrt{N}}\sum_{k=1}^{N}\int_{0}^{T}\int_{\bD^2}\int_{\Upsilon_k(r^-)\wedge\Upsilon_k^N(r^-)}^{\Upsilon_k(r^-)\vee\Upsilon_k^N(r^-)} Q_k(dr,d\lambda,d\gamma,du)\right)^2\right]\\
\begin{aligned}
&\leq2\E\left[\left(\frac{\ell\delta^\alpha}{\sqrt{N}}\sum_{k=1}^{N}\int_{0}^{T}\int_{\bD^2}\int_{\Upsilon_k(r^-)\wedge\Upsilon_k^N(r^-)}^{\Upsilon_k(r^-)\vee\Upsilon_k^N(r^-)} \overline Q_k(dr,d\lambda,d\gamma,du)\right)^2\right]\\
& \quad +2\E\left[\left(\frac{\ell\delta^\alpha}{\sqrt{N}}\sum_{k=1}^{N}\int_{0}^{T}\left|\Upsilon^N_k(r)-\Upsilon_k(r)\right|dr\right)^2\right]\\
&\leq2\ell^2\delta^{2\alpha}\int_{0}^{T}\E\left[\left|\Upsilon_1^N(r)-\Upsilon_1(r)\right|\right]dr+2\ell^2\delta^{2\alpha}N\int_{0}^{T}\E\left[\left|\Upsilon_1^N(r)-\Upsilon_1(r)\right|^2\right]dr\\
&\leq2\ell^2\delta^{2\alpha}\int_{0}^{T}\E\left[\left|\Upsilon_1^N(r)-\Upsilon_1(r)\right|\right]dr+2C_T\ell^2\delta^{2\alpha}\,,
\end{aligned}
\end{multline*}
where the second term in the last line follows from applying \eqref{fluc-2} with $k=2$. Note that 
 from \eqref{eqgam} the first term in the last line tends to 0 as $N\to\infty$, while the second term is  independent of $N$.

Then, thanks to $\alpha>\frac{1}{2}$ (Assumption~\ref{TCL-AS-lambda-2}), 
\begin{equation}
\lim_{\delta\to0}\limsup_{N\to\infty}\sup_{0\leq t\leq T}\frac{1}{\delta}\mathbb{P}\left(\sup_{0\le v\le \delta}\frac{\ell v^\alpha}{\sqrt{N}}\sum_{k=1}^{N}\int_{0}^{T}\int_{\bD^2}\int_{\Upsilon_k(r^-)\wedge\Upsilon_k^N(r^-)}^{\Upsilon_k(r^-)\vee\Upsilon_k^N(r^-)} Q_k(dr,d\lambda,d\gamma,du)\geq\theta\right)=0.\label{eqc5}
\end{equation}
On the other hand, for all $0<t\leq T$, since $(Q_k)_k$ are i.i.d, by H\"older's inequality and exchangeability, we obtain 
\begin{align*}
&\mathbb{P}\left(\frac{\lambda_*}{\sqrt{N}}\sum_{k=1}^{N}\int_{t}^{t+\delta}\int_{\bD^2}\int_{\Upsilon_k(r^-)\wedge\Upsilon_k^N(r^-)}^{\Upsilon_k(r^-)\vee\Upsilon_k^N(r^-)}Q_k(dr,d\lambda,d\gamma,du)>\theta\right)\\
&\leq\frac{1}{\theta^2}\E\left[\left(\frac{\lambda_*}{\sqrt{N}}\sum_{k=1}^{N}\int_{t}^{t+\delta}\int_{\bD^2}\int_{\Upsilon_k(r^-)\wedge\Upsilon_k^N(r^-)}^{\Upsilon_k(r^-)\vee\Upsilon_k^N(r^-)}Q_k(dr,d\lambda,d\gamma,du)\right)^2\right]\nonumber\\
&\leq\frac{2\lambda_*^2}{\theta^2}\left\{\E\left[\frac{1}{N}\sum_{k=1}^{N}\left(\int_{t}^{t+\delta}\int_{\bD^2}\int_{\Upsilon_k(r^-)\wedge\Upsilon_k^N(r^-)}^{\Upsilon_k(r^-)\vee\Upsilon_k^N(r^-)}\overline{Q}_k(dr,d\lambda,d\gamma,du)\right)^2\right]\right.\nonumber
\\&\hspace*{5cm}\left.+\E\left[\left(\frac{1}{\sqrt N}\sum_{k=1}^{N}\int_{t}^{t+\delta}\left|\Upsilon_k^N(r)-\Upsilon_k(r)\right|dr\right)^2\right]\right\}\nonumber\\
&\leq\frac{2\lambda_*^2}{\theta^2}\left\{\E\left[\int_{t}^{t+\delta}\left|\Upsilon_1^N(r)-\Upsilon_1(r)\right|dr\right]+N\E\left[\left(\int_{t}^{t+\delta}\left|\Upsilon_1^N(r)-\Upsilon_1(r)\right|dr\right)^2\right]\right\}\nonumber\\
&\leq\frac{2\lambda_*^2}{\theta^2}\left\{\int_0^{T+\delta}\E\left[\left|\Upsilon_1^N(r)-\Upsilon_1(r)\right|\right]dr+C_T\delta^2\right\},
\end{align*}
where the last line follows from H\"older's inequality and applying \eqref{fluc-2} with $k=2$.
Note that  the second term is independent of $N$ and from \eqref{eqgam} the first term tends to 0 as $N\to\infty$.
Consequently, 
\begin{equation}\label{TCL-in-eqc6}
\lim_{\delta\to0}\limsup_{N\to\infty}\sup_{0\leq t\leq T}\frac{1}{\delta}\mathbb{P}\left(\sup_{0\leq v\leq\delta}\frac{\lambda_*}{\sqrt{N}}\sum_{k=1}^{N}\int_{t}^{t+v}\int_{\bD^2}\int_{\Upsilon_k(r^-)\wedge\Upsilon_k^N(r^-)}^{\Upsilon_k(r^-)\vee\Upsilon_k^N(r^-)}Q_k(dr,d\lambda,d\gamma,du)>\epsilon\right)=0.
\end{equation}
Moreover, by Markov's inequality,
\begin{multline*}
\mathbb{P}\left(\frac{1}{\sqrt{N}}\sum_{k=1}^{N}\sum_{j=1}^{\ell-1}\int_{0}^{t}\int_{\bD^2}\int_{\Upsilon_k(r^-)\wedge\Upsilon_k^N(r^-)}^{\Upsilon_k(r^-)\vee\Upsilon_k^N(r^-)} \mathds{1}_{t-r<\xi^j\leq t+\delta-r}Q_k(dr,d\lambda,d\gamma,du)\geq\theta\right)\\
\begin{aligned}
&\leq\frac{1}{\theta^2}\E\left[\left(\frac{1}{\sqrt{N}}\sum_{k=1}^{N}\sum_{j=1}^{\ell-1}\int_{0}^{t}\int_{\bD^2}\int_{\Upsilon_k(r^-)\wedge\Upsilon_k^N(r^-)}^{\Upsilon_k(r^-)\vee\Upsilon_k^N(r^-)} \mathds{1}_{t-r<\xi^j\leq t+\delta-r}Q_k(dr,d\lambda,d\gamma,du)\right)^2\right]\\
&\leq\frac{2}{\theta^2}\E\left[\left(\frac{1}{\sqrt{N}}\sum_{k=1}^{N}\sum_{j=1}^{\ell-1}\int_{0}^{t}\int_{\bD^2}\int_{\Upsilon_k(r^-)\wedge\Upsilon_k^N(r^-)}^{\Upsilon_k(r^-)\vee\Upsilon_k^N(r^-)} \mathds{1}_{t-r<\xi^j\leq t+\delta-r}\overline{Q}_k(dr,d\lambda,d\gamma,du)\right)^2\right]\\
&\quad+\frac{2}{\theta^2}\E\left[\left(\frac{1}{\sqrt{N}}\sum_{j=1}^{\ell-1}\sum_{k=1}^{N}\int_{0}^{t}\left(F_j(s+\delta-r)-F_j(t-r)\right)\left|\Upsilon_k^N(r)-\Upsilon_k(r)\right|dr\right)^2\right].
\end{aligned}
\end{multline*}
Consequently, since $(Q_k)_k$ are i.i.d.,  from \eqref{hypF} and H\"older's inequality and from exchangeability, we obtain 
\begin{multline*}
\mathbb{P}\left(\frac{1}{\sqrt{N}}\sum_{k=1}^{N}\sum_{j=1}^{\ell-1}\int_{0}^{t}\int_{\bD^2}\int_{\Upsilon_k(r^-)\wedge\Upsilon_k^N(r^-)}^{\Upsilon_k(r^-)\vee\Upsilon_k^N(r^-)} \mathds{1}_{t-r<\xi^j\leq t+\delta-r}Q_k(dr,d\lambda,d\gamma,du)\geq\theta\right)\\
\begin{aligned}
&\leq\frac{2}{N\theta^2}\sum_{k=1}^{N}\E\left[\left(\sum_{j=1}^{\ell-1}\int_{0}^{t}\int_{\bD^2}\int_{\Upsilon_k(r^-)\wedge\Upsilon_k^N(r^-)}^{\Upsilon_k(r^-)\vee\Upsilon_k^N(r^-)} \mathds{1}_{t-r<\xi^j\leq t+\delta-r}\overline{Q}_k(dr,d\lambda,d\gamma,du)\right)^2\right]\\
&\quad+\frac{2\ell^2\delta^{2\rho}}{\theta^2}\sum_{k=1}^{N}\E\left[\left(\int_{0}^{t}\left|\Upsilon_k^N(r)-\Upsilon_k(r)\right|dr\right)^2\right]\\
&\leq\frac{2}{\theta^2}\sum_{j=1}^{\ell-1}\int_{0}^{t}\left(F_j(s+\delta-r)-F_j(t-r)\right)\E\left[\left|\Upsilon_1^N(r)-\Upsilon_1(r)\right|\right]dr
+\frac{2\ell^2\delta^{2\rho}}{\theta^2}N\int_{0}^{t}\E\left[\left|\Upsilon_1^N(r)-\Upsilon_1(r)\right|^2\right]dr\\
&\leq\frac{2\ell\delta^\rho}{\theta^2}\int_{0}^{t}\E\left[\left|\Upsilon_1^N(r)-\Upsilon_1(r)\right|\right]dr+\frac{2C_T\ell^2\delta^{2\rho}}{\theta^2}.
\end{aligned}
\end{multline*}
where the last line follows from H\"older's inequality and applying \eqref{fluc-2} with $k=2$.
Note that the second term is independent of $N$ and from \eqref{eqgam} the first term tend to 0 as $N\to\infty$.

Therefore, as $\rho>\frac{1}{2}$ (Assumption~\ref{TCL-AS-lambda-2}),
\begin{equation}\label{TCL-eqc7'}
\lim_{\delta\to0}\limsup_{N\to\infty}\sup_{0\leq t\leq T}\frac{1}{\delta}\mathbb{P}\left(\sup_{0\le v\le \delta}\frac{1}{\sqrt{N}}\sum_{j=1}^{\ell-1}\sum_{k=1}^{N}\int_{0}^{t}\int_{\bD^2}\int_{\Upsilon_k(r^-)\wedge\Upsilon_k^N(r^-)}^{\Upsilon_k(r^-)\vee\Upsilon_k^N(r^-)}\mathds{1}_{t-r<\xi^j\leq t+v-r}Q_k(dr,d\lambda,d\gamma,du)\geq\theta\right)=0.
\end{equation}
Thus, from \eqref{TCL-eqc4}, \eqref{eqc5}, \eqref{TCL-in-eqc6}  and \eqref{TCL-eqc7'}, we deduce that 
\begin{equation*}
\lim_{\delta\to0}\limsup_N\sup_{0\leq t\leq T}\frac{1}{\delta}\mathbb{P}\left(\sup_{0\le v\le \delta}|\Xi^N_1(t+v)-\Xi^N_1(t)|\geq\theta\right)=0,
\end{equation*}
and thanks to Lemme~\ref{TCL-Lem-20}, $(\Xi^N_1)_N$ is $\bC-$tight in $\bD$.

\subsubsection{Proof of Lemma~\ref{TCL-tight-F2}}
We recall that 
\begin{equation}\label{eqhG}
\hat{\mathfrak S}^N_{1,2}(t):=\frac{1}{\sqrt{N}}\sum_{k=1}^{N}\int_{0}^{t}\int_{\bD^2}\int_{\Upsilon_k(s^-)\wedge\Upsilon_k^N(s^-)}^{\Upsilon_k(s^-)\vee\Upsilon_k^N(s^-)}\gamma(t-s)\mathds{1}_{P_k\left(s,t,\gamma,\overline{\mathfrak F}\right)=0}sign(\Upsilon_k^N(s^-)-\Upsilon_k(s^-))Q_k(ds,d\lambda,d\gamma,du).
\end{equation}
\begin{lemma}
	$\hat{\mathfrak S}^N_{1,2}$ is $\bC-$tight in $\bD$.
\end{lemma}
\begin{proof}
	By the expression in \eqref{eqhG}, tightness of the processes $(\hat{\mathfrak S}^N_{1,2})_N$ can be deduced from the tightness of the following processes
	
	\begin{equation}\Xi^N_3(t)=\frac{1}{\sqrt{N}}\sum_{k=1}^{N}\int_{0}^{t}\int_{\bD^2}\int_{\Upsilon_k(s^-)\wedge\Upsilon_k^N(s^-)}^{\Upsilon_k(s^-)\vee\Upsilon_k^N(s^-)}\gamma(t-s)\mathds{1}_{P_k\left(s,t,\gamma,\overline{\mathfrak F}\right)=0}Q_k(ds,d\lambda,d\gamma,du).
	\end{equation}	
	
	For all $0\leq s\leq t\leq T$,
	\begin{align*}
	&\left|\Xi^N_3(t)-\Xi^N_3(s)\right|\\
	&\leq\frac{1}{\sqrt{N}}\sum_{k=1}^{N}\int_{s}^{t}\int_{\bD^2}\int_{\Upsilon_k(r^-)\wedge\Upsilon_k^N(r^-)}^{\Upsilon_k(r^-)\vee\Upsilon_k^N(r^-)}Q_k(dr,d\lambda,d\gamma,du)\\
	&+\frac{\ell(t-s)^\alpha}{\sqrt{N}}\sum_{k=1}^{N}\int_{0}^{s}\int_{\bD^2}\int_{\Upsilon_k(r^-)\wedge\Upsilon_k^N(r^-)}^{\Upsilon_k(r^-)\vee\Upsilon_k^N(r^-)}Q_k(dr,d\lambda,d\gamma,du)\\
	&+\frac{1}{\sqrt{N}}\sum_{j=1}^{\ell-1}\sum_{k=1}^{N}\int_{0}^{s}\int_{\bD^2}\int_{\Upsilon_k(r^-)\wedge\Upsilon_k^N(r^-)}^{\Upsilon_k(r^-)\vee\Upsilon_k^N(r^-)}\mathds{1}_{s-r<\zeta^j\leq t-r}Q_k(dr,d\lambda,d\gamma,du)\\
	&+\frac{1}{\sqrt{N}}\sum_{k=1}^{N}\int_{0}^{s}\int_{\bD^2}\int_{\Upsilon_k(r^-)\wedge\Upsilon_k^N(r^-)}^{\Upsilon_k(r^-)\vee\Upsilon_k^N(r^-)}\left(\mathds{1}_{P_k\left(r,s,\gamma,\overline{\mathfrak F}\right)=0}-\mathds{1}_{P_k\left(r,t,\gamma,\overline{\mathfrak F}\right)=0}\right)Q_k(dr,d\lambda,d\gamma,du),
	\end{align*}
	where we use Lemme~\ref{TCL-lem-barlambda-inc-bound}.
	
	From \eqref{eqc5}, \eqref{TCL-in-eqc6}  and \eqref{TCL-eqc7'} it remains to prove that, as $\delta\to0,$
	\begin{align}\label{TCL-A-7-eq22}
	&\limsup_N\sup_{0\leq t\leq T}\frac{1}{\delta}\mathbb{P}\left(\frac{1}{\sqrt{N}}\sum_{k=1}^{N}\int_{0}^{t}\int_{\bD^2}\int_{\Upsilon_k(r^-)\wedge\Upsilon_k^N(r^-)}^{\Upsilon_k(r^-)\vee\Upsilon_k^N(r^-)}\left(\mathds{1}_{P_k\left(r,t,\gamma,\overline{\mathfrak F}\right)=0}-\mathds{1}_{P_k\left(r,t+\delta,\gamma,\overline{\mathfrak F}\right)=0}\right)Q_k(dr,d\lambda,d\gamma,du)\geq\epsilon\right)\nonumber\\&\hspace*{12cm}\to0.
	\end{align}
	We recall that for $0 \le t\le T$, 
	\[\mathcal{F}^N_t=\sigma\left\{(\lambda_{k,i},\gamma_{k,i})_{\substack{1\leq k\leq N\\1\leq i\leq A^N_k(t)}},\,(Q_k|_{[0,t]})_{1\leq k\leq N}\right\}.\]  
	We set
	\[D^N_k(t)=\int_{0}^{t}\int_{\bD^2}\int_{\Upsilon_k(v^-)\wedge\Upsilon_k^N(v^-)}^{\Upsilon_k(v^-)\vee\Upsilon_k^N(v^-)}Q_k(dv,d\lambda,d\gamma,du),\]
	and $\overline{D}^N_k$ is given by the same expression as $D^N_k$ but where we replace $Q_k$ by its compensated measure.
	
	Since
	\[0\leq\mathds{1}_{P_k\left(v,t,\gamma,\overline{\mathfrak F}\right)=0}-\mathds{1}_{P_k\left(v,t+\delta,\gamma,\overline{\mathfrak F}\right)=0}\leq \int_{t}^{t+\delta}\int_{0}^{\lambda_*}Q_k(du_1,du_2)=C_k(t,t+\delta),\]
	by exchangeability and the fact that $(Q_k)_k$ are independent and $C_k(t,t+\delta)$ and $Q_k|_{[0,t]}$ are independent, it follows that 
	\begin{multline}\label{TCL-A-7-1}
	\E\left[\left(\frac{1}{\sqrt{N}}\sum_{k=1}^{N}\int_{0}^{t}\int_{\bD^2}\int_{\Upsilon_k(v^-)\wedge\Upsilon_k^N(v^-)}^{\Upsilon_k(v^-)\vee\Upsilon_k^N(v^-)}\left(\mathds{1}_{P_k\left(v,t,\gamma,\overline{\mathfrak F}\right)=0}-\mathds{1}_{P_k\left(v,t+\delta,\gamma,\overline{\mathfrak F}\right)=0}\right)Q_k(dv,d\lambda,d\gamma,du)\right)^2\right]\\
	\begin{aligned}
	&\leq\E\left[\left(\frac{1}{\sqrt{N}}\sum_{k=1}^{N}C_k(t,t+\delta)D^N_k(t)\right)^2\right]\\
	&=\E\left[C_1^2(t,t+\delta)\left(D^N_1(t)\right)^2\right]+(N-1)\E\left[C_1(t,t+\delta)C_2(t,t+\delta)D^N_1(t)D^N_2(t)\right]\\
	&=\E\left[C_1^2(t,t+\delta)\right]\E\left[\left(D^N_1(t)\right)^2\right]+(N-1)\E\left[C_1(t,t+\delta)\right]\E\left[C_2(t,t+\delta)\right]\E\left[D^N_1(t)D^N_2(t)\right]\\
	&\leq2(\lambda_*\delta+\lambda_*^2\delta^2)\left(\E\left[\left(\overline{D}^N_1(t)\right)^2\right]+\E\left[\left(\int_{0}^{t}|\Upsilon^N_1(v)-\Upsilon_1(v)|dv\right)^2\right]\right)+\lambda_*^2\delta^2N\E\left[D^N_1(t)D^N_2(t)\right]\\
	&\leq2(\lambda_*\delta+\lambda_*^2\delta^2)\left(\int_{0}^{t}\E\left[|\Upsilon_1(v)-\Upsilon_1^N(v)|\right]dv+T\int_{0}^{t}\E\left[|\Upsilon^N_1(v)-\Upsilon_1(v)|^2\right]dv\right) \\
	& \qquad +\lambda_*^2\delta^2N\E\left[D^N_1(t)D^N_2(t)\right],
	\end{aligned}
	\end{multline}
	where we use H\"older's inequality in the last inequality.
	 
	At this stage, we admit the following inequality holds (and we will show this immediately below):
	\begin{equation}\label{TCL-eq-in-1}
	N\E\left[D^N_1(t)D^N_2(t)\right]\leq C,
	\end{equation}
	for some $C>0$ independent of $N$.
	
	Consequently from Corollary~\ref{inr}, from \eqref{TCL-eq-in-1} and \eqref{TCL-A-7-1}, \eqref{TCL-A-7-eq22} follows.
\end{proof}

Now we establish the inequality \eqref{TCL-eq-in-1}. To do this, we define  the following process  for each $k\in\mathbb{N}$,
\[B^N_k(t)=\int_{0}^{t}\int_{0}^{\varTheta^N_k(r^-)}Q_k(dr,du),\]
where
\[\varTheta^N_k(t)=\gamma_{k,B^N_k(t)}(\vartheta^N_k(t))\overline{\mathfrak F}^N_{(1,2)}(t),\]
and $\vartheta_k^N$ is defined in the same manner as $\varsigma^N_1$ with $ B^N_k$ instead of $ A^N_1$ in \eqref{TCL-t-1}.
\begin{lemma}
	\label{TCL-A-9-lem_inq}For $k\in\mathbb{N}$ and $T\geq0$, 
	\begin{equation}
	\mathbb{E}\left[\sup_{t\in[0,T]}\left|A^N_k(t)-B^N_k(t)\right|\right]\leq\int_{0}^{T}\mathbb{E}\Big[\left|\Upsilon^N_k(t)-\varTheta^N_k(t)\right|\Big]dt=:\delta^N(T)\label{TCL-A-9-eqA}
	\end{equation}
	and 
	\begin{equation*}
	\mathbb{E}\left[\sup_{t\in[0,T]}\left|\varsigma^N_k(t)-\vartheta^N_k(t)\right|\right]\leq T\delta^N(T).
	\end{equation*}
	Moreover, 
	\begin{equation}\delta^N(T)\leq\frac{\E\left[Y\right]}{N}T\exp(2\lambda^*T).\label{TCL-A-9-eqdelta}\end{equation}
\end{lemma}
\begin{proof}
	Since \[\left|A^N_k(t)-B^N_k(t)\right|=\int_{0}^{t}\int_{0}^{+\infty}\mathds{1}_{\min\left(\Upsilon_k^N(r^-),\varTheta^N_k(r^-)\right)< u\leq\max\left(\Upsilon_k^N(r^-),\varTheta^N_k(r^-)\right)}Q_k(du,dr),\]
	we have
	\begin{equation*}
	\mathbb{E}\left[\sup_{t\in[0,T]}\left|A^N_k(t)-B^N_k(t)\right|\right]\leq\int_{0}^{T}\mathbb{E}\Big[\left|\Upsilon^N_k(t)-\varTheta_k^N(t)\right|\Big]dt=\delta^N(T).
	\end{equation*}
	We recall that \[\Upsilon^N_k(t)=\gamma_{k,A_k^N(t)}(\varsigma_{k}^N(t))\overline{\mathfrak{F}}^N(t) \quad \text{ and } \quad \varTheta^N_k(t)=\gamma_{k,B^N_k(t)}(\vartheta^N_k(t))\overline{\mathfrak F}^N_{(1,2)}(t).\]
	However, since $\gamma_{k,i}\leq1$ and  $0\leq\overline{\mathfrak{F}}^N(t),\,\overline{\mathfrak{F}}^N_{(1,2)}(t)\leq\lambda^\ast$, we obtain 
	\begin{align}\label{TCL-A-9-EL0}
	\mathbb{E}\Big[\left|\Upsilon^N_k(t)-\varTheta_k^N(t)\right|\Big]&\leq\mathbb{E}\left[\left|\Upsilon^N_k(t)-\varTheta_k^N(t)\right|\mathds{1}_{A^N_k(t)=B^N_k(t),\varsigma_{k}(t)=\vartheta^N_k(t)}\right]+\nonumber\\&\hspace*{3cm}\lambda^*\mathbb{P}\left(A^N_k(t)\neq B^N_k(t)\text{ or }\varsigma_{k}^N(t)\neq \vartheta^{N}_k(t)\right).
	\end{align}
	On the other hand, as $\gamma_{k,i}\leq1$, from \eqref{TCL-eq1}, we have 
	\begin{align}
	\mathbb{E}\left[\left|\Upsilon^N_k(t)-\varTheta_k^N(t)\right|\mathds{1}_{A^N_k(t)=B_k^N(t),\varsigma_{k}^N(t)=\vartheta^N_k(t)}\right]&\leq\mathbb{E}\left[\left|\overline{\mathfrak F}^N(t)-\overline{\mathfrak F}^N_{(1,2)}(t)\right|\right]\nonumber\\
	&\leq\frac{1}{N}\E\left[Y\right]\label{TCL-A-9-EL1}.
	\end{align}
	Moreover, since \[\left\{A^N_k(t)\neq B_k^N(t)\text{ or }\varsigma_{k}^N(t)\neq \vartheta^N_k(t)\right\}\subset\left\{\sup_{r\in[0,t]}|A^N_k(r)-B_k^N(r)|\geq1\right\},\]
	we have 
	\[\mathbb{P}\left(A_k^N(t)\neq B_k^N(t)\text{ or }\varsigma_{k}^N(t)\neq \vartheta_{k}^N(t)\right)\leq\E\left[\sup_{r\in[0,t]}|A^N_k(r)-B_k^N(r)|\right]\leq\delta^N(t).\]
	Thus, from \eqref{TCL-A-9-EL0} and \eqref{TCL-A-9-EL1}, we have 
	\begin{equation*}
	\mathbb{E}\Big[\left|\Upsilon^N_k(t)-\varTheta_k(t)\right|\Big]\leq\frac{\E\left[Y\right]}{N}+\lambda^*\delta^N(t).
	\end{equation*}
	Hence,  from \eqref{TCL-A-9-eqA}, we deduce that
	\[\delta^N(T)\leq\frac{\E\left[Y\right]}{N}T+\lambda^*\int_{0}^{T}\delta^N(t)dt,\]
	and by Gronwall's lemma, it follows that 
	\begin{equation*}\delta^N(T)\leq\frac{\E\left[Y\right]}{N}T\exp(\lambda^*T).\end{equation*}
	Moreover, 
	\begin{align*}
	\mathbb{E}\left[\sup_{t\in[0,T]}\left|\varsigma^N_k(t)-\vartheta_k^N(t)\right|\right]&=\mathbb{E}\left[\mathds{1}_{\{\exists t\in[0,T],\varsigma^N_k(t)\neq\vartheta_k^N(t)\}}\sup_{t\in[0,T]}\left|\varsigma^N_k(t)-\vartheta_k^N(t)\right|\right]\\
	&\leq T\mathbb{P}\left(\exists t\in[0,T],\varsigma^N_k(t)\neq\vartheta_k^N(t)\right)\\
	&=T\mathbb{P}\left(\sup_{t\in[0,T]}\left|A^N_k(t)-B_k^N(t)\right|\neq0\right)\\
	&\leq T\E\left[\sup_{t\in[0,T]}\left|A^N_k(t)-B_k^N(t)\right|\right]\\
	&\leq T\delta^N(T).
	\end{align*}
	This concludes the proof of the lemma.
\end{proof}
The same proof of Proposition~\ref{prop-c1} gives the following Lemma, where we replace $A_k$ by $B^N_k$ and $\Upsilon_k$ by $\varTheta^N_k$ and using Lemma~\ref{TCL-A-9-eqdelta} instead of Lemma~\ref{lem_inq}.

\begin{lemma}\label{TCL-A-9-lem-inq2}
	For all $p\in\mathbb{N},\,t\in[0,T]$,
	\begin{equation*}
	\mathbb{P}\left((\varsigma_{k'}^N(s))_{s\in[0,t]}\neq(\vartheta_{k'}^N(s))_{s\in[0,t]},\,\forall k'=1,\cdots, p\right)\leq\frac{C_T}{N^p},
	\end{equation*}
	and
	\[\E\left[\left|\Upsilon^N_k(t)-\varTheta_k^N(t)\right|^p\right]\leq\frac{C_T}{N^p}.\]	
\end{lemma}
We can now establish the following lemma for the inequality \eqref{TCL-eq-in-1}. 
\begin{lemma}There exists $C>0$ such that the inequality \eqref{TCL-eq-in-1} holds for all $N$. 
\end{lemma}
\begin{proof}
	By exchangeability and the fact that $Q_1$ and $Q_2$ are independent, 
	\begin{align}
	\E\left[D^N_1(t)D^N_2(t)\right]&=\E\left[\overline{D}^N_1(t)\overline{D}^N_2(t)\right]+2\int_{0}^{t}\E\left[|\Upsilon_1^N(v)-\Upsilon_1(v)|\overline{D}^N_2(t)\right]dv
	\nonumber
	\\&\hspace{3cm}+\int_{0}^{t}\int_{0}^{t}\E\left[|\Upsilon_1^N(v)-\Upsilon_1(v)||\Upsilon_2^N(r)-\Upsilon_2(r)|\right]dvdr\nonumber\\
	&=2\int_{0}^{t}\E\left[|\Upsilon_1^N(v)-\Upsilon_1(v)|\overline{D}^N_2(v)\right]dv \nonumber \\
	& \quad +\int_{0}^{t}\int_{0}^{t}\E\left[|\Upsilon_1^N(v)-\Upsilon_1(v)||\Upsilon_2^N(r)-\Upsilon_2(r)|\right]dvdr,\label{TCL-A-9-eq1}
	\end{align}
	where the first term in the last equality follows from the fact that 
	\begin{align*}
		\E\left[\overline{D}^N_2(t)-\overline{D}^N_2(v)\Big|\mathcal{F}^N_v\right]&=\E\left[\int_{v}^{t}\int_{\bD^2}\int_{\Upsilon_2(r^-)\wedge\Upsilon_2^N(r^-)}^{\Upsilon_2(r^-)\vee\Upsilon_2^N(r^-)}\overline{Q}_2(dr,d\lambda,d\gamma,du)\Big|\mathcal{F}^N_v\right]=0.
	\end{align*}
	As 
	\[|\Upsilon_1^N(v)-\Upsilon_1(v)|\leq|\Upsilon_1^N(v)-\varTheta_1^N(v)|+|\varTheta_1^N(v)-\Upsilon_1(v)|,\]
	by H\"older's inequality, it follows that 
	\begin{align*}
		\E\left[|\Upsilon_1^N(v)-\Upsilon_1(v)||\overline{D}^N_2(v)|\right]&\leq\E\left[|\Upsilon_1^N(v)-\varTheta_1^N(v)||\overline{D}^N_2(v)|\right]+\E\left[|\varTheta_1^N(v)-\Upsilon_1(v)||\overline{D}^N_2(v)|\right]\\
		&\leq\left(\E\left[|\Upsilon_1^N(v)-\varTheta_1^N(v)|^2\right]\right)^{1/2}\left(\E\left[(\overline{D}^N_2(v))^2\right]\right)^{1/2}+\E\left[|\varTheta_1^N(v)-\Upsilon_1(v)||\overline{D}^N_2(v)|\right].
	\end{align*}
	Moreover, as
	\begin{align*}
		|\overline{D}^N_2(v)|&\leq |D^N_2(v)|+\int_0^v|\Upsilon_2^N(r)-\Upsilon_2(r)|dr\\
		&\leq|D^N_2(v)-\tilde{D}^N_2(v)|+\tilde{D}^N_2(v)+\int_0^v|\Upsilon_2^N(r)-\Upsilon_2(r)|dr,
	\end{align*}
		where
	\[\tilde{D}^N_k(t)=\int_{0}^{t}\int_{\bD^2}\int_{\Upsilon_k(v^-)\wedge\varTheta_k^N(v^-)}^{\Upsilon_k(v^-)\vee\varTheta_k^N(v^-)}Q_k(dv,d\lambda,d\gamma,du),\]
	it follows that,
	\begin{align*}
	\E\left[|\Upsilon_1^N(v)-\Upsilon_1(v)||\overline{D}^N_2(v)|\right]
	&\leq\left(\E\left[|\Upsilon_1^N(v)-\varTheta_1^N(v)|^2\right]\right)^{1/2}\left(\E\left[(\overline{D}^N_2(v))^2\right]\right)^{1/2}+\lambda_*\E\left[|D^N_2(v)-\tilde{D}^N_2(v)|\right]\\
	&\quad+\E\left[|\varTheta_1^N(v)-\Upsilon_1(v)|\E\left[\tilde{D}^N_2(v)\Big|\overline{\mathfrak F}^N_{(1,2)},Q_1\right]\right]\\&\qquad\quad+\int_{0}^{v}\E\left[|\varTheta_1^N(v)-\Upsilon_1(v)||\Upsilon_2^N(r)-\Upsilon_2(r)|\right]dr,
	\end{align*}
	where we use the fact that $|\varTheta_1^N(v)-\Upsilon_1(v)|\leq\lambda_*.$
	
	Consequently from Lemma~\ref{TCL-A-9-lem-inq2} and the fact that $Q_2$ and $\big(Q_1,\overline{\mathfrak F}^N_{(1,2)}\big)$ are independent, it follows that
	\begin{multline*}
	\E\left[|\Upsilon_1^N(v)-\Upsilon_1(v)||\overline{D}^N_2(v)|\right]\leq\frac{C_T}{N^{5/4}}+\lambda_*\E\left[\int_{0}^{v}\int_{\bD^2}\int_{\Upsilon_2^N(r^-)\wedge\varTheta_2^N(r^-)}^{\Upsilon_2^N(r^-)\vee\varTheta_2^N(r^-)}Q_2(dr,d\lambda,d\gamma,du)\right]\\
	\begin{aligned}
    & \hspace{2cm} +\int_0^v\E\left[|\varTheta_1^N(v)-\Upsilon_1(v)||\varTheta_2^N(r)-\Upsilon_2(r)|\right]dr+\int_{0}^{v}\E\left[|\varTheta_1^N(v)-\Upsilon_1(v)||\Upsilon_2^N(r)-\Upsilon_2(r)|\right]dr\\
	&= \frac{C_T}{N^{5/4}}+\lambda_*\int_{0}^{v}\E\left[|\varTheta_2^N(r)-\Upsilon_2^N(r)|\right]dr+\int_0^v\E\left[|\varTheta_1^N(v)-\Upsilon_1(v)||\varTheta_2^N(r)-\Upsilon_2(r)|\right]dr\\
	&\hspace*{2cm}+\int_{0}^{v}\E\left[|\varTheta_1^N(v)-\Upsilon_1(v)||\Upsilon_2^N(r)-\Upsilon_2(r)|\right]dr\\
	&\leq\frac{2C_T}{N^{5/4}}+\frac{C_T(\lambda_*+1)}{N},
	\end{aligned}
	\end{multline*}
	where the last line follows from Lemma~\ref{TCL-A-9-lem-inq2} and Corollary~\ref{inr}.
	
	Hence, from \eqref{TCL-A-9-eq1} and Corollary~\ref{inr} it follows that the inequality \eqref{TCL-eq-in-1} holds.
\end{proof}

\subsubsection{Proof of Lemma~\ref{TCL-tight-F3}}\label{TCL-sub-sec-tight-2}
We recall that 
\begin{equation}
\hat{\mathfrak S}^N_{1,1}(t)=\frac{1}{\sqrt{N}}\sum_{k=1}^{N}\int_{0}^{t}\int_{\bD^2}\int_0^{\Upsilon_k^N(r^-)}\gamma(t-r)\left(\mathds{1}_{P_k\left(r,t,\gamma,\overline{\mathfrak F}^N\right)=0}-\mathds{1}_{P_k\left(r,t,\gamma,\overline{\mathfrak F}\right)=0}\right)Q_k(dr,d\lambda,d\gamma,du).
\end{equation}

We set
\[H_k(s,t,\gamma,\phi_1,\phi_2)=\int_{s}^{t}\int_{{\gamma(v-s)(\phi_1\wedge\phi_2(v-)}}^{{\gamma(v-s)(\phi_1\vee\phi_2(v-)}}Q_k(dv,dw).\]
Using the fact that for all $A,\,B\in\mathbb{N},$
\[\left|\mathds{1}_{A=0}-\mathds{1}_{B=0}\right|\leq\left|A-B\right|.
\]
From Assumptions~\ref{TCL-AS-lambda-1} and~\ref{TCL-AS-lambda-2} and Lemma~\ref{TCL-lem-barlambda-inc-bound}, it follows that 
\begin{multline*}
\left|\hat{\mathfrak S}^N_{1,1}(t)-\hat{\mathfrak S}^N_{1,1}(s)\right|\\
\begin{aligned}
&\leq+\frac{2}{\sqrt{N}}\sum_{k=1}^{N}\int_{s}^{t}\int_{\bD^2}\int_0^{\lambda_*}H_k(r,t,\gamma,\overline{\mathfrak F}^N,\overline{\mathfrak F})Q_k(dr,d\lambda,d\gamma,du)\\
&\quad+\frac{2(t-s)^\alpha}{\sqrt{N}}\sum_{k=1}^{N}\int_{0}^{s}\int_{\bD^2}\int_0^{\lambda_*}H_k(r,t,\gamma,\overline{\mathfrak F}^N,\overline{\mathfrak F})Q_k(dr,d\lambda,d\gamma,du)\\
&\quad+\frac{2}{\sqrt{N}}\sum_{j=1}^{\ell-1}\sum_{k=1}^{N}\int_{0}^{s}\int_{\bD^2}\int_0^{\lambda_*}\mathds{1}_{s-r<\zeta^j\leq t-r}H_k(r,t,\gamma,\overline{\mathfrak F}^N,\overline{\mathfrak F})Q_k(dr,d\lambda,d\gamma,du)\\
&\quad+\frac{1}{\sqrt{N}}\sum_{k=1}^{N}\int_{0}^{s}\int_{\bD^2}\int_0^{\lambda_*}\left(H_k(r,t,\gamma,\overline{\mathfrak F}^N,\overline{\mathfrak F})+H_k(r,s,\gamma,\overline{\mathfrak F}^N,\overline{\mathfrak F})\right)Q_k(dr,d\lambda,d\gamma,du).
\end{aligned}
\end{multline*}
\begin{lemma}
	As $\delta\to0,$
	\begin{equation} \label{eqn-lm7.9}
	\limsup_{N\to\infty}\sup_{0\leq t\leq T}\frac{1}{\delta}\mathbb{P}\left(\frac{1}{\sqrt{N}}\sum_{k=1}^{N}\int_{t}^{t+\delta}\int_{\bD^2}\int_0^{\lambda_*} H_k(r,t+\delta,\gamma,\overline{\mathfrak F}^N,\overline{\mathfrak F}) Q_k(dr,d\lambda,d\gamma,du)\geq\epsilon\right)\to0.
	\end{equation} 
\end{lemma}
\begin{proof}
	By exchangeability, we have
	\begin{multline}\label{TCL-A-7-30}
	\E\left[\left(\frac{1}{\sqrt{N}}\sum_{k=1}^{N}\int_{t}^{t+\delta}\int_{\bD^2}\int_0^{\lambda_*}H_k(r,t+\delta,\gamma,\overline{\mathfrak F}^N,\overline{\mathfrak F})Q_k(dr,d\lambda,d\gamma,du)\right)^2\right]\\
	\begin{aligned}
	&=\E\left[\left(\int_{t}^{t+\delta}\int_{\bD^2}\int_0^{\lambda_*}H_1(r,t+\delta,\gamma,\overline{\mathfrak F}^N,\overline{\mathfrak F})Q_1(dr,d\lambda,d\gamma,du)\right)^2\right]\\
	&\quad+(N-1)\E\left[\left(\int_{t}^{t+\delta}\int_{\bD^2}\int_0^{\lambda_*}H_1(r,t+\delta,\gamma,\overline{\mathfrak F}^N,\overline{\mathfrak F})Q_1(dr,d\lambda,d\gamma,du)\right)\right.\\
	&\hspace*{3cm}\left.\left(\int_{t}^{t+\delta}\int_{\bD^2}\int_0^{\lambda_*}H_2(r,t+\delta,\gamma,\overline{\mathfrak F}^N,\overline{\mathfrak F})Q_2(dr,d\lambda,d\gamma,du)\right)\right]
	\end{aligned}
	\end{multline}
	Let 
	\begin{equation*}
	\chi^N_k(t)=\int_{t}^{t+\delta}\int_{\bD^2}\int_0^{\lambda_*}H_k(r,t+\delta,\gamma,\overline{\mathfrak F}^N,\overline{\mathfrak F})Q_k(dr,d\lambda,d\gamma,du),
	\end{equation*}
	and 
	\begin{equation*}
	\tilde\chi^N_k(t)=\int_{t}^{t+\delta}\int_{\bD^2}\int_0^{\lambda_*}H_k(r,t+\delta,\gamma,\overline{\mathfrak F}_{(1,2)}^N,\overline{\mathfrak F})Q_k(dr,d\lambda,d\gamma,du).
	\end{equation*}
	Note that, as $\overline{\mathfrak F}_{(1,2)}^N$ and $Q_k$ for $k\in\{1,2\}$ are independent, from Theorem~\ref{TCL-th-1} we have
	\begin{align}\label{TCL-A-7-23}
	\E\left[\tilde\chi^N_k(t)\Big|\overline{\mathfrak F}^N_{(1,2)}\right]
	&=\E\left[\int_{t}^{t+\delta}\int_{\bD^2}\int_0^{\lambda_*}H_k(r,t+\delta,\gamma,\overline{\mathfrak F}_{(1,2)}^N,\overline{\mathfrak F})Q_k(dr,d\lambda,d\gamma,du)\Big|\overline{\mathfrak F}^N_{(1,2)}\right]\nonumber\\
	&=\lambda_*\E\left[\int_{t}^{t+\delta}\int_\bD\int_{r}^{t+\delta}\gamma(v-r)\left|\overline{\mathfrak F}_{(1,2)}^N(v)-\overline{\mathfrak F}(v)\right|dv\mu(d\gamma)dr\Big|\overline{\mathfrak F}^N_{(1,2)}\right]\nonumber\\
	&\leq\lambda_*\int_{t}^{t+\delta}\int_{t}^{t+\delta}\left|\overline{\mathfrak F}_{(1,2)}^N(v)-\overline{\mathfrak F}(v)\right|dvdr\nonumber\\
	&=\lambda_*\delta\int_{t}^{t+\delta}\left|\overline{\mathfrak F}_{(1,2)}^N(v)-\overline{\mathfrak F}(v)\right|dv.
	\end{align}
	Consequently, conditioning on $\overline{\mathfrak F}^N_{(1,2)},$ and using the fact that $Q_1,\,Q_2$ and $\overline{\mathfrak F}^N_{(1,2)}$ are independent, we have 
	\begin{align}\label{TCL-A-7-key}
	\E\left[\tilde\chi^N_1(t)\tilde\chi^N_2(t)\right]&=\E\left[\E\left[\tilde\chi^N_1(t)\tilde\chi^N_2(t)\Big|\overline{\mathfrak F}^N_{(1,2)}\right]\right]\nonumber\\
	&=\E\left[\E\left[\tilde\chi^N_1(t)\Big|\overline{\mathfrak F}^N_{(1,2)}\right]\E\left[\tilde\chi^N_2(t)\Big|\overline{\mathfrak F}^N_{(1,2)}\right]\right]\nonumber\\
	&\leq\lambda_*^2\delta^2\E\left[\left(\int_{t}^{t+\delta}\left|\overline{\mathfrak F}_{(1,2)}^N(v)-\overline{\mathfrak F}(v)\right|dv\right)^2\right]\nonumber\\
	&\leq\lambda_*^2\delta^3\int_{t}^{t+\delta}\E\left[\left|\overline{\mathfrak F}_{(1,2)}^N(v)-\overline{\mathfrak F}(v)\right|^2\right]dv\nonumber\\
	&\leq2\lambda_*^2\delta^4\left(\frac{1}{N^2}\E\left[Y^2\right]+\frac{C_T}{N}\right), 
	\end{align}
	where we use H\"older's inequality and the fact that $|\overline{\mathfrak F}_{(1,2)}^N(v)-\overline{\mathfrak F}(v)|\leq\frac{Y}{N}+|\overline{\mathfrak F}^N(v)-\overline{\mathfrak F}(v)|$, $(a+b)^2\leq2(a^2+b^2)$ and Proposition~\ref{prop-c1}.

	In addition for $k\in\{1,2\}$, using subsection~\ref{TCL-sub-qua}
	\begin{align}
	\left|\chi^N_k(t)-\tilde\chi^N_k(t)\right|&=\int_{t}^{t+\delta}\int_{\bD^2}\int_0^{\lambda_*}H_k(r,t+\delta,\gamma,\overline{\mathfrak F}_{(1,2)}^N,\overline{\mathfrak F}^N)Q_k(dr,d\lambda,d\gamma,du)\nonumber\\
	&\leq\int_{t}^{t+\delta}\int_{\bD^2}\int_{0}^{\lambda_*}\left(\int_{r}^{t+\delta}\int_{\gamma(v-r)(\overline{\mathfrak F}_{(1,2)}^N(v^-)-Y/N)}^{\gamma(v-r)(\overline{\mathfrak F}_{(1,2)}^N(v^-)+Y/N)}Q_k(dv,dw)\right)Q_k(dr,d\lambda,d\gamma,du)\,. 
	\label{TCL-m11}
	\end{align}
	Consequently, using Theorem~\ref{TCL-th-1}
	\begin{equation}\label{TCL-eqr-5}
	\E\left[\left|\chi^N_1(t)-\tilde\chi^N_1(t)\right|\Big|\overline{\mathfrak F}^N_{(1,2)}\right]\leq\frac{2Y\lambda_*\delta}{N}, 
	\end{equation}
	and from \eqref{TCL-m11} we deduce that as $N\to\infty,\,\chi^N_1(t)\to0$ in probability and as 
	\[\chi^N_1(t)\leq\left(\int_t^{t+\delta}\int_0^{\lambda_*}Q_1(dv,dw)\right)^2,\] it follows that,  as $N\to\infty$,
	\begin{equation}\label{TCL-A-7-29}
	\E\left[\left(\chi^N_1(t)\right)^2\right]\to0.
	\end{equation}
	We have 
	\begin{align}\label{TCL-A-7-27}
	\E\left[\chi^N_1(t)\chi^N_2(t)-\tilde\chi^N_1(t)\tilde\chi^N_2(t)\right]&=\E\left[\tilde\chi^N_1(t)\left(\chi^N_2(t)-\tilde\chi^N_2(t)\right)\right]+\E\left[\left(\chi^N_1(t)-\tilde\chi^N_1(t)\right)\tilde\chi^N_2(t)\right]\nonumber\\&\hspace*{3.5cm}+\E\left[\left(\chi^N_1(t)-\tilde\chi^N_1(t)\right)\left(\chi^N_2(t)-\tilde\chi^N_2(t)\right)\right]. 
	\end{align}
	As $\big(Q_1,\,Q_2\big)$ and $\big(\overline{\mathfrak F}^N_{(1,2)},Y\big)$ are independent,  from Theorem~\ref{TCL-th-1}, \eqref{TCL-m11}, \eqref{TCL-eqr-5} and \eqref{TCL-A-7-23} it follows that 
	\begin{multline}\label{TCL-A-7-24}
	\left|\E\left[\tilde\chi^N_1(t)\left(\chi^N_2(t)-\tilde\chi^N_2(t)\right)\right]\right|\\
	\begin{aligned}
	&\leq\E\left[\E\left[\tilde\chi^N_1(t)\int_{t}^{t+\delta}\int_{\bD^2}\int_{0}^{\lambda_*}\left(\int_{r}^{t+\delta}\int_{\gamma(v-r)(\overline{\mathfrak F}_{(1,2)}^N(v^-)-Y/N)}^{\gamma(v-r)(\overline{\mathfrak F}_{(1,2)}^N(v^-)+Y/N)}Q_2(dv,dw)\right)Q_2(dr,d\lambda,d\gamma,du)\Big|\overline{\mathfrak F}_{(1,2)}^N,Y\right]\right]\\
	&=\E\left[\E\left[\tilde\chi^N_1(t)\Big|\overline{\mathfrak F}_{(1,2)}^N,Y\right]\right.\\
	&\hspace{2cm}\left.\times\E\left[\int_{t}^{t+\delta}\int_{\bD^2}\int_{0}^{\lambda_*}\left(\int_{r}^{t+\delta}\int_{\gamma(v-r)(\overline{\mathfrak F}_{(1,2)}^N(v^-)-Y/N)}^{\gamma(v-r)(\overline{\mathfrak F}_{(1,2)}^N(v^-)+Y/N)}Q_2(dv,dw)\right)Q_2(dr,d\lambda,d\gamma,du)\Big|\overline{\mathfrak F}_{(1,2)}^N,Y\right]\right]\\
	&\leq\frac{2\lambda_*^2\delta^2}{N}\E\left[Y\int_{t}^{t+\delta}\left|\overline{\mathfrak F}_{(1,2)}^N(v)-\overline{\mathfrak F}(v)\right|dv\right]\\
	&\le\frac{2\lambda_*^3\delta^3}{N}\E\left[Y\right],
	\end{aligned}
	\end{multline}
	where the last line follows from the fact that $\left|\overline{\mathfrak F}_{(1,2)}^N(v)-\overline{\mathfrak F}(v)\right|\leq\lambda_*.$
	
	Similarly, we show that
	\begin{equation}\label{TCL-A-7-26}
	\left|\E\left[\left(\chi^N_1(t)-\tilde\chi^N_1(t)\right)\tilde\chi^N_2(t)\right]\right|\le\frac{2\lambda_*^3\delta^3}{N}\E\left[Y\right].
	\end{equation}
	From \eqref{TCL-m11}, and as in the setting of \eqref{TCL-A-7-24}, it follows that
	\begin{equation}\label{TCL-A-7-25}
	\E\left[\left|\chi^N_1(t)-\tilde\chi^N_1(t)\right|\left|\chi^N_2(t)-\tilde\chi^N_2(t)\right|\right]\leq\frac{4\lambda_*^4\delta^4}{N^2}\E\left[Y^2\right].\\
	\end{equation}
	Consequently from \eqref{TCL-A-7-25},\eqref{TCL-A-7-26}, \eqref{TCL-A-7-24},  \eqref{TCL-A-7-27}, and \eqref{TCL-A-7-key}, it follows that 
	\begin{equation*}
	\lim_{\delta\to0}\limsup_{N\to\infty}\sup_{0\leq t\leq T}\frac{1}{\delta}N\E\left[\chi^N_1(t)\chi^N_2(t)\right]=0.
	\end{equation*}
	
	Therefore, from \eqref{TCL-A-7-29} and \eqref{TCL-A-7-30}, it follows that \eqref{eqn-lm7.9} holds.
\end{proof}
Applying the same method as in the previous lemma, we establish the following Lemma, 
\begin{lemma}
	As $\delta\to0$,
	\begin{equation*}
	\limsup_{N\to\infty}\sup_{0\leq t\leq T}\frac{1}{\delta}\mathbb{P}\left(\frac{\delta^\alpha}{\sqrt{N}}\sum_{k=1}^{N}\int_{0}^{t}\int_{\bD^2}\int_0^{\lambda_*}H_k(r,t+\delta,\gamma,\overline{\mathfrak F}^N,\overline{\mathfrak F}) Q_k(dr,d\lambda,d\gamma,du)\geq\epsilon\right)\to0,
	\end{equation*}
	\begin{equation*}
	\limsup_{N\to\infty}\sup_{0\leq t\leq T}\frac{1}{\delta}\mathbb{P}\left(\frac{1}{\sqrt{N}}\sum_{k=1}^{N}\int_{0}^{t}\int_{\bD^2}\int_0^{\lambda_*} H_k(r,t,\gamma,\overline{\mathfrak F}^N,\overline{\mathfrak F})Q_k(dr,d\lambda,d\gamma,du)\geq\epsilon\right)\to0,
	\end{equation*}
	and
	\begin{equation*}
	\limsup_{N\to\infty}\sup_{0\leq t\leq T}\frac{1}{\delta}\mathbb{P}\left(\frac{1}{\sqrt{N}}\sum_{k=1}^{N}\int_{0}^{t}\int_{\bD^2}\int_0^{\lambda_*} H_k(r,t+\delta,\gamma,\overline{\mathfrak F}^N,\overline{\mathfrak F})Q_k(dr,d\lambda,d\gamma,du)\geq\epsilon\right)\to0,
	\end{equation*}
\end{lemma}
To conclude the tightness of $\hat{\mathfrak S}^N_{1,1}$, it remains to establish the following Lemma.

\begin{lemma}
	As $\delta\to0,$
	\begin{align} \label{eqn-lem7.12} 
	&\limsup_{N\to\infty}\sup_{0\leq t\leq T}\frac{1}{\delta}\mathbb{P}\Bigg(\frac{1}{\sqrt{N}}\sum_{j=1}^{\ell-1}\sum_{k=1}^{N}\int_{0}^{t}\int_{\bD^2}\int_0^{\lambda_*}\mathds{1}_{t-r<\zeta^j\leq t+\delta-r} H_k(r,t+\delta,\gamma,\overline{\mathfrak F}^N,\overline{\mathfrak F})Q_k(dr,d\lambda,d\gamma,du)\geq\epsilon\Bigg)\nonumber\\&\hspace*{12cm}\to0.
	\end{align} 	
\end{lemma}
\begin{proof}
	By exchangeability it follows that 
	\begin{multline}\label{TCL-A-7-30-1}
	\E\left[\left(\frac{1}{\sqrt{N}}\sum_{j=1}^{\ell-1}\sum_{k=1}^{N}\int_{0}^{t}\int_{\bD^2}\int_0^{\lambda_*}\mathds{1}_{t-r<\zeta^j\leq t+\delta-r}H_k(r,t+\delta,\gamma,\overline{\mathfrak F}^N,\overline{\mathfrak F})Q_k(dr,d\lambda,d\gamma,du)\right)^2\right]\\
	\begin{aligned}
	&=\E\left[\left(\sum_{j=1}^{\ell-1}\int_{0}^{t}\int_{\bD^2}\int_0^{\lambda_*}\mathds{1}_{t-r<\zeta^j\leq t+\delta-r}H_1(r,t+\delta,\gamma,\overline{\mathfrak F}^N,\overline{\mathfrak F})Q_1(dr,d\lambda,d\gamma,du)\right)^2\right]\\
	&\quad+(N-1)\E\left[\left(\sum_{j=1}^{\ell-1}\int_{0}^{t}\int_{\bD^2}\int_0^{\lambda_*}\mathds{1}_{t-r<\zeta^j\leq t+\delta-r}H_1(r,t+\delta,\gamma,\overline{\mathfrak F}^N,\overline{\mathfrak F})Q_1(dr,d\lambda,d\gamma,du)\right)\right.\\
	&\hspace{2cm}\left.\left(\sum_{j=1}^{\ell-1}\int_{0}^{t}\int_{\bD^2}\int_0^{\lambda_*}\mathds{1}_{t-r<\zeta^j\leq t+\delta-r}H_2(r,t+\delta,\gamma,\overline{\mathfrak F}^N,\overline{\mathfrak F})Q_2(dr,d\lambda,d\gamma,du)\right)\right]\,. 
	\end{aligned}
	\end{multline}
	Let \[\chi^N_k(t)=\sum_{j=1}^{\ell-1}\int_{0}^{t}\int_{\bD^2}\int_0^{\lambda_*}\mathds{1}_{t-r<\zeta^j\leq t+\delta-r}H_k(r,t+\delta,\gamma,\overline{\mathfrak F}^N,\overline{\mathfrak F})Q_k(dr,d\lambda,d\gamma,du)\,,
	\]
	and 
	\[\tilde\chi^N_k(t)=\sum_{j=1}^{\ell-1}\int_{0}^{t}\int_{\bD^2}\int_0^{\lambda_*}\mathds{1}_{t-r<\zeta^j\leq t+\delta-r}H_k(r,t+\delta,\gamma,\overline{\mathfrak F}^N_{(1,2)},\overline{\mathfrak F})Q_k(dr,d\lambda,d\gamma,du).\]
	Since $\overline{\mathfrak F}_{(1,2)}^N$ and $Q_k$ for $k\in\{1,2\}$ are independent, from Theorem~\ref{TCL-th-1},  we have
	\begin{multline}\label{TCL-A-7-23-1}
	\E\left[\tilde\chi^N_k(t)\Big|\overline{\mathfrak F}^N_{(1,2)}\right]\\
	\begin{aligned}
	&=\sum_{j=1}^{\ell-1}\E\left[\int_{0}^{t}\int_{\bD^2}\int_0^{\lambda_*}\mathds{1}_{t-r<\zeta^j\leq t+\delta-r}H_k(r,t+\delta,\gamma,\overline{\mathfrak F}^N_{(1,2)},\overline{\mathfrak F})Q_k(dr,d\lambda,d\gamma,du)\Big|\overline{\mathfrak F}^N_{(1,2)}\right]\\
	&=\lambda_*\sum_{j=1}^{\ell-1}\E\left[\int_{0}^{t}\int_{\bD}\mathds{1}_{t-r<\zeta^j\leq t+\delta-r}\int_{r}^{t+\delta}\gamma(v-r)\left|\overline{\mathfrak F}_{(1,2)}^N(v)-\overline{\mathfrak F}(v)\right|dv\mu(d\gamma)dr\Big|\overline{\mathfrak F}^N_{(1,2)}\right]\\
	&\leq\lambda_*\sum_{j=1}^{\ell-1}\E\left[\int_{0}^{t}\left(G_j(t+\delta-r)-G_j(t-r)\right)\int_{r}^{t+\delta}\left|\overline{\mathfrak F}_{(1,2)}^N(v)-\overline{\mathfrak F}(v)\right|dvdr\Big|\overline{\mathfrak F}^N_{(1,2)}\right]\\
	&\leq\lambda_*\ell T\delta^{\rho}\int_{0}^{T}\left|\overline{\mathfrak F}_{(1,2)}^N(v)-\overline{\mathfrak F}(v)\right|dv,
	\end{aligned}
	\end{multline}
	where we use $\gamma\leq1$, to get the third line.
	
	Consequently, conditioning by $\overline{\mathfrak F}^N_{(1,2)},$ and using the fact that $Q_1,\,Q_2$ and $\overline{\mathfrak F}^N_{(1,2)}$ are independent, we have 
	\begin{align}\label{TCL-A-7-key-1}
	\E\left[\tilde\chi^N_1(t)\tilde\chi^N_2(t)\right]&=\E\left[\E\left[\tilde\chi^N_1(t)\tilde\chi^N_2(t)\Big|\overline{\mathfrak F}^N_{(1,2)}\right]\right]\nonumber\\
	&=\E\left[\E\left[\tilde\chi^N_1(t)\Big|\overline{\mathfrak F}^N_{(1,2)}\right]\E\left[\tilde\chi^N_2(t)\Big|\overline{\mathfrak F}^N_{(1,2)}\right]\right]\nonumber\\
	&\leq\lambda_*^2\ell^2 T^2\delta^{2\rho}\E\left[\left(\int_{0}^{T}\left|\overline{\mathfrak F}_{(1,2)}^N(v)-\overline{\mathfrak F}(v)\right|dv\right)^2\right]\nonumber\\
	&\leq\lambda_*^2\ell^2 T^3\delta^{2\rho}\int_{0}^{T}\E\left[\left|\overline{\mathfrak F}_{(1,2)}^N(v)-\overline{\mathfrak F}(v)\right|^2\right]dv\nonumber\\
	&\leq2\lambda_*^2\ell^2 T^3\delta^{2\rho}\left(\frac{1}{N^2}\E\left[Y^2\right]+\frac{C_T}{N}\right), 
	\end{align}
	where we use H\"older's inequality and the fact that $|\overline{\mathfrak F}_{(1,2)}^N(v)-\overline{\mathfrak F}(v)|\leq\frac{Y}{N}+|\overline{\mathfrak F}^N(v)-\overline{\mathfrak F}(v)|$, $(a+b)^2\leq2(a^2+b^2)$ and Proposition~\ref{prop-c1}.
	
	In addition for $k\in\{1,2\}$, using subsection~\ref{TCL-sub-qua}
	\begin{align}
	& \left|\chi^N_k(t)-\tilde\chi^N_k(t)\right|\nonumber\\
	&=\sum_{j=1}^{\ell-1}\int_{0}^{t}\int_{\bD^2}\int_0^{\lambda_*}\mathds{1}_{t-r<\zeta^j\leq t+\delta-r} H_k(r,t+\delta,\gamma,\overline{\mathfrak F}^N_{(1,2)},\overline{\mathfrak F}) Q_k(dr,d\lambda,d\gamma,du)\nonumber\\
	&\leq\sum_{j=1}^{\ell-1}\int_{0}^{t}\int_{\bD^2}\int_{0}^{\lambda_*}\mathds{1}_{t-r<\zeta^j\leq t+\delta-r}\left(\int_{r}^{t+\delta}\int_{\gamma(v-r)(\overline{\mathfrak F}_{(1,2)}^N(v^-)-Y/N)}^{\gamma(v-r)(\overline{\mathfrak F}_{(1,2)}^N(v^-)+Y/N)}Q_k(dv,dw)\right)Q_k(dr,d\lambda,d\gamma,du)\nonumber\\
	&=:\hat\chi^N_k(t).\label{TCL-m11-1}
	\end{align}
	Consequently, using Theorem~\ref{TCL-th-1}, it follows that,
	\begin{equation}
	\E\left[\left|\chi^N_1(t)-\tilde\chi^N_1(t)\right|\right]\leq\frac{2T^2\lambda_*\ell\delta^{\rho}}{N}\E\left[Y\right].
	\end{equation}
	and from \eqref{TCL-m11-1} we deduce that as $N\to\infty,\,\chi^N_1\to0$ in probability, and as it is bounded by a square-integrable process, it follows that,  as $N\to\infty$,
	\begin{equation}\label{TCL-A-7-29-1}
	\E\left[\left(\chi^N_1(t)\right)^2\right]\to0.
	\end{equation}
	We have 
	\begin{align}\label{TCL-A-7-27-1}
	\E\left[\chi^N_1(t)\chi^N_2(t)-\tilde\chi^N_1(t)\tilde\chi^N_2(t)\right]&=\E\left[\tilde\chi^N_1(t)\left(\chi^N_2(t)-\tilde\chi^N_2(t)\right)\right]+\E\left[\left(\chi^N_1(t)-\tilde\chi^N_1(t)\right)\tilde\chi^N_2(t)\right]\nonumber\\&\hspace*{1.5cm}+\E\left[\left(\chi^N_1(t)-\tilde\chi^N_1(t)\right)\left(\chi^N_2(t)-\tilde\chi^N_2(t)\right)\right],
	\end{align}
	As $(Q_1,\,Q_2)$ and $(\overline{\mathfrak F}^N_{(1,2)},Y)$ are independent from \eqref{TCL-m11-1} and \eqref{TCL-A-7-23-1},
	\begin{align}\label{TCL-A-7-24-1}
	\left|\E\left[\tilde\chi^N_1(t)\left(\chi^N_2(t)-\tilde\chi^N_2(t)\right)\right]\right|&\leq\E\left[\E\left[\tilde\chi^N_1(t)\hat{\chi}^N_k(t)\Big|\overline{\mathfrak F}_{(1,2)}^N,Y\right]\right]\nonumber\\
	&=\E\left[\E\left[\tilde\chi^N_1(t)\Big|\overline{\mathfrak F}_{(1,2)}^N,Y\right]\E\left[\hat\chi^N_1(t)\Big|\overline{\mathfrak F}_{(1,2)}^N,Y\right]\right]\nonumber\\
	&\leq\frac{2T^3\lambda_*^2\ell^2\delta^{2\rho}}{N}\E\left[Y\int_{0}^{T}\left|\overline{\mathfrak F}_{(1,2)}^N(v)-\overline{\mathfrak F}(v)\right|dv\right]\nonumber\\
	&\le\frac{2T^3\lambda_*^2\ell^2\delta^{2\rho}}{N^{3/2}}\E\left[Y\right].
	\end{align}
	Similarly we show that
	\begin{equation}\label{TCL-A-7-26-1}
	\left|\E\left[\left(\chi^N_1(t)-\tilde\chi^N_1(t)\right)\tilde\chi^N_2(t)\right]\right|\le\frac{2T^3\lambda_*^2\ell^2\delta^{2\rho}}{N^{3/2}}\E\left[Y\right].
	\end{equation}
	From \eqref{TCL-m11-1}, and as in the setting in \eqref{TCL-A-7-24-1}, it follows that 
	\begin{equation}\label{TCL-A-7-25-1}
	\E\left[\left|\chi^N_1(t)-\tilde\chi^N_1(t)\right|\left|\chi^N_2(t)-\tilde\chi^N_2(t)\right|\right]\leq\frac{4T^4\lambda_*^4\ell^4\delta^{4\rho}}{N^2}\E\left[Y^2\right].\\
	\end{equation}
	Consequently from \eqref{TCL-A-7-25-1},\eqref{TCL-A-7-26-1}, \eqref{TCL-A-7-24-1},  \eqref{TCL-A-7-27-1}, and \eqref{TCL-A-7-key-1}, it follows that 
	\begin{equation*}
	\lim_{\delta\to0}\limsup_{N\to\infty}\sup_{0\leq t\leq T}\frac{1}{\delta}N\E\left[\chi^N_1(t)\chi^N_2(t)\right]=0.
	\end{equation*}
	
	Therefore, \eqref{eqn-lem7.12} follows from \eqref{TCL-A-7-29-1} and \eqref{TCL-A-7-30-1}.
\end{proof}
\subsubsection{Proof of Lemma~\ref{TCL-tight-F4}}\label{TCL-sub-sec-10}
We recall that 
\begin{equation*}
\hat{\mathfrak S}_{1,0}^N(t)=\frac{1}{\sqrt{N}}\sum_{k=1}^{N}\gamma_{k,0}(t)\left(\mathds{1}_{P_{k}(0,t,\gamma_{k,0},\overline{\mathfrak F}^N)=0}-\mathds{1}_{P_{k}(0,t,\gamma_{k,0},\overline{\mathfrak F})=0}\right).
\end{equation*}
Using the fact that for all $A,\,B\in\mathbb{N},$
\[\left|\mathds{1}_{A=0}-\mathds{1}_{B=0}\right|\leq\left|A-B\right|,\]
from Assumption~\ref{TCL-AS-lambda-1} and~\ref{TCL-AS-lambda-2} and Lemma~\ref{TCL-lem-barlambda-inc-bound}, it follows that 
\begin{multline}
\left|\hat{\mathfrak S}_{1,0}^N(t)-\hat{\mathfrak S}_{1,0}^N(s)\right|\leq\frac{(t-s)^\alpha}{\sqrt{N}}\sum_{k=1}^{N}H_k(0,t,\gamma_{k,0},\overline{\mathfrak F}^N,\overline{\mathfrak F})\\
\begin{aligned}
&+\frac{1}{\sqrt{N}}\sum_{k=1}^{N}\sum_{j=1}^{\ell-1}\mathds{1}_{s-r<\zeta^j_k\leq t-r}H_k(0,t,\gamma_{k,0},\overline{\mathfrak F}^N,\overline{\mathfrak F})\\
&+\frac{1}{\sqrt{N}}\sum_{k=1}^{N}\left(H_k(0,t,\gamma_{k,0},\overline{\mathfrak F}^N,\overline{\mathfrak F})+H_k(0,s,\gamma_{k,0},\overline{\mathfrak F}^N,\overline{\mathfrak F})\right)\,.
\end{aligned}
\end{multline}
Hence as in the setting of \eqref{eqc5}, \eqref{TCL-in-eqc6}  and \eqref{TCL-eqc7'} respectively, we establish the tightness of $\hat{\mathfrak S}_{1,0}^N$.  

\bibliographystyle{plain}
\bibliography{Epidemic-Age,VIVS_raphael}
\end{document}